\documentclass[11pt]{amsart}
\usepackage{amsmath,amssymb,amsfonts}
\usepackage{enumerate}
\usepackage{enumitem}
\usepackage{graphics} %% add this and next lines if pictures should be in esp format
\usepackage{float}
\usepackage{epsfig} %For pictures: screened artwork should be set up with an 85 or 100 line screen
\usepackage{graphicx}
\usepackage{epstopdf}%This is to transfer .eps figure to .pdf figure; please compile your paper using PDFLeTex or PDFTeXify.
\usepackage{hyperref}
\usepackage{framed}
\usepackage{verbatim}
\usepackage[numbers,sort]{natbib}

\usepackage{bm}
\usepackage{color}
\usepackage[a4paper]{geometry}
\usepackage{bbm}

\newtheorem{theorem}{Theorem}[section]
\newtheorem{corollary}[theorem]{Corollary}

\newtheorem{lemma}[theorem]{Lemma}
\newtheorem{proposition}[theorem]{Proposition}

\theoremstyle{definition}
\newtheorem{definition}[theorem]{Definition}
\newtheorem{remark}[theorem]{Remark}
\newtheorem{example}[theorem]{Example}

\newcommand{\scalarp}[1]{\left\langle #1 \right\rangle}

\newcommand{\Lip}{\textup{Lip}}

\newcommand{\N}{\mathbb{N}}

\newcommand{\R}{\mathbb{R}}

\newcommand{\M}{\mathcal{M}}
\newcommand{\W}{\mathcal{W}}

\newcommand{\diver}{\textup{div}}
\newcommand{\Law}{\textup{Law}}
\newcommand{\conv}{*}
\newcommand{\defin}{:=}
\newcommand{\state}{Y}
\newcommand{\appstate}[2]{A_{#1,#2}}
\newcommand{\mylabel}[2]{#2\def\@currentlabel{#2}\label{#1}}
\newcommand{\zweivect}[3][]{{\left(\!\!#1\begin{array}{c}#2\\#3\end{array}#1\!\!\right)}}
\newcommand{\weak}{\rightharpoonup}

\DeclareMathOperator{\supp}{supp}

\newcommand{\Mu}{{\mathcal{M}(\R^d)}}

\newcommand{\gwbase}{\mathcal{W}_g}
\newcommand{\gw}[1]{{\gwbase\left(#1\right)}}
\newcommand{\Pt}[1]{\left({#1}\right)}

\allowdisplaybreaks

\title{Leader formation with mean-field birth and death models}

\author{Giacomo Albi}
\address{Department of Computer Science, University of Verona, Str. Le Grazie 15, Verona, IT-37134,
Italy.}
\email{{giacomo.albi@univr.it}}

\author{Mattia Bongini}
\address{Big Data and Marketing Analytics, CRIF, Via M. Fantin 3, Bologna, IT-40131, Italy}
\email{{mattia.bongini@crif.com}}

\author{Francesco Rossi}
\address{Department of Mathematics ``Tullio Levi-Civita'', University of Padova,Via Trieste, 63, Padova, IT-35121,
Italy.}
\email{{francesco.rossi@math.unipd.it}}

\author{Francesco Solombrino}
\address{Department of Mathematics and Applications ``R. Caccioppoli'', University of Naples ``Federico II''
Via Cintia, Monte S. Angelo, Naples, IT-80126, Italy.}
\email{{francesco.solombrino@unina.it}}

\date{}

\begin{document}

\markboth{G. Albi, M. Bongini, F. Rossi, F. Solombrino}{Leader formation with mean-field birth and death models}

%%%%%%%%%%%%%%%%%%% Publisher's Area please ignore %%%%%%%%%%%%%%%%%%%%%%%
%
%\catchline{}{}{}{}{}
%
%%%%%%%%%%%%%%%%%%%%%%%%%%%%%%%%%%%%%%%%%%%%%%%%%%%%%%%%%%%%%%%%%%%%%%%%%%

%\title{Leader formation with mean-field birth and death models}

%\author{Giacomo Albi\footnote{giacomo.albi@univr.it}}
%\address{Department of Computer Science, University of Verona,\\ Str. Le Grazie 15, Verona, IT-37134,
%Italy.}
%
%\author{Mattia Bongini\footnote{mattia.bongini@crif.com}}
%\address{Big Data and Marketing Analytics, CRIF, \\ Via M. Fantin 3, Bologna, IT-40131, Italy}
%
%\author{Francesco Rossi\footnote{francesco.rossi@math.unipd.it}}
%\address{Department of Mathematics ``Tullio Levi-Civita'', University of Padova,\\ Via Trieste, 63, Padova, IT-35121,
%Italy.}
%
%\author{Francesco Solombrino\footnote{francesco.solombrino@unina.it}}
%\address{Department of Mathematics and Applications ``R. Caccioppoli'', University of Naples ``Federico II''\\
%Via Cintia, Monte S. Angelo - IT-80126 Naples, Italy.}
%

%\author{Author\footnote{Typeset names in 8 pt Roman. Use the footnote to indicate the
%present or permanent address of the author.}}
%
%\address{University Department, University Name, Address\\
%City, State ZIP/Zone,
%Country\footnote{State completely without abbreviations, the
%affiliation and mailing address, including country. Typeset in 8 pt
%Times italic.}\\
%first\_author@university.edu}

\maketitle

\begin{abstract}
	We provide a mean-field description for a leader-follower dynamics with mass transfer among the two populations. This model allows the transition from followers to leaders and vice versa, with scalar-valued transition rates depending nonlinearly on the global state of the system at each time.
	
\noindent We first prove the existence and uniqueness of solutions for the leader-follower dynamics, under suitable assumptions. We then establish, for an appropriate choice of the initial datum, the equivalence of the system with a PDE-ODE system, that consists of a continuity equation over the state space and an ODE for the transition from leader to follower or vice versa. We further introduce a stochastic process approximating the PDE, together with a jump process that models the switch between the two populations. Using a propagation of chaos argument, we show that the particle system generated by these two processes converges in probability to a solution of the PDE-ODE system. Finally, several numerical simulations of social interactions dynamics modeled by our system are discussed.
\end{abstract}

%\tableofcontents

\section{Introduction}\label{sec:intro}

The mathematical modeling of collective behavior for systems of interacting agents has spawned an enormous wealth of literature in recent years. From the study of biological, social and economical phenomena \cite{parisi08,albi2015invisible,cordier2005kinetic,ahn13} to automatic learning \cite{maggioni17,huang2018learning} and optimization heuristics \cite{dorigo2005ant,kennedy2011particle}, these models lay at the heart of some of today's most prominent lines of research: for the latest development in the field, we point to the surveys \cite{bak2013nature,bonginichapsparse,carrillo2017review,viza12} and references therein.

The modeling of such phenomena typically starts from particle-like systems as in statistical physics. These particle models are also called Agent Based Models, and they usually consist of a set of ODEs (one for each agent) interwined in a nonlinear way. Such a modeling approach is quite useful, with one of the main advantages being the explicit description of the mutual interaction among agents, but has huge problems to treat large systems of particles, as is the case with cells, molecules and social networks’ users. A classical approach to attack the problem is to pass to a continuous description of the system, which means to pass from a particle description to a kinetic descriptions where the unknown is the particle density distribution in the state space.

A useful tool in solving this problem is the \textit{mean-field} limit \cite{golseparticle},  which amounts to replace the influence of all the other individuals in the dynamics of any given agent by a single averaged effect, a technique that goes back to \cite{dobrushin79} in plasma physics: to exemplify it, if applied to a Hegselmann-Krause-type discrete particle system over $\R^d$ (see \cite{krause02})

$$\dot{x}_i(t) = \frac1N\sum^N_{j = 1} K(x_i(t) - x_j(t)), \qquad i = 1,\ldots,N \text{ and } t \in [0,T]$$

\noindent where $K$ denotes the \textit{interaction kernel} (which models the interaction between particles) it leads to a continuity equation of Vlasov type

$$\partial_t \mu_t = -\diver((K*\mu_t)\mu_t), \qquad  t \in [0,T]$$

\noindent with $\mu_t$ denoting the probability distribution of the particles over the state space $\R^d$ and
$$(K*\mu_t)(x) = \int_{\R^d} K(x-y)d\mu_t(y), \quad \text{ for every } x \in \R^d.$$
Notice how, in the process, the information of the pointwise positions $x_j(t)$ is replaced by the knowledge of the space distribution of the particles $\mu_t$. Such approach has the advantage of reducing the computational complexity of the models (overcoming the curse of dimensionality \cite{bellmandynamic}) and allows the so-called \textit{microfundation of macromodels}, i.e., the validation of the macroscopic dynamics from the coherence with the behavior of individuals (a central issue in the field of macroeconomics). The mean-field limit of systems of interacting agents has been thoroughly studied also in conjunction with irregular interaction kernel \cite{carrillo2014derivation,garroni2018convergence}, control problems \cite{lasry2007mean,MFOC,bongini2016pontryagin,fornasier2014mean,albi2017mean} and multiple populations \cite{cirant2015multi,bongini2016optimal,albi2014boltzmann,albi2017opinion}. Also models where the total mass of the system is not preserved in time, due to the presence of source (or sink) terms, have been considered (see for instance \cite[Sections 4-5]{Rossi-Review}). In other models, the total mass of the system is preserved, but not the role of the agents, since exchanges of mass between different populations are allowed. One of these models is the leader-follower dynamics studied in \cite{markowich,carrillo2018splitting}, given by

\begin{align}\label{eq:macroleadfollstrong}
\left\{\begin{aligned}
\partial_t \mu^F_t& = -\diver\big((K^{F}\conv\mu^F_t + K^{L}\conv\mu^L_t)\mu^F_t\big) \\
&\quad \quad \quad \quad \quad \quad \quad \quad \quad \quad - \alpha_F(\mu^F_t,\mu^L_t)\mu^F_t + \alpha_L(\mu^F_t,\mu^L_t)\mu^L_t,\\
\partial_t \mu^L_t& = -\diver\big((K^{F}\conv\mu^F_t + K^{L}\conv\mu^L_t)\mu^L_t\big) \\
&\quad \quad \quad \quad \quad \quad \quad \quad \quad \quad +\alpha_F(\mu^F_t,\mu^L_t)\mu^F_t -\alpha_L(\mu^F_t,\mu^L_t)\mu^L_t.
\end{aligned}\quad t\in[0,T].\right.
\end{align}
%\begin{align}\label{eq:macroleadfollstrong}
%\left\{\begin{aligned}
%\partial_t \mu^F_t& = -\diver\big((K^{F}\conv\mu^F_t + K^{L}\conv\mu^L_t)\mu^F_t\big)  - \alpha_F(\mu^F_t,\mu^L_t)\mu^F_t + \alpha_L(\mu^F_t,\mu^L_t)\mu^L_t,\\
%\partial_t \mu^L_t& = -\diver\big((K^{F}\conv\mu^F_t + K^{L}\conv\mu^L_t)\mu^L_t\big) +\alpha_F(\mu^F_t,\mu^L_t)\mu^F_t -\alpha_L(\mu^F_t,\mu^L_t)\mu^L_t.
%\end{aligned}\quad t\in[0,T].\right.
%\end{align}
Here, two competing populations $\mu^F_t$ and $\mu^L_t$, of followers and leaders respectively, are in interaction. Both the masses of followers and leaders vary in time, while their sum is constant. The functionals $K^{i}:\R^d\rightarrow\R^d$ for $i\in\{F,L\}$ are interaction kernels, modeling their mutual spatial influence, while the transition rates $\alpha_F,\alpha_L:\mathcal{M}(\R^d)\times\mathcal{M}(\R^d)\rightarrow[0,+\infty)$ govern the exchange of mass between $\mu^F_t$ and $\mu^L_t$. In this paper, we shall provide a thorough mean-field analysis of \eqref{eq:macroleadfollstrong}, discussing its well-posedness and rigorously deriving it from an Agent Based Model. In order to do this, we will restrict our attention to the case where the transition rates $\alpha_i$ are scalar-valued, that is they depend on the global state of the system at each time $t$, but not explicitly on the position $x$. The usefulness of such a simplification in our analysis is discussed later on in Remark \ref{rem:simplify}.

In order to carry out our analysis, we shall first establish the well-posedness of  \eqref{eq:macroleadfollstrong} by means of a compactness argument in the space of finite positive measures with compact support endowed with the \textit{generalized Wasserstein distance} $\mathcal{W}_g$, see \cite{gw}. To do so, we shall introduce an explicit Euler approximation of the dynamics and show that it converges, as the time step vanishes, to the unique solution of system \eqref{eq:macroleadfollstrong}.

We shall then prove the equivalence between \eqref{eq:macroleadfollstrong} and another system for which we can more easily provide a particle dynamics.
Intuitively, this equivalent system is introduced by defining the measures $(\nu_t,\sigma_t)$ as\footnote{With a slight abuse of notation, from now on we write $\sigma(F),\sigma(L)$ instead of $\sigma(\{F\}),\sigma(\{L\})$.}
$$\nu_t \defin \mu^{F}_t + \mu^{L}_t,\qquad (\sigma_t(F),\sigma_t(L)):=(\mu^F_t(\R^d)/\nu_t(\R^d),\mu^L_t(\R^d)/\nu_t(\R^d)).$$
The idea of the equivalence is that, under suitable hypotheses on the initial data, one can recover $\mu^F_t,\mu^L_t$ from $\nu_t,\sigma_t$ by the relations
\begin{equation}\label{eq: munusigma}
\mu^F_t=\sigma_t(F)\nu_t,\qquad \mu^L_t=\sigma_t(L)\nu_t.
\end{equation}
We shall show that, if the initial datum  $\mu^F_0$, $\mu^L_0$ satisfies \eqref{eq: munusigma} for $t=0$, the system \eqref{eq:macroleadfollstrong} is equivalent to
\begin{align}\label{eq:macronusigma}
\begin{cases}
\partial_t\nu_t = -\diver\Big(\scalarp{K,\nu_t\times\sigma_t}\nu_t\Big),\\
\partial_t \sigma_t = \appstate{\nu_t}{\sigma_t}\sigma_t,
\end{cases}\quad t\in[0,T],
\end{align}
where $\nu_t\times \sigma_t$ is the product measure, the vector field for $\nu_t$ is
\begin{equation}
\scalarp{K,\nu_t\times\sigma_t} \defin \sigma_t(F) K^F\conv \nu_t + \sigma_t(L) K^L\conv \nu_t,
\label{e-vf2}
\end{equation}
and the birth-death transition matrix is
\begin{eqnarray}
\appstate{\nu_t}{\sigma_t} \defin \left[\begin{matrix}
-\alpha_F(\sigma_t(F)\nu_t,\sigma_t(L)\nu_t) & \alpha_L(\sigma_t(F)\nu_t,\sigma_t(L)\nu_t)\\
\alpha_F(\sigma_t(F)\nu_t,\sigma_t(L)\nu_t) & -\alpha_L(\sigma_t(F)\nu_t,\sigma_t(L)\nu_t)
\end{matrix}\right].
\label{e-A}
\end{eqnarray}

The advantage of the measures $\nu_t$ and $\sigma_t$, with respect to $\mu^F_t$ and $\mu^L_t$, is that they are probability measuers over $\R^d$ and $\{F,L\}$, respectively, and we can therefore use a propagation of chaos argument (see \cite{snitzman}) to show that there exists a sequence of stochastic processes whose mean-field limit for $N\rightarrow\infty$ is system \eqref{eq:macronusigma}. We will actually provide such processes in explicit form: we denote them as $(X^{1,N}_t,\state^{1,N}_t),\ldots,(X^{N,N}_t,\state^{N,N}_t)$, where for every $t \in [0,T]$ and $i = 1,\ldots,N$ we have $(X^{i,N}_t,\state^{i,N}_t) \in \R^d\times\{F,L\}$. Setting
\begin{align*}
\nu^N_t \defin \frac1N \sum^N_{i = 1}\delta_{X^{i,N}_t} \quad \text{ and } \quad \sigma^N_t \defin \frac1N \sum^N_{i = 1}\delta_{\state^{i,N}_t},
\end{align*}
then their dynamics is given by 
	\begin{itemize}
	\item $dX^{i,N}_t = \langle K, \nu^N_t\times\sigma^N_t\rangle (X^{i,N}_t)dt\,,$
	\item $\state^{i, N}_t$ obeys a jump process, with conditional transition rates for the realization of $(\nu ^N_t,\sigma^N_t)$ at time $t$, given by
	\begin{itemize}
		\item if $\state^{i,N}_t = F$ then $F \rightarrow L$ with rate $\alpha_F(\nu ^N_t,\sigma^N_t)$,
		\item if $\state^{i,N}_t = L$ then $L \rightarrow F$ with rate $\alpha_L(\nu^N_t,\sigma^N_t)$.
	\end{itemize}
\end{itemize}

By virtue of the equivalence between  \eqref{eq:macroleadfollstrong} and \eqref{eq:macronusigma}, the mean-field limit for $N\rightarrow \infty$ of the above Agent Based Model is system \eqref{eq:macroleadfollstrong}.

The final part of our paper is devoted to numerical implementations of  \eqref{eq:macroleadfollstrong}. 
Three model applications are considered: 
\begin{itemize}
\item consensus dynamics for two populations $\mu_t^F, \mu_t^L$ with  a bounded confidence interaction kernel of  Hegselmann-Krause type;
\item aggregation dynamics with competition among repulsive followers $\mu_t^F$, and attractive  leaders $\mu_t^L$;
\item  the problem of steering a population towards a desired position via leaders' action.
\end{itemize}
In the case of consensus we compare the effect of suitably chosen density-dependent birth and death rates, allowing the system to reach consensus, with constant ones, where instead the system ends up clustering around different states.

For the second case, observe that aggregation models are used to describe several biological phenomena, but also as building brick of social interactions such as crowd motion \cite{cristiani2014multiscale,Bellomo2017crowd}. We show that a controlled generation of leaders, with attraction kernel, is able to confine the whole density, balancing the repulsiveness of followers' interactions.

In the third case, we study the case where leaders' generation is conditioned to achievement of a desired position, in analogy with control problem for pedestrian dynamics   \cite{albi2015invisible,burger2014mean}. Thus leaders' motion influences the followers' density towards a specific goal, whereas followers' interactions are ruled by an aggregation equation. We show that the whole population is steered to the desired state, with the leaders' mass diminishing, and eventually vanishing, as soon as the followers are sufficiently close to the final state.

As a final remark we observe that most of our results can be straightforwardly extended to the case of a \textit{finite hierarchy} of labels $\{L_1,\ldots,L_n\}$ instead of $\{F,L\}$ with transitions given by
	$$L_1 \leftrightarrow L_2 \leftrightarrow \ldots \leftrightarrow L_n \leftrightarrow L_1.$$
	All the proofs would follow along the same lines, though at the expense of notation. Actually, we conjecture that the results of the paper hold true even in the case of a countable number of labels $\{L_k\}_{k\in\N}$, as treated in a simplified scenario in \cite{thai}.

A further issue, which falls for the moment outside the scope of our methods, and is likely to require a finer analysis, is the mean-field derivation of system   \eqref{eq:macroleadfollstrong} in the case where the birth rates take values in a functional space, for example when they explicitly depend on the position $x$. We plan to address these aspects in future contributions.

The structure of the paper is the following. After discussing some measure-theoretical preliminaries in Section \ref{sec:preliminaries},  we turn our attention to system \eqref{eq:macroleadfollstrong}. We introduce a general set of assumptions and prove the existence and uniqueness of solutions, using explicit Euler approximations of the dynamics and a compactness argument in the space of positive measures with bounded mass and compact support, endowed with the generalized Wasserstein distance $\mathcal{W}_g$. This is done in Section \ref{sec:equivalence}, where we also establish a bijection between solutions of \eqref{eq:macroleadfollstrong} and of \eqref{eq:macronusigma} under certain assumptions on the initial data (Proposition \ref{p-equivalenza}). In Section \ref{sec:meanfield} we derive system \eqref{eq:macronusigma}  as mean-field limit of  a particle system which couples a SDE \eqref{stoc-X} for the particles' motion with a nonlinear master equation \eqref{stoc-Y} for their labels. Section \ref{sec-num} is devoted to numerical experiments, which make use of the finite volume scheme discussed in  \ref{sec:FVscheme}. In \ref{sec:assumptions} we introduce some explicit examples of transition functionals which comply with our assumptions and are indeed used in the experiments of Section \ref{sec-num}.

\section{Preliminaries}\label{sec:preliminaries}

Let $X$ be a Radon space; we denote by $\M(X)$ the set of finite positive measures on $X$, and by $\M_c(X)$ the subset of finite positive measures with compact support. The space $\mathcal{P}(X)$ is the subset of $\M(X)$ whose elements are the probability measures on $X$, i.e., those $\mu\in\M(X)$ for which $\mu(X)=1$. The space $\mathcal{P}_p(X)$ is the subset of $\mathcal{P}(X)$ whose elements have finite $p$-th moment, i.e.,
$$\int_X|x|^p\,d\mu(x)<+\infty.$$
%Clearly $\mathcal{P}_p(X)=\mathcal{P}(X)$ when $\O$ is bounded.
We denote by $\mathcal{P}_c(X)$ the subset of $\mathcal{P}(X)$ which consists of all probability measures with compact support. %Notice that, if $(\mu_{\ell})_{\ell \in \N}$ is a sequence in $\mathcal{P}_c(\R^d)$ and it exists $R>0$ such that $\supp(\mu_{\ell}) \subseteq B(0, R)$ for all $\ell\in\N$, then $(\mu_{\ell})_{\ell \in \N}$ is compact in $\mathcal{P}_p(\R^d)$ for all $p \geq 1$.
We denote the mass of a measure as $|\mu|=\mu(X)$.

If $X_1$ and $X_2$ are Radon spaces, for any\footnote{more in general, also if $\mu$ is a signed Borel measure on $X_1$} $\mu\in\mathcal{M}(X_1)$ and any Borel function $f:X_1\to X_2$, we denote by $f{\#}\mu\in\mathcal{M}(X_2)$ the {\it push-forward of $\mu$ through $f$}, defined by
$$f{\#}\mu(E):=\mu(f^{-1}(E))\qquad\text{for every Borel set $E$ of }X_2.$$
In particular, if one considers the projection operators $\pi_1$ and $\pi_2$ defined on the product space $X_1\times X_2$, for every $\rho\in\mathcal{P}(X_1\times X_2)$ we call {\it first} (resp., {\it second}) {\it marginal} of $\rho$ the probability measure $\pi_{1}{\#}\rho$ (resp., $\pi_{2}{\#}\rho$). Given $\mu\in\mathcal{P}(X_1)$ and $\nu\in\mathcal{P}(X_2)$, we denote by $\Gamma(\mu,\nu)$ the subset of all probability measures in $\mathcal{P}(X_1\times X_2)$ with first marginal $\mu$ and second marginal $\nu$.

We denote the weak convergence of measures as follows:
$$\mu_n\weak \mu\mbox{~~~~ when~~~~ for all~$f\in \mathcal{C}^\infty_c(\R^d)$~ it holds}\int f\,d\mu_n\to\int f\,d\mu.$$

\subsection{The Wasserstein distance}

In this section, we recall the definition of Wasserstein distance, as well as some of its useful properties.

\begin{definition}[Wasserstein distance]
For every $\mu,\nu \in \mathcal{P}_p(X)$ we define
\begin{equation}\label{e_Wp}
\W_p(\mu,\nu)\defin\inf\left\{\int_{X^2}|x-y|^p\,d\rho(x,y)\ :\ \rho\in\Gamma(\mu,\nu)\right\}^{1/p}.
\end{equation}
\end{definition}
%If $p=1$ we have the equivalent expression for the Wasserstein distance:
%$$\W_1(\mu,\nu)=\sup\left\{\int_{\R^d}\varphi(x)\,d(\mu-\nu)(x)\ :\ \varphi\in\Lip(\R^d),\ \Lip_{\R^d}(\varphi)\le1\right\},$$
%where $\Lip_{\R^d}(\varphi)$ stands for the Lipschitz constant of $\varphi$ on $\R^d$.
\begin{remark}
We denote by $\Gamma_o(\mu,\nu)$ the set of optimal plans for which the infimum in \eqref{e_Wp}  is attained, i.e.,
$$\rho\in\Gamma_o(\mu,\nu)\iff\rho\in\Gamma(\mu,\nu)\text{ and }\int_{\R^{2d}}|x-y|^p\,d\rho(x,y)=\W^p_p(\mu,\nu).$$
It is well-known that $\Gamma_o(\mu,\nu)$ is non-empty for every $(\mu,\nu)\in\mathcal{P}_p(X)\times\mathcal{P}_p(X)$, hence the infimum in \eqref{e_Wp} is actually a minimum, see \cite{villani}.
\end{remark}

\begin{remark}
    Under suitable conditions (see \cite[Theorem 5.9]{villani}), by Kantorovich-Rubinstein duality we have
  \begin{align}\label{eq: Kantorovich}  
\W_1(\mu,\nu) = \sup\left\{\int_{X}\varphi(x)d(\mu-\nu)(x) : \Lip(\varphi)\leq1\right\}.
\end{align}
 In analogy with $\Gamma_o(\mu,\nu)$, we denote by $\Lambda(\mu,\nu)$ the set of Lipschitz maps $\varphi:X\rightarrow\R$ with $\Lip(\varphi)\leq1$, and by $\Lambda_o(\mu,\nu)$ the subset of $\Lambda(\mu,\nu)$ for which the above supremum is attained, i.e.,
    $$\varphi\in\Lambda_o(\mu,\nu)\iff\varphi\in\Lambda(\mu,\nu)\text{ and }\int_{X}\varphi(x)d(\mu-\nu)(x)=\W_1(\mu,\nu).$$
    Then, by \cite[Theorem 5.9]{villani}, it follows that $\Lambda_o(\mu,\nu)$ is non-empty.
\end{remark}

We finally recall the following result, see e.g. \cite{villani}. 
\begin{proposition} Wasserstein distances are ordered, in the sense that $1 \leq p_1\leq p_2$ implies
\begin{equation*}
W_{p_1}(\mu,\nu)\leq W_{p_2}(\mu,\nu). 
\end{equation*}
\end{proposition}

\subsection{Solutions of transport equations}

We now recall the precise definition of solutions to systems \eqref{eq:macroleadfollstrong} and \eqref{eq:macronusigma}. Indeed, a solution of system \eqref{eq:macroleadfollstrong} must be interpreted in the sense of distributions, as follows.

\begin{definition}[Solution of system \eqref{eq:macroleadfollstrong}]\label{def:solmacroleadfoll}
	Let $(\overline{\mu}^F,\overline{\mu}^L)\in\mathcal{M}_c(\R^d)\times \mathcal{M}_c(\R^d)$ be given, as well as $\mu^F,\mu^L:[0,T]\rightarrow\mathcal{M}_c(\R^d)$. We say that the couple $(\mu^F_t,\mu^L_t)$ is a solution of system \eqref{eq:macroleadfollstrong} with initial datum $(\overline{\mu}^F,\overline{\mu}^L)$ when
	\begin{enumerate}
		\item $\mu^F_0 = \overline{\mu}^F$ and $\mu^L_0 = \overline{\mu}^L$;
		\item for each $i\in\{F,L\}$, the function $t\to \mu^i_t$ is continuous with respect to the topology of weak convergence of measures;
		\item there exists $R_T > 0$ such that $\bigcup_{t \in [0,T]}\supp(\mu^i_t)\subseteq B(0,R_T)$ for every $i \in \{F,L\}$;
		\item for every $\varphi \in \mathcal{C}^1_c(\R^d)$ and $i \in \{F,L\}$ it holds
		\begin{align*}\begin{split}
		\frac{d}{dt}\int_{\R^d}\varphi(x)d\mu^i_t(x)& = \int_{\R^d}\nabla\varphi(x)\cdot\left[\sum_{j\in\{F,L\}}(K^{j}\conv\mu^j_t)(x)\right]d\mu^i_t(x) \\
		&\;\;\;-\alpha_i(\mu^F_t,\mu^L_t)\int_{\R^d}\varphi(x)d\mu^i_t(x)+\alpha_{\neg i}(\mu^F_t,\mu^L_t)\int_{\R^d}\varphi(x)d\mu^{\neg i}_t(x),
		\end{split}
		\end{align*}
		for almost every $t\in[0,T]$, with
		\begin{align*}
		\neg i \defin \begin{cases}
		L & \text{ if } i = F,\\
		F & \text{ if } i = L.
		\end{cases}
		\end{align*}
	\end{enumerate}
\end{definition}

Similarly, we introduce the concept of solution of system \eqref{eq:macronusigma}.

\begin{definition}[Solution of system \eqref{eq:macronusigma}]\label{def:solnusigma}
	Let $(\overline{\nu},\overline{\sigma})\in\mathcal{M}_c(\R^d)\times\mathcal{P}(\{F,L\})$ be given, as well as $\nu:[0,T]\rightarrow\mathcal{M}_c(\R^d)$ and $\sigma:[0,T]\rightarrow\mathcal{P}(\{F,L\})$. We say that $(\nu_t,\sigma_t)$ is a solution of system \eqref{eq:macronusigma} with initial datum $(\overline{\nu},\overline{\sigma})$ when
	\begin{enumerate}
		\item $\nu_0 = \overline{\nu}$ and $\sigma_0 = \overline{\sigma}$;
		\item the function $t\to \nu_t$ is continuous with respect to the topology of weak convergence of measures, while $t\to(\sigma_t(F),\sigma_t(L))$ is absolutely continuous\footnote{It is sufficient to prove absolute continuity of one component only, since $\sigma_t(F)+\sigma_t(L)=1$.} ;
		\item there exists $R_T > 0$ such that $\cup_{t \in [0,T]}\supp(\nu_t)\subseteq B(0,R_T)$;
		\item $(\nu_t,\sigma_t)$ satisfy 
		\begin{align*}
		\dot{\sigma}_t(i) = \appstate{\nu_t}{\sigma_t}\sigma_t(i)
		\end{align*}
		for almost every $t\in[0,T]$, with $\appstate{\nu_t}{\sigma_t}$ given in \eqref{e-A}, as well as
				\begin{align*}
		\frac{d}{dt}\int_{\R^d}\varphi(x)d\nu_t(x) = \int_{\R^d}\nabla\varphi(x)\cdot\scalarp{K,\nu_t\times\sigma_t}(x)d\nu_t(x)
		\end{align*}
		for every $\varphi \in \mathcal{C}^1_c(\R^d)$.
		%\begin{align}\label{eq:macronuweak}
		%\frac{d}{dt}\int_{\R^d}\varphi(x)d\nu_t(x) = %\int_{\R^d}\nabla\varphi(x)\cdot\scalarp{K,\nu_t\times\sigma_t}(x)d\nu_t(x);
		%\end{align}
		%\item for every $\psi:\{F,L\} \rightarrow\R$ it holds
		%\begin{align}\label{eq:macrosigmaweak}
		%\frac{d}{dt}\int_{\{F,L\}}\psi(i)d\sigma_t(i) = %\int_{\{F,L\}}\appstate{\nu_t}{\sigma_t}\psi(i)d\sigma(i).
		%\end{align}
	\end{enumerate}
\end{definition}

\begin{remark}\label{rem:simplify}
Throughout the paper it is assumed that the transition rates encoded by the matrix $A_{\nu,\sigma}$ only depend on the global state of the system and not on the position $x$. While being already useful at the level of deducing existence of solutions of  \eqref{eq:macroleadfollstrong}, this restriction will be needed in order to show equivalence between solutions of \eqref{eq:macroleadfollstrong} and \eqref{eq:macronusigma}, provided the initial datum is suitably chosen, satisfying Assumption \ref{ass:posmeas0} below. This will be apparent in the proof of Proposition \ref{p-equivalenza}. As already discussed in the Introduction, this equivalence is a crucial step for the mean field derivation in Section \ref{sec:meanfield}.
\end{remark}

\subsection{The method of characteristics}

In this section, we recall the method of characteristics to find solutions of transport equations. In particular, we recall the connection between the solutions of an ordinary differential equation with vector field $v$ and the  solution to transport equations as the evolution of the corresponding probability distribution.

We start with the classical definitions of Carath{\'e}odory functions and solutions.
\begin{definition}
A function $g:[0,T]\times \Omega \to \R^d$ is a Carath{\'e}odory function if
\begin{enumerate}
\item For all $t\in[0,T]$, the application $x\mapsto g(t,x)$ is Lipschitz.
\item For all $x\in\mathbb{R}^d$, the application $t\mapsto g(t,x)$ is measurable.
\item There exists $M>0$ such that $|g(t,x)|\leqslant M(1+|x|)$ for all $t,x$.
\end{enumerate}

A Carath{\'e}odory solution of
\begin{equation}\label{cara}
\dot y(t)=g(t, y(t)) \quad \text{ for } t\in[0,T],
\end{equation}
is an absolutely continuous function $y: [0,T]\to \R^d$ which satisfies \eqref{cara}  a.e.\ in $[0,T]$.
\end{definition} 
If the Lipschitz constant $L_t$ of the function $g(t,\cdot)$ belongs to $L^1(0,T)$,  existence and uniqueness of  Carath{\'e}odory solutions to \eqref{cara} can be shown, see e.g. \cite{filipov}. From now on, we denote by $\Phi^g_t$ the flow of \eqref{cara}, i.e. the map $x_0\mapsto\Phi^g_t(x_0)$ that associates to each initial data $x_0$ the corresponding solution of \eqref{cara} at time $t$. Carath\'eodory solutions of finite dimensional systems and weak solutions of continuity equations are intimately related, as the following classical result shows.

\begin{lemma}\label{le:equivPDEODEtransp}
	Let $v:[0,T]\times\R^d\rightarrow\R^d$ be a Carath\'eodory function and $X:[0,T]\rightarrow\R^d$ be a Carath\'eodory solution of
	\begin{eqnarray*}
	\begin{cases}
	\dot{x} = v_t(x), \\
	x(0) = x_0.
	\end{cases}
	\end{eqnarray*}
	Then $\mu_t = {\Phi^v_t}\#\mu_0$ is the unique weak solution of
	\begin{equation}
	\begin{cases}
	\partial_t\mu_t &= - \diver(v_t\mu_t) , \\
	\mu(0) &= \mu_0.
	\end{cases}\label{e-pde}
	\end{equation}
	 As a consequence, if $\supp(\mu_0)\subset B(0,R)$, then for each $t>0$ it holds
\begin{eqnarray}
\supp(\mu_t)\subset B(0,R+t\|v\|_{\mathcal{C}^0}).\label{e-evolsupp}
\end{eqnarray}

Moreover, consider the inhomogeneous transport equation
\begin{equation}
	\begin{cases}
	\partial_t\mu_t &= - \diver(v_t\mu_t) +s_t, \\
	\mu(0) &= \mu_0.
	\end{cases}\label{e-pde-inhom}
	\end{equation}
for $s_t$ being a measurable family (with respect to the weak topology of meaures) of signed Borel measures such that there exist $M_s,R_s$ with 
$$|s_t^+|+|s_t^-|\leq M_s,\qquad \supp(s_t)\subset B(0,R_s)$$
for all $t\in[0,T]$.

Then, there exists a unique solution to \eqref{e-pde-inhom}, that satisfies the Duhamel's formula
\begin{equation}
\mu_t=\Phi^v_t\#\mu_0+\int_0^t\Phi^v_{(\tau,t)}\#s_\tau\,d\tau.
\label{e-duhamel}
\end{equation}
Here, $\Phi^v_{(\tau,t)}$ is the flow of the non-autonomous vector field $v_t$ starting at time $\tau$, i.e. the function $x_\tau \mapsto \Phi^v_{(\tau,t)}(x_\tau)$ that associates to $x_\tau$ the solution at time $t$ of 
	\begin{eqnarray*}
	\begin{cases}
	\dot{x} = v_t(x), \\
	x(\tau)  = x_\tau.
	\end{cases}
	\end{eqnarray*}
	As a consequence, if $\supp(\mu_0)\subset B(0,R)$, then for each $t>0$ it holds
\begin{eqnarray}
\supp(\mu_t)\subset B(0,\max\{R,R_s\}+t\|v\|_{\mathcal{C}^0}).\label{e-evolsuppinhom}
\end{eqnarray}
\end{lemma}
\begin{proof}
	For the existence of a solution to \eqref{e-pde}, which is the push-forward of the initial datum via the flow map, see e.g. \cite{villani}. Uniqueness comes from standard arguments for the linear continuity equation, see e.g. \cite{AGS}.

We now prove \eqref{e-evolsupp}. For a given $t>0$, consider a test function $\varphi$ with compact support, such that $\varphi\equiv 0$ on $B(0,R+t\|v\|_{\mathcal{C}^0})$. It then holds
\begin{equation}
\label{e-intsupp}
\int_{\R^d} \varphi(x)\,d(\Phi^v_t\#\eta_0)(x) =\int_{\R^d} \varphi(\Phi^v_t(x))\,d\eta_0(x).
\end{equation}
Recall the elementary estimate for ordinary differential equations $$|\Phi^v_t(x)-x|=\left|\int_0^t v(s,x(s))\,ds\right|\leq t\|v\|_{\mathcal{C}^0}.$$ Since for each $x\in\supp(\mu(0))$ it holds $|x|\leq R$, then $\varphi(\Phi^v_t(x))=0$. Thus, the integral in \eqref{e-intsupp} is zero. Since this holds for any test function with support outside $B(0,R+t\|v\|_{\mathcal{C}^0})$, this implies that $\eta_t$ is supported in $B(0,R+t\|v\|_{\mathcal{C}^0})$.

The proof of existence for the inhomogeneous case is similar. Duhamel's formula is a re-writing of the method of variations of constants, that can be verified with direct computations. Uniqueness can be proved with the standard method: the difference between two solutions solves \eqref{e-pde} with $\mu_0\equiv 0$, then its unique solution is $\mu_t\equiv 0$. The proof for \eqref{e-evolsuppinhom} follows the proof for \eqref{e-evolsupp}.
\end{proof}

\subsection{The Generalized Wasserstein distance}

The main technical issue about the transport equation \eqref{eq:macroleadfollstrong} is that it mixes two different phenomena: on one side the non-local dynamics given by convolutions $K^i\conv \mu^i$; on the other side, sources and sink that make the total mass of $\mu^i$ non-constant.

It has been shown in several examples that the Wasserstein distance is a powerful tool to deal with transport equation with non-local vector fields, see
e.g. \cite{ambrosio,pedestrian,dobrushin79,golseparticle,villani}. Neverthelss, the Wasserstein distance is defined between measures with the same mass, hence it is not useful for problems in which the mass varies in time. This issue recently led to the development of a series of different generalizations of the Wasserstein distance to measures with different masses. See e.g. \cite{Chizat2018,KMV16,LMS16,gw}.

In this article, we choose to use the generalized Wasserstein distance, that has been introduced in \cite{gw,gw2}. Indeed, it has been already proved in \cite{gw} that, under suitable hypotheses written in terms of the generalized Wasserstein distance, transport equations with both non-local velocities and source terms admit existence and uniqueness of the solution.

We now recall the definition of the generalized Wasserstein distance, together with some key properties.
\begin{definition}
Let $\mu,\nu\in\Mu$ be two measures. Given $a,b>0$ and $p\geq 1$, we define the functional
\begin{equation}
\gwbase^{a,b,p}(\mu,\nu):=\inf_{\tilde\mu,\tilde\nu\in\Mu,\,|\tilde\mu|=|\tilde\nu|}\Pt{a^p\Pt{|\mu-\tilde\mu|+|\nu-\tilde\nu|}^p+b^pW_p^p(\tilde\mu,\tilde\nu)}^{1/p}.\label{e-gw}
\end{equation}
\end{definition}

\begin{proposition} \label{p-gw}
The following properties hold:
\begin{enumerate}
\item The functional $\gwbase^{a,b,p}$ is a distance on $\Mu$.\\
\item The distance $\gwbase^{a,b,p}$ metrizes the weak convergence on compact sets, i.e. given $\mu_n,\mu$ with  $\supp(\mu_n),\supp(\mu)\subset B_R(0)$ it holds
$$\mu_n\rightharpoonup \mu\mbox{~~~~if and only if~~~~}\gwbase^{a,b,p}(\mu_n,\mu)\to 0.$$
\item The space $\Mu$ is complete with respect to $\gwbase^{a,b,p}$.
\item Let $v_t,w_t$ be two Lipschitz vector fields, with $L$ a Lipschitz constant for both $v_t,w_t$. It then holds:
\begin{eqnarray}
&&\hspace{-18mm}\gwbase^{a,b,p}(\mu,\Phi^v_t\#\mu)\leq b \sup_{\tau\in[0,t]}\{\|v_\tau\|_{\mathcal{C}^0}\} t\,|\mu|\label{e-gw1} \\
&&\hspace{-18mm}\gwbase^{a,b,p}(\Phi^{v}_{t}\#\mu,\Phi^{w}_{t}\#\nu)\leq e^{\frac{p+1}{p} L t} \gwbase^{a,b,p}(\mu,\nu)+
\mbox{${\tiny |\mu|\frac{be^{L t/p}(e^{Lt}-1)}{L} }$}\sup_{\tau\in[0,t]}\{\|v_\tau-w_\tau\|_{\mathcal{C}^0}\}.\label{e-gw2}
\end{eqnarray}
\item It holds 
\begin{eqnarray}
\gw{\mu,\nu}\leq |\mu|+|\nu|.
\label{e-gwmass}
\end{eqnarray}
\end{enumerate}
\end{proposition}
\begin{proof}See \cite{gw,gw2}.% Item 4 is a particular case of \cite[Theorem 3]{gw}.
\end{proof}

We now recall a result equivalent to the Kantorovich-Rubinstein duality for the generalized Wasserstein distance $\gwbase^{1,1,1}$. It states that it coincides with the so-called flat distance, see e.g. \cite{dudley}.
\begin{theorem} \label{t-flat}
Let $\mu,\nu\in\Mu$. Then 
\begin{eqnarray*}
\gwbase^{1,1,1}(\mu,\nu)=\sup\left\{\int f d(\mu-\nu)\ |\  \|f\|_\infty\leq 1,\ \mathrm{Lip}(f)\leq 1\right\}.%\label{e-flat}
\end{eqnarray*}
As simple consequences, for $\lambda,\bar\lambda>0$ it holds
\begin{eqnarray}
&&\W_g^{1,1,1}(\lambda\mu,\bar\lambda\mu)=|\lambda-\bar\lambda|\,|\mu|\label{e-diff}\\
&&\W_g^{1,1,1}(\lambda\mu,\lambda\nu)=\lambda\W_g^{1,1,1}(\mu,\nu).\label{e-omo}
\end{eqnarray}
\end{theorem}
The proof is given in \cite{gw2}.

From now on, we will only deal with the generalized Wasserstein distance $\gwbase^{1,1,1}$, i.e. with the flat distance. For this reason, we will drop the parameters, and use the notation $$\gw{\mu,\nu}:=\gwbase^{1,1,1}(\mu,\nu).$$
Moreover, we use the same notation for the corresponding distance on $\mathcal{M}_c(\R^d)\times \mathcal{M}_c(\R^d)$: given $\mu=(\mu^F,\mu^L)$ and $\nu=(\nu^F,\nu^L)$, we write
\begin{eqnarray}
\gw{\mu,\nu}:=\gw{\mu^F,\nu^F} + \gw{\mu^L,\nu^L}.
\label{e-distFL}
\end{eqnarray}

Finally, we use again the same notation for the supremum distance on $C([0,T],\mathcal{M}_c(\R^d)\times \mathcal{M}_c(\R^d))$: given $\mu:t\mapsto\mu_t=(\mu^F_t,\mu^L_t)$ and $\nu:t\mapsto\nu_t=(\nu^F_t,\nu^L_t)$, we write
\begin{eqnarray}
\gw{\mu,\nu}:=\sup_{t\in[0,T]}\W_g((\mu^{F,k}_t,\mu^{L,k}_t),(\nu^{F,k}_t,\nu^{L,k}_t)).\label{e-distsup}
\end{eqnarray}

\section{Well-posedness and equivalence for the leader-follower dynamics}\label{sec:equivalence}

We now turn our attention to system \eqref{eq:macroleadfollstrong} and use the tools introduced in Section \ref{sec:preliminaries} to prove the existence and uniqueness of solutions. To do so, we will define a sequence of measures $(\mu^{F,k},\mu^{L,k})$ as explicit Euler approximations of the dynamics \eqref{eq:macroleadfollstrong} and, by a a compactness argument in the space $\mathcal{M}_c(\R^d)$ embedded with the generalized Wasserstein distance $\mathcal{W}_g$, we show that it converges, up to subsequences, to the unique solution $(\mu^F,\mu^L)$ of system \eqref{eq:macroleadfollstrong}. Next, we shall establish a bijection between solutions of \eqref{eq:macroleadfollstrong} and of \eqref{eq:macronusigma} under certain assumptions on the initial data. As a byproduct of the previous results, such equivalence yields the well-posedness of \eqref{eq:macronusigma} as well, paving the way for the mean-field analysis of the subsequent sections.

\subsection{Main assumptions}

In this section we discuss the set of assumptions we shall assume henceforth. These assumptions assure, in particular, the existence and uniqueness of solutions of \eqref{eq:macroleadfollstrong}, as well as the  equivalence between \eqref{eq:macroleadfollstrong} and \eqref{eq:macronusigma}, that is more amenable to a mean-field analysis, as we will show in Section \ref{sec:meanfield}. We warn in advance the reader that Assumption \ref{ass:posmeas0} below, differently from the other ones, is not needed for the existence result in Proposition \ref{p-ex}, but will be used for the equivalence result in Proposition \ref{p-equivalenza}.
\begin{enumerate}[label=(H\arabic*)]
    \item \label{ass:posmeas0}\label{h1} There exist $\overline{\sigma} \in\mathcal{P}(\{F,L\})$ and $\overline{\nu} \in \mathcal{M}_c(\R^d)$ such that $\overline{\mu}^F = \overline{\sigma}(F)\overline{\nu}$ and $\overline{\mu}^L=\overline{\sigma}(L)\overline{\nu}$.
	\item \label{ass:lipK}\label{h2} there exists a constant $L_K > 0$ such that, for every $x_1,x_2 \in \R^d$ and $i \in \{F,L\}$, it holds
	$$|K^i(x_1) - K^i(x_2)| \leq L_K|x_1 - x_2|.$$
	\item \label{ass:sublinK} \label{h3}there exists a constant $B_K > 0$ such that, for every $x \in \R^d$ and $i\in\{F,L\}$, it holds
	$$|K^i(x)| \leq B_K(1 + |x|).$$
	%\item\label{item:alphaboundbelow} there exists $c>0$ such that $\alpha(i,\mu^F_t,\mu^L_t)\in(c,1]$;
	
	\item \label{ass:supalpha} \label{h4}there exists a constant $M_{\alpha}$ such that for every $i \in \{F,L\}$ and $(\mu^F,\mu^L) \in \mathcal{M}_c(\R^d)\times\mathcal{M}_c(\R^d)$ it holds $$0\leq \alpha_i(\mu^F,\mu^L)\leq M_\alpha.$$

	\item \label{ass:lipalpha} \label{h5} there exists a constant $L_{\alpha, M, R}$ such that, for every $i \in \{F,L\}$ and $(\mu^F,\mu^L),(\nu^F,\nu^L) \in \mathcal{M}_c(\R^d)\times\mathcal{M}_c(\R^d)$ satisfying 
	\begin{equation}
	|\mu^F|+|\mu^L|=|\nu^F|+|\nu^L|\le M,
	\label{e-massconstraint}
\end{equation}	
and
\begin{equation}
	\mathrm{supp }(\mu^j),\mathrm{supp }(\nu^j)\subset B(0, R),\qquad j\in\{F,L\}
	\label{e-boundedsupp}
\end{equation}	
 it holds
\begin{equation}
	|\alpha_i(\mu^F,\mu^L)- \alpha_i(\nu^F,\nu^L)|\leq L_{\alpha, M, R}(\gw{\mu^F,\nu^F} + \gw{\mu^L,\nu^L}).
	\label{e-gwLip-prel}
	\end{equation}
	\end{enumerate}

We now list some useful consequences of the previous hypotheses.

\begin{proposition} Let \ref{h4}-\ref{h5} hold, and $\mu,\nu$ satisfy \eqref{e-massconstraint}-\eqref{e-boundedsupp}. Then, it exists $L'_{\alpha,M,R}$ such that for each $i\in\{F,L\}$ it holds
\begin{equation}
	\gw{\alpha_i(\mu^F,\mu^L)\mu^i, \alpha_i(\nu^F,\nu^L)\nu^i}\leq L'_{\alpha, M, R}\gw{\mu,\nu}
	\label{e-gwLip}
	\end{equation}
\end{proposition}
\begin{proof} Use the triangular inequality and \eqref{e-diff}-\eqref{e-omo} to write
\begin{eqnarray*}
&&\gw{\alpha_i(\mu^F,\mu^L)\mu^i, \alpha_i(\nu^F,\nu^L)\nu^i}\leq\\
&&\gw{\alpha_i(\mu^F,\mu^L)\mu^i, \alpha_i(\nu^F,\nu^L)\mu^i}+\gw{\alpha_i(\nu^F,\nu^L)\mu^i, \alpha_i(\nu^F,\nu^L)\nu^i}=\\
&&|\alpha_i(\mu^F,\mu^L)- \alpha_i(\nu^F,\nu^L)|\,|\mu^i|+\alpha_i(\nu^F,\nu^L)\gw{\mu^i,\nu^i}\leq\\
&& L_{\alpha, M, R}(\gw{\mu^F,\nu^F} + \gw{\mu^L,\nu^L}) M + M_\alpha \gw{\mu^i,\nu^i},
\end{eqnarray*}
from which the result easily follows.
\end{proof}

For $\nu \in \mathcal{P}_1(\mathbb{R}^d)$ and $\sigma \in \mathcal{P}_1(\{F, L\})$, we will use (as  already done in \eqref{e-A}) the notations $\alpha_F(\nu, \sigma)$ and $\alpha_L(\nu, \sigma)$ to indicate the transition rates defined by
\begin{equation}\label{eq:alphashortcut}
\alpha_F(\nu, \sigma):=\alpha_F(\sigma(F)\nu, \sigma(L)\nu),\quad \alpha_L(\nu, \sigma):=\alpha_L(\sigma(F)\nu, \sigma(L)\nu)\,.
\end{equation}
If \eqref{e-gwLip-prel} holds, it easily follows from the definition \eqref{e-gw} that, for $\nu_1$, $\nu_2 \in \mathcal{P}_1(\mathbb{R}^d)$ satisfying \eqref{e-boundedsupp}  and $\sigma_1$, $\sigma_2 \in \mathcal{P}_1(\{F, L\})$, we have
\[
|\alpha_i(\nu_1,\sigma_1)- \alpha_i(\nu_2,\sigma_2)|\leq L_{\alpha, R} (\mathcal{W}_1(\nu_1, \nu_2)+ |\sigma_1(F)-\sigma_2(F)|)
\]
for $i=F, L$. We additionally exploited above the inequality $\gw{\nu_1, \nu_2}\le \mathcal{W}_1(\nu_1, \nu_2)$ which immediately stems out of \eqref{e-gw} whenever $\nu_1$ and $\nu_2$ are probability measures. If we endow the set $\{F, L\}$ with the usual distance on finite sets defined by  
\begin{align}\label{eq:distfinite}
    |y - \overline{y}|_{\{F,L\}} \defin \begin{cases}
0 & \text{ if } y = \overline{y},\\
1 & \text{ otherwise.}
\end{cases}\end{align}
we can rewrite the above inequality as
\begin{equation}\label{e-w1lip}
|\alpha_i(\nu_1,\sigma_1)- \alpha_i(\nu_2,\sigma_2)|\leq L_{\alpha, R} (\mathcal{W}_1(\nu_1, \nu_2)+ \mathcal{W}_1(\sigma_1,\sigma_2))\,.
\end{equation}

\subsection{Existence and uniqueness}

In this section, we prove existence and uniqueness of the solutions to Cauchy problems with dynamics given by systems \eqref{eq:macroleadfollstrong} and \eqref{eq:macronusigma}. For the first case, we will adapt ideas from \cite{gw}, while for the second we will use the equivalence of the two problems.

We first prove an existence result for \eqref{eq:macroleadfollstrong}.

\begin{proposition} \label{p-ex}
Let an initial data $(\mu^F_0,\mu^L_0)\in \mathcal{M}_c(\R^d)\times \mathcal{M}_c(\R^d)$ and a time interval $[0,T]$ be fixed. For each $k\in \mathbb{N}$, define an {\bf explicit Euler approximation} $\mu^{F,k},\mu^{L,k}$ of the solution to system \eqref{eq:macroleadfollstrong} as follows: fix $\Delta t=T/2^k$ and define
\begin{eqnarray}
&&\hspace{-15mm}(\mu_0^{F,k},\mu_0^{L,k}):=(\mu_0^{F},\mu_0^{L});\label{e-scheme0}\\
&&\hspace{-15mm}v^k_{n\Delta t}:=K^F\conv \mu^{F,k}_{n\Delta t}+K^L\conv \mu^{L,k}_{n\Delta t},\hspace{5cm}
n=0,\ldots,2^k,\label{e-schemev}\\
&&\hspace{-15mm}\mu^{F,k}_{(n+1)\Delta t}:=\Phi^{v^k_{n\Delta t}}_{\Delta t}\#\left(\mu^{F,k}_{n\Delta t} +\Delta t (-\alpha_F(\mu^{F,k}_{n\Delta t},\mu^{L,k}_{n\Delta t})\mu^{F,k}_{n\Delta t}+\alpha_L(\mu^{F,k}_{n\Delta t},\mu^{L,k}_{n\Delta t})\mu^{L,k}_{n\Delta t})\right);\label{e-schemeF}\\
&&\hspace{-15mm}\mu^{L,k}_{(n+1)\Delta t}:=\Phi^{v^k_{n\Delta t}}_{\Delta t}\#\left(\mu^{L,k}_{n\Delta t} +\Delta t (\alpha_F(\mu^{F,k}_{n\Delta t},\mu^{L,k}_{n\Delta t})\mu^{F,k}_{n\Delta t}-\alpha_L(\mu^{F,k}_{n\Delta t},\mu^{L,k}_{n\Delta t})\mu^{L,k}_{n\Delta t})\right).\label{e-schemeL}
\end{eqnarray}
Also define the solution on intermediate times: for $\tau\in (0,1)$ define
\begin{eqnarray}
&&\hspace{-15mm}\mu^{F,k}_{(n+\tau)\Delta t}:=\Phi^{v^k_{n\Delta t}}_{\tau\Delta t}\#\left(\mu^{F,k}_{n\Delta t} +\tau\Delta t (-\alpha_F(\mu^{F,k}_{n\Delta t},\mu^{L,k}_{n\Delta t})\mu^{F,k}_{n\Delta t}+\alpha_L(\mu^{F,k}_{n\Delta t},\mu^{L,k}_{n\Delta t})\mu^{L,k}_{n\Delta t})\right);\label{e-schemeFt}\\
&&\hspace{-15mm}\mu^{L,k}_{(n+\tau)\Delta t}:=\Phi^{v^k_{n\Delta t}}_{\tau\Delta t}\#\left(\mu^{L,k}_{n\Delta t} +\tau\Delta t (\alpha_F(\mu^{F,k}_{n\Delta t},\mu^{L,k}_{n\Delta t})\mu^{F,k}_{n\Delta t}-\alpha_L(\mu^{F,k}_{n\Delta t},\mu^{L,k}_{n\Delta t})\mu^{L,k}_{n\Delta t})\right).\label{e-schemeLt}
\end{eqnarray}

Let Hypotheses \ref{ass:lipK}-\ref{ass:sublinK}-\ref{ass:supalpha}-\ref{ass:lipalpha} hold. Let moreover be $\Delta t M_\alpha<1$. Then, the following properties hold:
\begin{enumerate}
\item both $\mu^{F,k}_t$ and $\mu^{F,k}_t$ are non-negative measures;
\item the total mass is preserved, since it satisfies 
\begin{equation}
|\mu^{F,k}_t|+|\mu^{L,k}_t|=|\mu^{F}_0|+|\mu^{L}_0|;\label{e-massa}
\end{equation}
\item the sequence has equi-bounded support, i.e there exists $R>0$ such that for all $t\in[0,T]$ and $k\in\mathbb{N}$ it holds
$$\mathrm{supp}(\mu^{F,k}_t),\mathrm{supp}(\mu^{L,k}_t)\subset B(0,R);$$
\item the sequence $\{(\mu^{F,k}_t,\mu^{L,k}_t)\}_{k\in\mathbb{N}}$ is uniformly bounded and uniformly Lipschitz in the $t$ variable with respect to the distance \eqref{e-distFL}.
\end{enumerate}
As a consequence, there exists a subsequence of $(\mu^{L,k},\mu^{F,k})$ converging with respect to the uniform convergence, i.e. with respect to the metric \eqref{e-distsup}.
%\begin{eqnarray}
%\W_g(\mu^{k},\nu^{k})=
%\W_g((\mu^{F,k},\mu^{L,k}),(\nu^{F,k},\nu^{L,k})):=\sup_{t\in[0,T]}\W_g((\mu^{F,k}_t,\mu^{L,k}_t),(\nu^{F,k}_t,\nu^{L,k}_t)).\label{e-distsup}
%\end{eqnarray} 
The limit of such subsequence is a solution to \eqref{eq:macroleadfollstrong}.

\end{proposition}
\begin{proof} We prove {\bf Property 1}. We first prove that $\mu^{F,k}_{n\Delta t},\mu^{L,k}_{n\Delta t}$ are non-negative measures for each $n=0,\ldots, 2^k$, by induction on $n$. It is clear that the property holds for $n=0$, since \eqref{e-scheme0} holds.

Let now be $\mu^{F,k}_{n\Delta t},\mu^{L,k}_{n\Delta t}$ non-negative measures. We aim to prove that $\mu^{F,k}_{(n+1)\Delta t},\mu^{L,k}_{(n+1)\Delta t}$ given by \eqref{e-schemeF}-\eqref{e-schemeL} are non-negative measures. We only prove it for $\mu^{F,k}_{(n+1)\Delta t}$, since the proof for $\mu^{L,k}_{(n+1)\Delta t}$ is similar. Observe that $\Delta t M_\alpha<1$, together with \ref{ass:supalpha}, implies 
\begin{equation*}
1-\Delta t \alpha_F(\mu^{F,k}_{n\Delta t},\mu^{L,k}_{n\Delta t})>0.
\end{equation*}
Then $\mu^{F,k}_{n\Delta t} (1-\Delta t \alpha_F(\mu^{F,k}_{n\Delta t},\mu^{L,k}_{n\Delta t}))$ is a non-negative measure, as well as $\Delta t \alpha_L(\mu^{F,k}_{n\Delta t},\mu^{L,k}_{n\Delta t})\mu^{L,k}_{n\Delta t}$. Their sum is thus a non-negative measure, and its push-forward by $\Phi^{v^k_{n\Delta t}}_{\Delta t}$ is non-negative too. By induction, this proves that $\mu^{F,k}_{n\Delta t},\mu^{L,k}_{n\Delta t}$ are non-negative measures for each $n=0,\ldots, 2^k$.

For intermediate times of the form $(n+\tau)\Delta t$, first observe that we just proved that $\mu^{F,k}_{n\Delta t},\mu^{L,k}_{n\Delta t}$ are non-negative measures. Moreover,  $\tau\in(0,1)$ implies $$1-\tau\Delta t \alpha_F(\mu^{F,k}_{n\Delta t},\mu^{L,k}_{n\Delta t})>0.$$ Then, following the proof of the previous case, we have that $\mu^{F,k}_{(n+\tau)\Delta t},\mu^{L,k}_{(n+\tau)\Delta t}$ are non-negative measures.

We now prove {\bf Property 2}. We first prove that \eqref{e-massa} holds for times of the form $n\Delta t$, again by induction on $n$. Definition \eqref{e-scheme0} implies that \eqref{e-massa} holds for $n=0$. If \eqref{e-massa} holds for a given $n$, then it holds for $n+1$, as a consequence of \eqref{e-schemeF}-\eqref{e-schemeL}. Indeed, by the proof of Proposition 1, we know that both $\mu^{F,k}_{n\Delta t} (1-\Delta t \alpha_F(\mu^{F,k}_{n\Delta t},\mu^{L,k}_{n\Delta t}))$ and $\Delta t \alpha_L(\mu^{F,k}_{n\Delta t},\mu^{L,k}_{n\Delta t})\mu^{L,k}_{n\Delta t}$ are non-negative measures, and the same holds for the corresponding terms in \eqref{e-schemeL}. Thus, the mass of the sum is the sum of the masses, and the push-forward of a non-negative measure preserves the mass. As a consequence, it holds
\begin{eqnarray*}
|\mu^{F,k}_{(n+1)\Delta t}|&+&|\mu^{L,k}_{(n+1)\Delta t}|=(1-\Delta t \alpha_F(\mu^{F,k}_{n\Delta t},\mu^{L,k}_{n\Delta t}))|\mu^{F,k}_{n\Delta t}| +\Delta t \alpha_L(\mu^{F,k}_{n\Delta t},\mu^{L,k}_{n\Delta t})|\mu^{L,k}_{n\Delta t}|+\nonumber\\
&&(1-\Delta t \alpha_L(\mu^{F,k}_{n\Delta t},\mu^{L,k}_{n\Delta t})\mu^{L,k}_{n\Delta t})|\mu^{L,k}_{n\Delta t}| +\Delta t \alpha_F(\mu^{F,k}_{n\Delta t},\mu^{L,k}_{n\Delta t})|\mu^{F,k}_{n\Delta t}|=\nonumber\\
&&|\mu^{F,k}_{n\Delta t}|+|\mu^{L,k}_{n\Delta t}|=|\mu^{F}_0|+|\mu^{F}_0|,
\end{eqnarray*}
where we used homogeneity of the mass $|\lambda \mu|=\lambda |\mu|$. The proof for intermediate times is  identical.

We now prove {\bf Property 3}. First observe that, due to \ref{ass:sublinK}, the hypothesis $$\supp(\mu^F),\supp(\mu^L)\subset B(0,R)$$ implies
\begin{eqnarray}
\|K^F\conv \mu^F+ K^L\conv\mu^L\|_{\mathcal{C}^0}\leq \|K^F\|_{\mathcal{C}^0} |\mu^F|+\|K^L\|_{\mathcal{C}^0} |\mu^L|\leq B_K (1+2R)(|\mu^F|+ |\mu^L|)\label{e-stimaC0}.
\end{eqnarray}

Choose now $R_0>0$ such that $\supp(\mu^F_0), \supp(\mu^L_0)\subset B(0,R_0)$. We now define a sequence $R_n^k$ such that  
\begin{equation}
\supp(\mu^{F,k}_{n\Delta t}), \supp(\mu^{L,k}_{n\Delta t})\subset B(0,R_n^k),\label{e-suppnk}
\end{equation} by induction. It first holds $R_0^k=R_0$ by \eqref{e-scheme0}. By definition of $v^k_{n\Delta t}$ in \eqref{e-schemev}, and also using \eqref{e-stimaC0} and Property 1, it holds
\begin{eqnarray}
\|v^k_{n\Delta t}\|_{\mathcal{C}^0}\leq B_K (1+2R^k_n)(|\mu^{F,k}_{n\Delta t}|+ |\mu^{L,k}_{n\Delta t}|)=B_K (1+2R^k_n)(|\mu^{F}_{0}|+ |\mu^{L}_{0}|). \label{e-stimavkn}
\end{eqnarray}
Apply \eqref{e-evolsupp} to \eqref{e-schemeF}-\eqref{e-schemeL}: since $\supp(\mu^{F,k}_{n\Delta t}), \supp(\mu^{L,k}_{n\Delta t})\subset B(0,R_n^k)$, then
\begin{eqnarray}
\supp(\mu^{F,k}_{(n+1)\Delta t}), \supp(\mu^{L,k}_{(n+1)\Delta t})\subset B(0,R_n^k+\Delta t B_K (1+2R^k_n)(|\mu^{F}_{0}|+ |\mu^{L}_{0}|)).
\label{e-supp1}
\end{eqnarray}

Define now the sequence $$R^k_{0}:=R_0,\qquad R^k_{n+1}=(1+\Delta t\, C) R^k_n+\Delta t\,C$$ with $C:=
2B_K (|\mu^{F}_{0}|+ |\mu^{L}_{0}|))$. With this choice, \eqref{e-suppnk} holds. Moreover, again by applying \eqref{e-evolsupp}  to the definition of $\mu^k_t$ at intermediate times \eqref{e-schemeFt}-\eqref{e-schemeLt}, for each $\tau\in (0,1)$ it holds 
\begin{eqnarray}
\supp(\mu^{F,k}_{(n+\tau)\Delta t}), \supp(\mu^{L,k}_{(n+\tau)\Delta t})\subset B(0,R^k_{n+1}).
\label{e-supp2}
\end{eqnarray}

We now recall that $n$ runs from $0$ to $2^k$. Since $R^k_n$ is an increasing sequence with respect to the parameter $n$, then \eqref{e-supp1}-\eqref{e-supp2} imply that for each $k\in\mathbb{N}$ and each $t\in[0,T]$ it holds
\begin{eqnarray*}
\supp(\mu^{F,k}_{t}), \supp(\mu^{L,k}_{t})\subset B(0,R^k_{2^k}).
\end{eqnarray*}
An explicit computation shows that 
\begin{eqnarray}
R^k_{2^k}=(1+\Delta t\, C)^{2^k} (R_0+1)-1\leq e^{2^{k} \Delta t\,C}(R_0+1)<e^{TC}(R_0+1),\label{e-globsupp}
\end{eqnarray}
thus, supports of $\mu^k_t$ are uniformly bounded.

We now prove {\bf Property 4}. Since we proved that $|\mu^{F,k}_{t}|+|\mu^{L,k}_{t}|=|\mu^{F}_0|+|\mu^{L}_0|$, it holds both $|\mu^{F,k}_t|\leq |\mu^{F}_0|+|\mu^{L}_0|$ and $|\mu^{F,k}_t|\leq |\mu^{F}_0|+|\mu^{L}_0|$. Then, by applying \eqref{e-gwmass}, it holds $$\gw{\mu^{F,k}_t,\mu^{L,k}_t}\leq |\mu^{F,k}_t|+|\mu^{L,k}_t|\leq 2(|\mu^{F}_0|+|\mu^{L}_0|),$$ then the sequence is equi-bounded.

We now prove equi-Lipschitz continuity. Let $k\in\mathbb{N}$ be fixed, and assume to have $t,s$ such that $n\Delta t\leq t<s\leq (n+1)\Delta t$. We then want to estimate  $\W_g(\mu^{k}_t,\mu^{k}_s)$. Observe that, by \eqref{e-schemeFt}-\eqref{e-schemeLt} and the property of composition of flows, it holds $\mu^{F,k}_s=\Phi^{v^k_{n\Delta t}}_{s-t}\#\mu^{F,k}_t$, and similarly for $\mu^{L,k}$. Apply now \eqref{e-gw1} to $v^k_{n\Delta t}$, that satisfies \eqref{e-stimavkn} and recall that $R^k_n\leq e^{TC}(R_0+1)$, as proved for Property 3. This implies 
\begin{eqnarray}
\W_g(\mu^{F,k}_t,\mu^{F,k}_s)\leq B_K(1+2e^{TC}(R_0+1))(|\mu^{F}_0|+|\mu^{L}_0|)^2|t-s|.
\label{e-stimaLip}
\end{eqnarray}
The same estimate holds for $\W_g(\mu^{L,k}_t,\mu^{L,k}_s)$, then for $\W_g(\mu^{k}_t,\mu^{k}_s)$ by doubling the right hand side. For general $t<s\in[0,T]$, one recovers \eqref{e-stimaLip} by applying the triangular inequality on each sub-interval $[t,n\Delta t ]$, $[n\Delta t,(n+1)\Delta t]$, $\ldots$, $[(n+k)\Delta t, s]$.

We finally prove the {\bf existence of a solution to  \eqref{eq:macroleadfollstrong}}. First observe that Property 4, together with the Arzel\`a-Ascoli theorem, implies the existence of a subsequence (that we do not relabel) $\mu^{k}$ that uniformly converges to some $\mu^*$ with respect to the metric $\W_g$.

We are left to prove that $\mu^*$ is a solution to \eqref{eq:macroleadfollstrong}, in the sense of Definition \ref{def:solmacroleadfoll}. Since $\mu^{F,k}_0=\bar\mu^F$, by uniform convergence it holds $\mu^{F,*}_0=\bar\mu^F$, and the same holds for $\mu^{L,*}_0$. Then, Condition 1 of Definition \ref{def:solmacroleadfoll} is proved.

Condition 3 of uniform boundedness of the support comes from Property 3. Indeed, $\mu^{k}$ has uniformly bounded support in some $B(0,R)$ implies that $\mu^*$ has uniformly bounded support too, in $B(0,R+1)$. To prove this classical result, it is sufficient to test $\mu^*$ with functions having support outside $B(0,R+1)$.

We now prove Condition 2 of continuity with respect to the weak convergence of measures. It is a consequence of the fact that the sequence $\mu^{k}$ is equi-Lipschitz, thus $\mu^*$ is Lipschitz with respect to the distance $\W_g$, and such distance metrizes weak convergence on measures with equi-bounded support (Proposition \ref{p-gw}, statement 2).

We now prove Condition 4. We first prove a list of auxiliary estimates. Take a function $\varphi$ with extra regularity, namely $\varphi\in \mathcal{C}^2_c(\R^d)$, and fix $t\in[0,T]$. For each $k$ in the subsequence $\mu^k\to\mu^*$, choose $n$ as the largest integer satisfying $n\Delta t\leq t$. Thus, $t-n\Delta t\geq 0$. We have the following estimates:
\begin{description}
\item[Estimate 1.] Take $m_1:=\|\varphi\|_{\mathcal{C}^1}=\|\varphi\|_{\mathcal{C}^0}+\mathrm{Lip}(\varphi).$
It then holds
\begin{eqnarray*}
\left| \int_{\R^d}\varphi\,d(\mu^{F,*}_t-\mu^{F,k}_t)\right|\leq m_1\W_g(\mu^{F,*}_t,\mu^{F,k}_t).%\label{ee1}
\end{eqnarray*}
This is a consequence of the Kantorovich-Rubinstein duality for the generalized Wasserstein distance, see Theorem \ref{t-flat}.
\item[Estimate 2.]  Define 
\begin{eqnarray}
v^*_t:=K^F\conv \mu^{F,*}_t+K^L\conv \mu^{K,*}_t.\label{e-vstar}
\end{eqnarray}
It exists $m_2$, independent on $t,k,n$, such that it holds
\begin{eqnarray}
&&\left| \int_{\R^d}\nabla\varphi(x)\cdot v^{k}_{n\Delta t} d\Phi^{v^k_{n\Delta t}}_{t-n\Delta t}\#\mu^{F,k}_{n\Delta t}-\int_{\R^d}\nabla\varphi(x)\cdot v^{*}_{t} d\mu^{*,k}_{t}\right|\leq \nonumber\\
&&m_2(\W_g(\mu^{F,*}_t,\mu^{F,k}_t)+(t-n\Delta t)).\label{ee2}
\end{eqnarray}
Indeed, we first observe that \eqref{e-stimavkn}-\eqref{e-globsupp} imply 
\begin{eqnarray}
\|v^k_{n\Delta t}\|_{\mathcal{C}^0}\leq B_K(1+(2e^{TC}(R_0+1)))(|\mu^{F}_{0}|+ |\mu^{L}_{0}|).\label{e-vknc0}
\end{eqnarray}
Second, recall that $v^k_n$ in \eqref{e-schemev} is defined as a convolution. The, Lipschitz continuity of $K^F,K^L$ given by \ref{ass:lipK}, implies equi-Lipschitz continuity of the $v^k_n$. Indeed, it holds
\begin{eqnarray}
 |(K^F\conv \mu)(x)-(K^F\conv\mu)(y)|&\leq&\int_{\R^d} |K^F(z-x)-K^F(z-y)|\,d\mu(z)\leq\nonumber\\
 && L_k|x-y|\,|\mu|.
 \label{e-Kequilip}
 \end{eqnarray}
Thus, equi-boundedness of masses (Property 2) implies equi-Lipschitz continuity.

Third, since $\varphi\in \mathcal{C}^2(\R^d)$, the family $\nabla\varphi\cdot v^{k}_{n\Delta t}$ is equi-bounded and equi-Lipschitz, i.e. $m'_2:=\sup_{n,k}\left\{\|\nabla\varphi\cdot v^{k}_{n\Delta t}\|_{\mathcal{C}^0},\mathrm{Lip}(\nabla\varphi\cdot v^{k}_{n\Delta t})\right\}$ is finite. Thus, Kantorovich-Rubinstein duality implies
\begin{eqnarray}
\left| \int_{\R^d}\nabla\varphi(x)\cdot v^{k}_{n\Delta t} d(\Phi^{v^k_{n\Delta t}}_{t-n\Delta t}\#\mu^{F,k}_{n\Delta t}- \mu^{F,*}_{t})\right|\leq m'_2\W_g(\Phi^{v^k_{n\Delta t}}_{t-n\Delta t}\#\mu^{F,k}_{n\Delta t},\mu^{F,*}_{t}).
\label{e-usain2}
\end{eqnarray}
We apply the triangular inequality to have
\begin{eqnarray}
\W_g(\Phi^{v^k_{n\Delta t}}_{t-n\Delta t}\#\mu^{F,k}_{n\Delta t},\mu^{F,*}_{t})\leq
\W_g(\Phi^{v^k_{n\Delta t}}_{t-n\Delta t}\#\mu^{F,k}_{n\Delta t},\mu^{F,k}_{t})+
\W_g(\mu^{F,k}_{t},\mu^{F,*}_{t}).
\label{e-usain21}
\end{eqnarray}
For the first term, recall the definition of \eqref{e-schemeFt} and apply the Kantorovich-Rubinstein duality, observing that it holds
\begin{eqnarray*}
&&\int_{\R^d} f\, d(\Phi^{v^k_{n\Delta t}}_{t-n\Delta t}\#\mu^{F,k}_{n\Delta t}-\mu^{F,k}_{t})=\nonumber\\
&&\int_{\R^d} (-f)\,d
\Phi^{v^k_{n\Delta t}}_{t-n\Delta t}\#\left((t-n\Delta t) (-\alpha_F(\mu^{F,k}_{n\Delta t},\mu^{L,k}_{n\Delta t})\mu^{F,k}_{n\Delta t}+\alpha_L(\mu^{F,k}_{n\Delta t},\mu^{L,k}_{n\Delta t})\mu^{L,k}_{n\Delta t})\right)\nonumber\\
&&\leq \|f\|_{\mathcal{C}^0}(t-n\Delta t) \left| -\alpha_F(\mu^{F,k}_{n\Delta t},\mu^{L,k}_{n\Delta t})\mu^{F,k}_{n\Delta t}+\alpha_L(\mu^{F,k}_{n\Delta t},\mu^{L,k}_{n\Delta t})\mu^{L,k}_{n\Delta t})\right|\leq\nonumber\\
&& 1 (t-n\Delta t) M_\alpha (|\mu^{F,k}_{n\Delta t}|+|\mu^{L,k}_{n\Delta t}|)=(t-n\Delta t) M_\alpha(|\mu_0^F|+|\mu_0^L|).
\end{eqnarray*}
We use the fact that push-forward conserves the mass, hypothesis \ref{ass:supalpha} and Property 2. Since such estimate is independent on $f$ satisfying $\|f\|_{\mathcal{C}^0}\leq 1$, this gives 
\begin{eqnarray*}
\W_g(\Phi^{v^k_{n\Delta t}}_{t-n\Delta t}\#\mu^{F,k}_{n\Delta t},\mu^{F,k}_{t})\leq (t-n\Delta t) M_\alpha(|\mu_0^F|+|\mu_0^L|).
\end{eqnarray*}
Merging it with \eqref{e-usain2}-\eqref{e-usain21}, we have \eqref{ee2}.

\item[Estimate 3.] Define $$s^k_{n\Delta t}:=-\alpha_F(\mu^{F,k}_{n\Delta t},\mu^{L,k}_{n\Delta t})\mu^{F,k}_{n\Delta t}+\alpha_L(\mu^{F,k}_{n\Delta t},\mu^{L,k}_{n\Delta t})\mu^{L,k}_{n\Delta t},$$
and similarly 
$$s^*_t:=-\alpha_F(\mu^{F,*}_{t},\mu^{L,*}_{t})\mu^{F,*}_{t}+\alpha_L(\mu^{F,*}_{t},\mu^{L,*}_{t})\mu^{L,*}_{t}.$$
It exists $m_3$, independent on $t,k,n$, such that it holds
\begin{eqnarray}
\left|\int_{\R^d}\varphi \,d\left(\Phi^{v^k_{n\Delta t}}_{t-n\Delta t}\#s^k_{n\Delta t}-s^*_t\right)\right|\leq m_3(\W_g(\mu^{*}_t,\mu^{k}_t)+(t-n\Delta t)).\label{ee3}
\end{eqnarray}
Indeed, we first consider the negative parts of the measures $\Phi^{v^k_{n\Delta t}}_{t-n\Delta t}\#s^k_{n\Delta t}$ and $s^*_t$. They satisfy
\begin{eqnarray*}
&&C_-:=\left|\int_{\R^d}\varphi(x)\,d(\Phi^{v^k_{n\Delta t}}_{t-n\Delta t}\#(\alpha_F(\mu^{F,k}_{n\Delta t},\mu^{L,k}_{n\Delta t})\mu^{F,k}_{n\Delta t})-\alpha_F(\mu^{F,*}_{t},\mu^{L,*}_{t})\mu^{F,*}_{t})\right|\leq\nonumber\\
&&m_1 \W_g(\Phi^{v^k_{n\Delta t}}_{t-n\Delta t}\#(\alpha_F(\mu^{F,k}_{n\Delta t},\mu^{L,k}_{n\Delta t})\mu^{F,k}_{n\Delta t}),\alpha_F(\mu^{F,k}_{n\Delta t},\mu^{L,k}_{n\Delta t})\mu^{F,k}_{n\Delta t})+\nonumber\\
&&m_1\W_g(\alpha_F(\mu^{F,k}_{n\Delta t},\mu^{L,k}_{n\Delta t})\mu^{F,k}_{n\Delta t},\alpha_F(\mu^{F,*}_{t},\mu^{L,*}_{t})\mu^{F,k}_{n\Delta t})+\nonumber\\
&&m_1\W_g(\alpha_F(\mu^{F,*}_{t},\mu^{L,*}_{t})\mu^{F,k}_{n\Delta t},\alpha_F(\mu^{F,*}_{t},\mu^{L,*}_{t})\mu^{F,*}_{t}).
\end{eqnarray*}
where we used the definition of $m_1$ in Estimate 1, the Kantorovich-Rubinstein duality and the triangular inequality. For the first term, use \eqref{e-gw1} together with the estimate \eqref{e-vknc0} for $\|v^k_{n\Delta t}\|_{\mathcal{C}^0}$, as well as \ref{ass:supalpha}. For the second and third terms, use \ref{ass:lipalpha}: since \eqref{e-massconstraint}-\eqref{e-boundedsupp} hold, then \eqref{e-gwLip-prel} holds with some $L_{\alpha,M,R}$. Moreover, use \eqref{e-diff} for the second term and \eqref{e-omo} for the third one. It then holds
\begin{eqnarray*}
&&C_-\leq m_1(t-n\Delta t) \|v^k_{n\Delta t}\|_{\mathcal{C}^0} |\alpha_F(\mu^{F,k}_{n\Delta t},\mu^{L,k}_{n\Delta t})\mu^{F,k}_{n\Delta t}|+\nonumber\\
&&m_1L_{\alpha,M,R}\gw{\mu^{k}_{n\Delta t},\mu^{*}_{t}} +m_1 \alpha_F(\mu^{F,*}_{t},\mu^{L,*}_{t})\W_g(\mu^{F,k}_{n\Delta t},\mu^{F,*}_{t})\leq\\
&&m'_3 (t-n\Delta t) M_\alpha |\mu^{F,k}_{n\Delta t}|+m_1 L_{\alpha,M,R}  \W_g(\mu^{k}_{n\Delta t},\mu^{*}_{t})+m_1M_\alpha \W_g(\mu^{F,k}_{n\Delta t},\mu^{F,*}_{t}),
\end{eqnarray*}
for some $m'_3$ independent on $t,k,n$. Recall that masses are equi-bounded (Property 2). Also apply the triangular inequality and uniform Lipschitz continuity of the $\mu^k$ (Property 4) to write
\begin{eqnarray*}
\W_g(\mu^{F,k}_{n\Delta t},\mu^{F,*}_{t})\leq \W_g(\mu^{F,k}_{n\Delta t},\mu^{F,k}_{t})+\W_g(\mu^{F,k}_{t},\mu^{F,*}_{t})\leq L' |t-n\Delta t|+\W_g(\mu^{F,k}_{t},\mu^{F,*}_{t}),
\end{eqnarray*}
for some $L'$, and similarly for $\W_g(\mu^{L,k}_{n\Delta t},\mu^{L,*}_{t})$. It then exists $m''_3$ such that 
$$C_-\leq m''_3((t-n\Delta t)+\W_g(\mu^{F,k}_{t},\mu^{F,*}_{t})+\W_g(\mu^{L,k}_{ t},\mu^{L,*}_{t})).$$
An equivalent estimate holds for the positive parts of the measures $\Phi^{v^k_{n\Delta t}}_{t-n\Delta t}\#s^k_{n\Delta t}$ and $s^*_t$. We then recover \eqref{ee3}.

\item[Estimate 4.] There exists $m_4$ independent on $t,k,n$ such that it holds
\begin{eqnarray}
\left|\int_{\R^d}\nabla\varphi(x)\cdot v^{k}_{n\Delta t}\, d\Phi^{v^k_{n\Delta t}}_{t-n\Delta t}\#s^k_{n\Delta t}\right|\leq m_4
\label{ee4}
\end{eqnarray}
Indeed, first recall that $\|v^k_{n\Delta t}\|_{\mathcal{C}^0}$ is uniformly bounded on the support of $\mu^k_{n\Delta t}$. Moreover, 
$$\left| \Phi^{v^k_{n\Delta t}}_{t-n\Delta t}\#s^k_{n\Delta t}\right|=|s^k_{n\Delta t}|=|(s^k_{n\Delta t})^+|+|(s^k_{n\Delta t})^-|$$ is uniformly bounded, as a consequence of \ref{ass:supalpha} and of uniform boundedness of masses (Property 2). This proves \eqref{ee4}.

\item[Estimate 5.] We now prove that $\mu^k_t$ solves an approximated version of \eqref{eq:macroleadfollstrong}. By the definition \eqref{e-schemeFt} of $\mu^{F,k}_t$, and applying elementary properties of derivation as well as Lemma \ref{le:equivPDEODEtransp}, it holds
\begin{eqnarray}
&&\frac{d}{dt}\int_{\R^d}\varphi\,d\mu^{F,k}_t=\frac{d}{dt}\int_{\R^d}\varphi\,d\Phi^{v^k_{n\Delta t}}_{t-n\Delta t}\#\mu^{F,k}_{n\Delta t}+\frac{d}{dt}\int_{\R^d}\varphi\,d\Phi^{v^k_{n\Delta t}}_{t-n\Delta t}\#((t-n\Delta t)s^k_n)=\nonumber\\
&&\int_{\R^d}\nabla\varphi\cdot v^{k}_{n\Delta t} d\Phi^{v^k_{n\Delta t}}_{t-n\Delta t}\#\mu^{F,k}_{n\Delta t}+\frac{d}{dt}\left[(t-n\Delta t)\int_{\R^d}\varphi\,d\Phi^{v^k_{n\Delta t}}_{t-n\Delta t}\#s^k_n\right]=\nonumber\\
&&\int_{\R^d}\nabla\varphi\cdot v^{k}_{n\Delta t} d\Phi^{v^k_{n\Delta t}}_{t-n\Delta t}\#\mu^{F,k}_{n\Delta t}+
\int_{\R^d}\varphi\,d\Phi^{v^k_{n\Delta t}}_{t-n\Delta t}\#s^k_n+\nonumber\\
&&(t-n\Delta t)\int_{\R^d}\nabla\varphi\cdot v^{k}_{n\Delta t}\, d\Phi^{v^k_{n\Delta t}}_{t-n\Delta t}\#s^k_n.\label{e-eqnk}
\end{eqnarray}
for all $t\neq n\Delta t$. The equivalent estimate for $\mu^{L,k}_t$ holds too, by replacing $\mu^{F,k}_t$ with $\mu^{L,k}_t$ and $s^k_n$ with $-s^k_n$.

One can write \eqref{e-eqnk} in integral form too, as follows:
for each $t\in[0,T]$ and $k\in\mathbb{N}$, choose the largest $n$ \footnote{the dependence of $n$  on $t$ and $k$ is omitted for the sake of notation.} satisfying $n\Delta t=n 2^{-k}T\leq t$. For each $\bar t\in[0,T]$, it holds 
\begin{eqnarray}
&&\hspace{-5mm}\int_0^{\bar t} dt \int_{\R^d}\varphi\,d\mu^{F,k}_t- \left(\int_{\R^d}\nabla\varphi\cdot v^{k}_{n\Delta t} d\Phi^{v^k_{n\Delta t}}_{t-n\Delta t}\#\mu^{F,k}_{n\Delta t}+\right.\nonumber\\
&&\hspace{-5mm}\left.\int_{\R^d}\varphi\,d\Phi^{v^k_{n\Delta t}}_{t-n\Delta t}\#s^k_{n}+(t-n\Delta t)\int_{\R^d}\nabla\varphi\cdot v^{k}_{n\Delta t}\, d\Phi^{v^k_{n\Delta t}}_{t-n\Delta t}\#s^k_{n}.\right)=0.
\label{e-mukintform}
\end{eqnarray}

\end{description}

We are now ready to prove Condition 4, that we prove in the equivalent integral form: for every $\bar t\in[0,T]$, the measure $\mu^{F,*}$ satisfies
\begin{eqnarray}
\int_0^{\bar t} dt \left(\int_{\R^d}\varphi\,d\mu^{F,*}_t- \int_{\R^d}\nabla\varphi\cdot v^*_t\, d\mu^{F,*}_t -
		\int_{\R^d}\varphi\,d s^*_t\right)=0,\label{e-muintform}
\end{eqnarray}
and a similar expression holds for $\mu^{L,*}$.

Assume that $\varphi\in \mathcal{C}^2_c(\R^d)$. We then prove that \eqref{e-muintform} holds by writing
\begin{eqnarray*}
&&C^*:=\left|\int_0^{\bar t} dt \left(\int_{\R^d}\varphi\,d\mu^{F,*}_t- \int_{\R^d}\nabla\varphi\cdot v^*_t\, d\mu^{F,*}_t -		\int_{\R^d}\varphi\,d s^*_t\right)\right|\leq\\
&& \int_0^{\bar t} dt \left(\left|\int_{\R^d}\varphi\,d(\mu^{F,*}_t-\mu^{F,k}_t)\right|+\left|\int_{\R^d}\nabla\varphi\cdot v^*_t\, d\mu^{F,*}_t-\int_{\R^d}\nabla\varphi\cdot v^{k}_{n\Delta t} d\Phi^{v^k_{n\Delta t}}_{t-n\Delta t}\#\mu^{F,k}_{n\Delta t}\right|\right.\\
&&+\left|\int_{\R^d}\varphi\,d (s^*_t-\Phi^{v^k_{n\Delta t}}_{t-n\Delta t}\#s^k_n)\right|+\left|\int_{\R^d}\varphi\,d\mu^{F,k}_t- \int_{\R^d}\nabla\varphi\cdot v^{k}_{n\Delta t} d\Phi^{v^k_{n\Delta t}}_{t-n\Delta t}\#\mu^{F,k}_{n\Delta t}-\right.\nonumber\\
&&\left.\int_{\R^d}\varphi\,d\Phi^{v^k_{n\Delta t}}_{t-n\Delta t}\#s^k_n\right|\leq 
\int_0^{\bar t} dt\, \left((m_1+m_2) \W_g(\mu^{F,*}_t,\mu^{F,k}_t)+m_3\W_g(\mu^{*}_t,\mu^{k}_t)+\right.\\
&&\left.(m_2+m_3)(t-n\Delta t)+\left|(t-n\Delta t)\int_{\R^d}\nabla\varphi\cdot v^{k}_{n\Delta t}\, d\Phi^{v^k_{n\Delta t}}_{t-n\Delta t}\#s^k_n\right|\right).
\end{eqnarray*}
We used here Estimates 1, 2, 3,  as well as Estimate 5 in its integral form \eqref{e-mukintform}. Recall now the definition of $\W_g(\mu^{*},\mu^{k})$ in \eqref{e-distsup} and use Estimate 4 for the last term. Also observe that it holds $t-n\Delta t\leq \Delta t= T 2^{-k}$ by the choice of $n$. By defining $m:=m_1+m_2+m_3+m_4$, it holds
\begin{eqnarray*}
&&C^*\leq \bar t (m \W_g(\mu^{*},\mu^{k})+mT2^{-k} ).
\end{eqnarray*}
Since such estimate holds for any $k$ in the converging subsequence, it holds $C^*=0$.

We have then proved that \eqref{e-muintform} is satisfied for any $\varphi\in \mathcal{C}^2_c(\R^d)$. Since for any $\mu^F,s^*$ the three operators $\varphi\to \int_{\R^d}\varphi\,d\mu^F, \int_{\R^d}\nabla\varphi\cdot v^*\, d\mu^{F},	\int_{\R^d}\varphi\,d s^*$ are continuous with respect to the norm $\mathcal{C}^1$, and $\mathcal{C}^2_c(\R^d)$ is dense in $\mathcal{C}^1_c(\R^d)$ with respect to such norm, then  \eqref{e-muintform} is satisfied for any $\varphi\in \mathcal{C}^1_c(\R^d)$.
\end{proof}

We now prove existence and uniqueness of the solution to \eqref{eq:macroleadfollstrong}.
\begin{proposition} \label{p-exun1}
Let an initial data $(\mu^F_0,\mu^L_0)\in \mathcal{M}_c(\R^d)\times \mathcal{M}_c(\R^d)$ and a time interval $[0,T]$ be fixed. Let \ref{ass:lipK}-\ref{ass:sublinK}-\ref{ass:supalpha}-\ref{ass:lipalpha} hold. Then, there exists a unique solution to \eqref{eq:macroleadfollstrong}.
\end{proposition}
\begin{proof} Existence of a solution was proved in Propostion \ref{p-ex}. We now prove uniqueness.

Let $\mu,\nu$ be two solutions of \eqref{eq:macroleadfollstrong}, in the sense of Definition \ref{def:solmacroleadfoll}. They are both continuous with respect to the topology of weak convergence of measures (Condition 2) and have equi-bounded support (Condition 3). By choosing $\varphi \in \mathcal{C}^1_c(\R^d)$ satisfying $\varphi\equiv 1$ on such equi-bounded support and using \ref{h4}, it holds
\begin{eqnarray*}
\partial_t |\mu^F_t|\leq 0 + M_\alpha (|\mu^F_t|+|\mu^L_t|),
\end{eqnarray*}
and similarly for $|\mu^L_t|$. This implies $|\mu^F_t|+|\mu^L_t|\leq e^{2M_\alpha t}(|\mu^F_0|+|\mu^L_0|)$, hence masses are equi-bounded too.

For the given solution $\mu_t$, define the corresponding vector field and source term $$w_t:=\sum_{j\in\{F,L\}}K^{j}\conv \mu^j_t,\qquad s_t:=-\alpha_F(\mu^F_t,\mu^L_t) \mu^F_t+\alpha_L(\mu^F_t,\mu^L_t)\mu^L_t.$$ Consider them as time-varying operators, not depending on $\mu$. 
By construction, it holds 
\begin{eqnarray}
\partial_t \mu^F_t=-\mathrm{div}(w_t\mu^F_t)+s_t.
\label{e-mut}
\end{eqnarray} Observe that $w_t$ is a time-varying vector field, continuous with respect to the time variable and uniformly Lipschitz with respect to the space variable, due to \ref{h3}, \eqref{e-Kequilip} and equi-boundedness of $|\mu_t|$. It is then a Carath{\'e}odory function. Moreover, $s_t$ is continuous with respect to time, with uniformly bounded mass due to \ref{h4} and with uniformly bounded support due to \ref{h5}. Then, hypotheses of Lemma \ref{le:equivPDEODEtransp} are satisfied, hence $\mu^F_t$ is the unique solution of \eqref{e-mut} and it satisfies the Duhamel's formula \eqref{e-duhamel}.
It is clear that the previous properties hold for $\mu^L$ too, with the same vector field $w_t$ and source $-s_t$. Moreover, the same properties hold for $\nu^F$ too, with vector field and source term
$$w'_t:=\sum_{j\in\{F,L\}}K^{j}\conv \nu^j_t,\qquad s'_t:=-\alpha_F(\nu^F_t,\nu^L_t) \nu^F_t+\alpha_L(\nu^F_t,\nu^L_t)\nu^L_t,$$
as well as for $\nu^L_t$, with $w'_t$ and $-s'_t$.

We now compute $\W_g(\mu_t,\nu_t)$ by using the Duhamel's formula and the Kantorovich-Rubinstein duality. Take $f$ such that $\|f\|_{\mathcal{C}^0},\mathrm{Lip}(f)\leq 1$ and compute
\begin{eqnarray}
&&\int_{\R^d} f\,d(\mu^F_t-\nu^F_t)= \int_{\R^d} f\,d(\Phi^{w_t}_t\#\mu^F_0-\Phi^{w'_t}_t\#\nu^F_0)+\nonumber\\
&&\int_0^t d\tau \int_{\R^d} f\,d({\Phi^{w}_{(\tau,t)}}\#s_\tau-{\Phi^{w'}_{(\tau,t)}}\#s'_\tau)\leq \gw{\Phi^{w_t}_t\#\mu^F_0,\Phi^{w'_t}_t\#\nu^F_0}+\nonumber\\
&&\int_0^t d\tau  \int_{\R^d} f\,d\left(-{\Phi^{w}_{(\tau,t)}}\#\alpha_F(\mu^F_\tau,\mu^L_\tau) \mu^F_\tau+{\Phi^{w'}_{(\tau,t)}}\#\alpha_F(\nu^F_\tau,\nu^L_\tau) \nu^F_\tau\right)+\nonumber\\
&&\int_0^t d\tau  \int_{\R^d} f\,d\left({\Phi^{w}_{(\tau,t)}}\#\alpha_L(\mu^F_\tau,\mu^L_\tau) \mu^L_\tau-{\Phi^{w'}_{(\tau,t)}}\#\alpha_L(\nu^F_\tau,\nu^L_\tau) \nu^L_\tau\right)\leq\nonumber\\
&&e^{2Lt}\gw{\mu^F_0,\nu^F_0}+|\mu^F_0|\frac{e^{2L t}(e^{Lt}-1)}{L} \sup_{\tau\in[0,t]}\{\|w_\tau-w'_\tau\|_{\mathcal{C}^0}\}+\label{e-dmunu}\\
&&\int_0^t d\tau e^{2L(t-\tau)}\left(\gw{\alpha_F(\mu^F_\tau,\mu^L_\tau) \mu^F_\tau,\alpha_F(\nu^F_\tau,\nu^L_\tau) \nu^F_\tau}+\right.\nonumber\\
&& \hspace{2cm}\left. \gw{\alpha_L(\mu^F_\tau,\mu^L_\tau) \mu^L_\tau,\alpha_L(\nu^F_\tau,\nu^L_\tau) \nu^L_\tau}\right)+\nonumber\\
&&\int_0^t d\tau(|\alpha_F(\mu^F_0,\mu^L_0) \mu^F_0|+|\alpha_L(\mu^F_0,\mu^L_0) \mu^L_0|)\frac{e^{2L (t-\tau)}(e^{L(t-\tau)}-1)}{L} \sup_{\tau'\in[\tau,t]}\{\|w_{\tau'}-w'_{\tau'}\|_{\mathcal{C}^0}\},\nonumber
\end{eqnarray}
where $L$ is a Lipschitz constant for both $w_\tau,w'_\tau$, that exists by \eqref{e-Kequilip}, and where we also used \eqref{e-gwLip}. Observe that it holds
\begin{eqnarray}
 |(K^F\conv \mu^F_t)(x)-(K^F\conv\nu^F_t)(x)|&\leq&\left|\int_{\R^d} K^F(z-x)\,d(\mu^F_t(z)-\nu^F_t(z))\right|\leq\nonumber\\
 && L_K\gw{\mu^F_t,\nu^F_t}
 \label{e-Kmunu}
 \end{eqnarray}
where we used the Kantorovich-Rubinstein duality and \ref{h2}. The same estimate holds for $K^L$, thus $|w_t(x)-w'_t(x)|\leq L_K\gw{\mu_t,\nu_t}$. 

Going back to \eqref{e-dmunu}, recall that $\gw{\mu_0,\nu_0}=0$, since the initial data coincide. Also apply the estimate \eqref{e-gwLip} and hypothesis \ref{h4}. Define $$\varepsilon(t):=\sup_{\tau\in[0,t]}\gw{\mu_\tau,\nu_\tau}$$ and observe that it holds
 \begin{eqnarray*}
 &&\int_{\R^d} f\,d(\mu^F_t-\nu^F_t)\leq 0+ |\mu^F_0|\frac{e^{2L t}(e^{Lt}-1)}{L} L_K \varepsilon(t)+
 \frac{e^{2Lt}-1}{L} L'_{\alpha,M,R} \varepsilon(t)+\nonumber\\
 &&M_\alpha(|\mu^F_0|+|\mu^L_0|)\frac{(e^{2L t}-1)(e^{Lt}-1)}{L}L_K\varepsilon(t)
 \end{eqnarray*}
 Since the left hand side does not depend on $f$, one can take the supremum over $f$ satisfying $\|f\|_{\mathcal{C}^0},\mathrm{Lip}(f)\leq 1$, i.e. replace it with $\gw{\mu^F_t,\nu^F_t}$. The equivalent estimate holds for $\gw{\mu^L_t,\nu^L_t}$. Merging them, it holds 
 \begin{eqnarray}
 \gw{\mu_t,\nu_t}\leq C_t  \varepsilon(t),
 \label{e-utilequi}
 \end{eqnarray} with 
 $$C_t:= (|\mu^F_0|+|\mu^L_0|)\frac{(e^{Lt}-1)}{L} L_K \left(e^{2Lt}+2M_\alpha (e^{2L t}-1)\right)+
 2\frac{e^{2Lt}-1}{L} L'_{\alpha,M,R}.$$
 
 Since the right hand side in \eqref{e-utilequi} is an increasing function with respect to $t$, one can replace $ \gw{\mu_t,\nu_t}$ with $\varepsilon(t)$ on the left hand side.  It then holds
 $$\varepsilon(t)\leq C_t \varepsilon(t).$$
 Since $\lim_{t\to 0} C_t=0$ and $C_t$ is continuous, it holds $\varepsilon(t)=0$ for $t$ sufficiently small. By iterating the estimate, this holds for any $t\in[0,T]$, thus $\mu_t=\nu_t$ for all $t\in[0,T]$.
\end{proof}

\subsection{Equivalence between systems \eqref{eq:macroleadfollstrong} and \eqref{eq:macronusigma}}

We now prove that, if \ref{ass:posmeas0} is satisfied, then systems \eqref{eq:macroleadfollstrong} and \eqref{eq:macronusigma} are equivalent, in the sense that there exists a bijection between solutions. We also use this equivalence to prove existence and uniqueness of solutions to system \eqref{eq:macronusigma}.
\begin{proposition} \label{p-equivalenza}
 Let $(\mu^F_t,\mu^L_t)$ be a solution to system \eqref{eq:macroleadfollstrong}, such that $(\mu^F_0,\mu^L_0)$ satisfies \ref{ass:posmeas0}. Assume that hypotheses \ref{ass:lipK}-\ref{ass:sublinK}-\ref{ass:supalpha}-\ref{ass:lipalpha} hold. Define
\begin{equation}
\label{e-defnusigma}
\nu_t:=\mu^F_t+\mu^L_t,\qquad (\sigma_t(F),\sigma_t(L)):=\left(\frac{|\mu^F_t|}{|\nu_t|},\frac{|\mu^L_t|}{|\nu_t|}\right).
\end{equation}
Then, $(\nu_t,\sigma_t)$ is a solution to system \eqref{eq:macronusigma}.

Conversely, let the hypotheses \ref{ass:lipK}-\ref{ass:sublinK}-\ref{ass:supalpha}-\ref{ass:lipalpha} hold and let $(\nu_t,\sigma_t)$ be a solution to system \eqref{eq:macronusigma}. Define 
\begin{equation}
\mu^F_t=\sigma_t(F)\nu_t,\qquad \mu^L_t=\sigma_t(L)\nu_t.\label{e-defmu}
\end{equation}
Then, $(\mu^F_t,\mu^L_t)$ is a solution to system \eqref{eq:macroleadfollstrong}.
\end{proposition}
\begin{proof} We prove {\bf Statement 1}.Take $(\mu^F_t,\mu^L_t)$ a solution to system \eqref{eq:macroleadfollstrong} with $(\mu^F_0,\mu^L_0)$ satisfying \ref{ass:posmeas0}. Define $(\nu_t,\sigma_t)$ according to \eqref{e-defnusigma}. By a direct computation, it holds 
\begin{equation}
\partial_t\nu_t=-\diver((K^{F}\conv\mu^F_t + K^{L}\conv\mu^L_t)\nu_t).
\label{e-nut}
\end{equation}
This also implies that $|\nu_t|$ is constant. Define now $\sigma_t$ according to \eqref{e-defnusigma}, and compute
\begin{eqnarray}
\partial_t \sigma_t(F)&=&\frac{\partial_t|\mu^F_t|}{|\nu_t|}=\frac{- \alpha_F(\mu^F_t,\mu^L_t)|\mu^F_t| + \alpha_L(\mu^F_t,\mu^L_t)|\mu^L_t|}{|\nu_t|}=\nonumber\\
&=&- \alpha_F(\mu^F_t,\mu^L_t)\sigma_t(F) + \alpha_L(\mu^F_t,\mu^L_t)\sigma_t(L).\label{e-sigmat}
\end{eqnarray}
We used the fact that $|\nu_t|$ is constant, as a consequence of \eqref{e-nut}, and the definition of $\sigma_t$. One easily recovers $\sigma_t(L)=1-\sigma_t(F)$, hence $\partial_t\sigma_t(L)=-\partial_t\sigma_t(F)$. The difficulty is now to prove that it holds $\mu^F_t=\sigma_t(F)\nu_t, \mu^L_t:=\sigma_t(L)\nu_t$ for all times.

Since $\mu^F_t,\mu^L_t$ are given, one can define the non-autonomous vector field and the coefficients for the source term $$v_t:=K^{F}\conv\mu^F_t + K^{L}\conv\mu^L_t,\qquad h^F_t:=\alpha_F(\mu^F_t,\mu^L_t),\qquad h^L_t:=\alpha_L(\mu^F_t,\mu^L_t).$$ Define $$\tilde \mu^F_t:=\sigma_t(F)\nu_t,\qquad \tilde \mu^L_t:=\sigma_t(L)\nu_t$$ Observe that it holds $\tilde \mu^F_0=\mu^F_0$ and $\tilde \mu^L_0=\mu^L_0$, as a consequence of \ref{ass:posmeas0}. Using \eqref{e-nut}-\eqref{e-sigmat}, it holds 
\begin{eqnarray*}
\partial_t\tilde \mu^F_t&=&-\diver((K^{F}\conv\mu^F_t + K^{L}\conv\mu^L_t)\sigma_t(F)\nu_t)- \alpha_F(\mu^F_t,\mu^L_t)\sigma_t(F)\nu_t +\\
&& \alpha_L(\mu^F_t,\mu^L_t)\sigma_t(L)\nu_t=-\diver(v_t \tilde \mu^F_t)- h^F_t \tilde \mu^F_t + h^L_t \tilde \mu^L_t.
\end{eqnarray*}
One similarly has $\partial_t\tilde \mu^L_t=-\diver(v_t \tilde \mu^L_t)+ h^F_t \tilde \mu^F_t - h^L_t \tilde \mu^L_t$. By construction, it also holds
$$\partial_t  \mu^F_t=-\diver(v_t   \mu^F_t)- h^F_t   \mu^F_t + h^L_t   \mu^L_t,\qquad
\partial_t  \mu^L_t=-\diver(v_t   \mu^L_t)+ h^F_t   \mu^F_t - h^L_t   \mu^L_t.$$

Take $f$ such that $\|f\|_{\mathcal{C}^0},\mathrm{Lip}(f)\leq 1$, and apply the Duhamel's formula for both $\mu_t,\tilde\mu_t$. It holds
\begin{eqnarray}
&&\int_{\R^d} f\,d(\mu^F_t-\tilde\mu^F_t)= \int_{\R^d} f\,d(\Phi^{v_t}_t\#\mu^F_0-\Phi^{v_t}_t\#\tilde\mu^F_0)+\label{e-lhs}\\
&&\int_0^t d\tau \left[h^F_\tau \int_{\R^d} f\,d(\Phi^{v_t}_{(\tau,t)}\#\tilde\mu^F_\tau-\Phi^{v_t}_{(\tau,t)}\#\mu^F_\tau)+h^L_\tau\int_{\R^d} f\,d(\Phi^{v_t}_{(\tau,t)}\#\mu^L_\tau-\Phi^{v_t}_{(\tau,t)}\#\tilde\mu^L_\tau)\right] \leq\nonumber\\
&&0+\int_0^t d\tau (h^F_\tau\gw{\Phi^{v_t}_{(\tau,t)}\#\mu^F_\tau,\Phi^{v_t}_{(\tau,t)}\#\tilde\mu^F_\tau}+h^L_\tau\gw{\Phi^{v_t}_{(\tau,t)}\#\mu^L_\tau,\Phi^{v_t}_{(\tau,t)}\#\tilde\mu^L_\tau}.\label{e-rhs}
\end{eqnarray}
Here we used the fact that $\mu^F_0=\tilde\mu^F_0$ implies $\Phi^{v_t}_t\#\mu^F_0=\Phi^{v_t}_t\#\tilde\mu^F_0$, as well as the Kantorovich-Rubinstein duality. Denote with $L$ a Lipschitz constant for $v_t$, that exists due to \eqref{e-Kequilip}, and apply \eqref{e-gw2}.  Observe that \eqref{e-rhs} does not depend on $f$, thus one can take the supremum in the left hand side of \eqref{e-lhs} with $\|f\|_{\mathcal{C}^0},\mathrm{Lip}(f)\leq 1$, i.e. replace it with $\gw{\mu^F_\tau,\tilde\mu^F_\tau}$. Also observe that \ref{h4} implies $|h^F_\tau|,|h^L_\tau|\leq M_\alpha$. By defining $\varepsilon(t):=\sup_{\tau\in[0,t]} \gw{\mu_\tau,\tilde\mu_\tau}$, it holds
$$\gw{\mu^F_\tau,\tilde\mu^F_\tau}\leq t e^{2Lt}M_\alpha \varepsilon(t),$$ and the same holds for $\gw{\mu^L_\tau,\tilde\mu^L_\tau}$. 

Observe that the right hand side is increasing with respect to $t$, thus one can replace the left hand side with $\varepsilon(t)$. It then holds $\varepsilon(t)\leq 2t e^{2Lt}M_\alpha \varepsilon(t)$, thus $\varepsilon(t)=0$ for $t$ sufficiently small. Applying then the result iteratively, it holds $\varepsilon(t)=0$ for all $t\in[0,T]$, then $\mu_t=\tilde\mu_t$, thus $$\mu^F_t=\sigma_t(F)\nu_t \qquad \text{ and } \qquad \mu^L_t=\sigma_t(L)\nu_t.$$

We prove {\bf Statement 2}. Since $(\nu_t,\sigma_t)$ is a solution to system \eqref{eq:macronusigma} in the sense of Definition \ref{def:solnusigma}, then $(\mu^F_t,\mu^L_t)$ defined by \eqref{e-defmu} satisfies Conditions 2 and 3 of Definition \ref{def:solmacroleadfoll}. Condition 1 is also satisfied, by trivially choosing $\bar \mu^F=\mu^F_0$ and $\bar\mu^L=\mu^L_0$. We are left to prove that Condition 4 is satisfied: the proof is direct, by computing derivatives.

\end{proof}

As a corollary to Proposition \ref{p-equivalenza}, we prove existence and uniqueness of solutions to system \eqref{eq:macronusigma}.
\begin{corollary}\label{cor:nusigmaexistence} Let the hypotheses \ref{ass:lipK}-\ref{ass:sublinK}-\ref{ass:supalpha}-\ref{ass:lipalpha} hold. Then, for each initial data $(\overline{\nu},\overline{\sigma})\in\mathcal{M}(\R^d)\times\mathcal{P}(\{F,L\})$, there exists a unique solution to system \eqref{eq:macronusigma}.
\end{corollary}
\begin{proof} For the existence part, define $$\bar\mu^F:=\overline{\sigma}(F)\overline{\nu},\qquad \bar\mu^L:=\overline{\sigma}(F)\overline{\nu}$$ and consider the corresponding solution $(\mu^F_t,\mu^L_t)$ to \eqref{eq:macroleadfollstrong}, that exists due to Proposition \ref{p-exun1}. Then, there exists a corresponding solution $(\nu_t,\sigma_t)$ to system \eqref{eq:macronusigma}, due to the first statement of Proposition \ref{p-equivalenza}. Such solution satisfies $(\nu_0,\sigma_0)=(\overline{\nu},\overline{\sigma})$, by construction.

For the uniqueness part, assume that there exist two solutions $(\nu_t,\sigma_t),(\tilde\nu_t,\tilde\sigma_t)$ to \eqref{eq:macronusigma} with the same initial data $(\overline{\nu},\overline{\sigma})$. Due to the second statement of Proposition \ref{p-equivalenza}, for each of the two solutions to \eqref{eq:macronusigma} there exists a solution $(\mu^F_t,\mu^L_t),(\tilde\mu^F_t,\tilde\mu^L_t)$ to system \eqref{eq:macroleadfollstrong}. It clearly holds $(\mu^F_0,\mu^L_0)=(\tilde\mu^F_0,\tilde\mu^L_0)$, then such two solutions coincide, due to uniqueness of the solution to system \eqref{eq:macroleadfollstrong}. Since the relation \eqref{e-defmu} is invertible, this implies $(\nu_t,\sigma_t)=(\tilde\nu_t,\tilde\sigma_t)$.
\end{proof}

\begin{remark}\label{rem:servepersteer}
By inspection of our proofs, other types of measure-dependent velocity fields can be encompassed by our approach, as long as the dependence is Lipschitz with repect to $\mathcal{W}_g$  (see, e.g., \cite{gw}). For instance, instead of the convolution term 
$
K^i * \mu^i
$
for $i=F$ or $i=L$, one could simply consider a weighted velocity of the form
$
|\mu^i|K^i(x)
$
which still allows for proving the existence, uniqueness and equivalence results of this section. 
Accordingly, in the equivalent system to be considered in Proposition \ref{p-equivalenza} one has to consider a velocity field of the form (if $i=L$ in the equation above)
\[
\sigma_t(F) K^F\conv \nu_t + \sigma_t(L) K^L
\]
for which also the mean-field derivation of Section \ref{sec:meanfield} can be performed without changing the proofs.
\end{remark}

\section{A mean-field description of the leader-follower dynamics}\label{sec:meanfield}

In this section we shall provide a mean-field description of system \eqref{eq:macronusigma}. To do this, we shall first introduce for every $N\in\N$ a particle system which consists of a transport part for the evolution over the state space $\R^d$ and a jump part for the change of label in $\{F,L\}$.

The connection between systems of interacting particles and nonlinear evolution equations has been studied by many authors, going back to McKean \cite{mckean1967propagation}; for detailed expositions on this topic, the reader may consult Sznitman \cite{snitzman} or M\'el\'eard \cite{meleard}. A central point of this connection is the introduction of a nonlinear \textit{averaged} particle system associated with the original one, whose marginal laws appear explicitly (and nonlinearly) in the generator of its dynamics. When the interactions are regular and the particles are exchangeable, a unique nonlinear process exists, and it describes the limiting behavior of one particle of the original system when their number tends to infinity. One further has the propagation of chaos property, which is in this case equivalent to a trajectorial law of large numbers and yields the final mean-field limit result.

\subsection{Definition of the stochastic processes}

Throughout this section, we shall fix $\overline{\nu}\in\mathcal{P}_1(\R^d)$ and $\overline{\sigma}\in\mathcal{P}_1(\{F,L\})$, as well as, for every $N > 0$, a sample of $N$ particles from $\overline{\nu}\times\overline{\sigma}$, i.e.,
$$(X^{i,N}_0,\state^{i,N}_0)_{i = 1}^N \sim \overline{\nu}\times\overline{\sigma}.$$ We assume that $\overline{\nu}$ has compact support in $\R^d$.

We introduce the stochastic processes $(X^{1,N}_t,\state^{1,N}_t),\ldots,(X^{N,N}_t,\state^{N,N}_t)$ defined for every $t \in [0,T]$ and $i = 1,\ldots,N$ as follows
	\begin{itemize}
		\item the initial conditions are $(X^{i,N}_0,\state^{i,N}_0)_{i = 1}^N$,
		\item we set
		\begin{align}\label{eq:empmean}
		\nu^N_t \defin \frac1N \sum^N_{i = 1}\delta_{X^{i,N}_t} \quad \text{ and } \quad \sigma^N_t \defin \frac1N \sum^N_{i = 1}\delta_{\state^{i,N}_t},
		\end{align}
		\item it holds
                \begin{align}\label{stoc-X}
               dX^{i,N}_t = \langle K, \nu^N_t\times\sigma^N_t\rangle (X^{i,N}_t)dt\,
               \end{align}
		\item the conditional transition rates for $\state^{i, N}$ at time $t$, for a realization of $(\nu ^N_t,\sigma^N_t)$,  are given by
		\begin{itemize}
			\item if $\state^{i,N}_t = F$ then $F \rightarrow L$ with rate $\alpha_F(\nu ^N_t,\sigma^N_t)$,
			\item if $\state^{i,N}_t = L$ then $L \rightarrow F$ with rate $\alpha_L(\nu^N_t,\sigma^N_t)$,
		\end{itemize}
	where we used the shorthand notation \eqref{eq:alphashortcut} for $\alpha_F$ and $\alpha_L$.
	\end{itemize}

More formally, we define the processes $(\state^{1,N}_t,\ldots,\state^{N,N}_t)$ to be the jump processes such that $\Law(\state^{i,N}_t)$ satisfy the system of ODEs
\begin{align}\label{stoc-Y}
\frac{d}{dt}\Law(\state^{i,N}_t) = \mathbb{E}\left(\appstate{\nu^N_t}{\sigma^N_t}\right)\Law(\state^{i,N}_t) \quad i = 1,\ldots,N.
\end{align}
Notice that \eqref{stoc-Y} clearly stems out of the above definitions and the law of total probability,  averaging on all realizations of $(\nu ^N_t,\sigma^N_t)$.

\begin{remark}
We shortly discuss the well-posedness of the above defined processes, leaving the details to the reader.

For a realization of  $(\state^{1,N}_t,\ldots,\state^{N,N}_t)$ in the space of c\'adl\'ag functions, the applications $\langle K, \nu^N_t\times\sigma^N_t\rangle$ and $\appstate{\nu^N_t}{\sigma^N_t}$ are both measurable and bounded in time. Thus, \eqref{stoc-Y} has a right-hand side which is measurable and bounded with respect to $t$ and Lipschitz continuous with resect to $X$, uniformly for $t\in [0,T]$. Hence, the existence of Lipschitz continuous solutions to \eqref{stoc-X} uniquely determined by the initial data follows directly from the general theory in \cite{filipov}.

Concerning the stochastic processes $\state^{i,N}_t$ with law given by \eqref{stoc-Y}, they can be, for instance, realised as limit of discrete-in-time processes of the form
\[
\zweivect{\mathbb{P}\left\{\state^{i,N}_{t+h}=F|(\nu^N_t, \sigma^N_t)=(\tilde\nu, \tilde\sigma)\right\}}{\mathbb{P}\left\{\state^{i,N}_{t+h}=L|(\nu^N_t,\sigma^N_t)=(\tilde\nu, \tilde\sigma)\right\}} =\left(I+h \appstate{\tilde\nu}{\tilde\sigma}\right)\zweivect{\mathbb{P}\left\{\state^{i,N}_{t}=F\right\}}{\mathbb{P}\left\{\state^{i,N}_{t}=L\right\}}\quad i = 1,\ldots,N
\]
for $I$ being the identity matrix and $h>0$ a vanishing time step.	In the equation above notice that, since by construction the vector $(1,1)^T$ %$\zweivect{1}{1}$ 
belongs to the kernel of the transpose  $\appstate{\tilde\nu}{\tilde\sigma}^*$ of $\appstate{\tilde\nu}{\tilde\sigma}$ for every realization $ (\tilde\nu, \tilde\sigma)$ of $(\nu ^N_t,\sigma^N_t)$,  the left-hand side above is well-defined as a conditional probability law on $\{F,L\}$.
\end{remark}

\begin{remark}\label{rem:boundedsupp}
Since $\overline{\nu}$ has compact support in $\R^d$, and using \eqref{e-vf2}, standard arguments (see, e.g. \cite[Appendix]{MFOC}) entail that it exists $R_T>0$ such that, for all  $N \in \mathbb{N}$, $i=1,\dots,N$ and $t \in [0, T]$, it holds
\begin{equation}\label{eq:a-priori}
|X^{i, N}_t|\le R_T, \mbox{ so that  supp}(\nu^N_t) \subset B(0, R_T)\,.
\end{equation}
This inclusion has clearly to be understood as being verified with probability $1$.
\end{remark}

Our next step is defining, for fixed $i$ and $N$, an auxiliary averaged process $(\overline{X}^{i,N}_t,\overline{\state}^{i,N}_t)$ having the solutions $\nu_t$ and $\sigma_t$ of system \eqref{eq:macronusigma} as laws. To this purpose, we need some preparation which will be useful also in the sequel.

\begin{proposition}\label{prop:equivPDEODEoverline} 
	Let $(\nu_t,\sigma_t)$ be a solution of \eqref{eq:macronusigma} and define a process $(\overline{X}_t,\overline{\state}_t)$ as follows
	\begin{itemize}
		\item $\overline{X}_0 \sim \nu_0$ and $\overline{\state}_t \sim \sigma_0$,
		\item $d\overline{X}_t = \langle K, \nu_t\times\sigma_t\rangle (\overline{X}_t)dt$,
		\item the transition rates at time $t$ are given by
		\begin{itemize}
			\item if $\overline{\state}_t = F$ then $F \rightarrow L$ with rate $\alpha_F(\nu_t,\sigma_t)$,
			\item if $\overline{\state}_t = L$ then $L \rightarrow F$ with rate $\alpha_L(\nu_t,\sigma_t)$.
		\end{itemize}
	\end{itemize}
	Then $\nu_t = \Law(\overline{X}_t)$ and $\sigma_t = \Law(\overline{\state}_t)$.
\end{proposition}
\begin{proof}
        Define $\eta_t = \Law(\overline{X}_t)$ and let $\varphi \in \mathcal{C}^1_c(\R^d)$ be any test function. For $v_t= \langle K, \nu_t\times\sigma_t\rangle$, by definition of $\overline{X}_t$ and linearity of the expected values it holds that
	\begin{align*}
	\partial_t \langle \varphi, \eta_t \rangle  = \partial_t \mathbb{E}[\varphi(\overline{X}_t)] = \langle \varphi, -\diver(v_t\eta_t) \rangle.
	\end{align*}
	The initial condition $\eta_0 = \Law(\overline{X}_0)=\nu_0$ holds by definition. 
	Hence, $\Law(\overline{X}_t)$ is a solution to the PDE \begin{align*}
	\left\{\begin{aligned}\partial_t\eta_t &= -\diver(\langle K,\nu_t\times\sigma_t\rangle\eta_t),\\
	\eta(0) &= \nu_0,
	\end{aligned}\right.
	\end{align*}
	which is unique  by Lemma \ref{le:equivPDEODEtransp}. Since $\nu_t$ solves the same problem,  we get $\nu_t = \Law(\overline{X}_t)$. Moreover, as both $\sigma_t$ and $\Law(\overline{Y}_t)$ are solutions of
	$$\dot{\eta}_t = \appstate{\nu_t}{\sigma_t}\eta_t$$
	with initial condition $\sigma_0$, then again by uniqueness we have $\sigma_t = \Law(\overline{\state}_t)$.
\end{proof}

We are now in a position to define, for fixed $i$ and $N$, the processes $\overline{X}^{i,N}_t$ and $\overline{\state}^{i,N}_t$ through the following dynamics:
	\begin{itemize}
	\item $\overline{X}^{i,N}_0 = X^{i,N}_0$ and $\overline{\state}^{i,N}_0 = \state^{i,N}_0$,
	    \item $\Law(\overline{X}^{i,N}_t) = \nu_t$ and $\Law(\overline{\state}^{i,N}_t) = \sigma_t$,
		\item $d\overline{X}^{i,N}_t = \langle K, \nu_t\times\sigma_t\rangle (\overline{X}^{i,N}_t)dt$,
		\item the transition rates at time $t$ are given by
		\begin{itemize}
			\item if $\overline{\state}^{i,N}_t = F$ then $F \rightarrow L$ with rate $\alpha_F(\nu_t,\sigma_t)$,
			\item if $\overline{\state}^{i,N}_t = L$ then $L \rightarrow F$ with rate $\alpha_L(\nu_t,\sigma_t)$.
		\end{itemize}
	\end{itemize}

The well-posedness of such  processes is indeed a corollary of our previous results.

\begin{corollary}
	The processes $\overline{X}^{i,N}_t$ and $\overline{\state}^{i,N}_t$ exist for every $N\in\N$ and every $i = 1,\ldots,N$.
\end{corollary}
\begin{proof}
	Follows from Proposition \ref{prop:equivPDEODEoverline} and the existence of $(\nu_t,\sigma_t)$ from Corollary \ref{cor:nusigmaexistence}.
\end{proof}

For every fixed $t$, the processes $\overline{X}^{i,N}_t$ with $i=1,\dots, N$ are clearly independent of each other, and so are the processes $\overline{\state}^{i,N}_t$ .

Now, all the above constructions still leave one free to choose how to couple the processes $\state^{i, N}_t$ and $\overline{\state}^{i,N}_t$ in their product space\footnote{The coupling between $X^{i,N}_t$ and $\overline{X}^{i,N}_t$ is as usual tacitly defined by asking that $X^{i,N}_t-\overline{X}^{i,N}_t$ solves the SDE obtained as difference of the ones for $X^{i,N}_t$ and $\overline{X}^{i,N}_t$, respectively.} : we namely assume that
\[
\Law\left(\state^{i, N}_t, \overline{\state}^{i,N}_t\right)\in \Gamma_o(\Law(\state^{i, N}_t), \sigma_t)\,.
\]
Here optimality of the transportation plans is meant with respect of the distance $\W_1$ on $\mathcal{P}_1(\{F, L\})$, where (as everywhere in what follows) the set $\{F, L\}$ is endowed with the distance \eqref{eq:distfinite}. With the above choice, by the definition of $\W_1$ we have
\begin{align}\label{eq:optimalcoupling}
\mathbb{E}|\state^{i, N}_t-\overline{\state}^{i,N}_t|=\W_1(\Law(\state^{i, N}_t), \sigma_t)
\end{align}
for all $i$, $N$, and $t$. Since $\state^{i, N}_t$ and $\overline{\state}^{i,N}_t$ are random variables on the discrete space $\{F, L\}$, a simple computation using \eqref{eq: Kantorovich} together with \eqref{eq:optimalcoupling} entails that 
\begin{align}\label{eq:optimalcoupling2}
\mathbb{E}|\state^{i, N}_t-\overline{\state}^{i,N}_t|=\left|\mathbb{P}\{\state^{i, N}_t=F\}-\mathbb{P}\{\overline\state^{i, N}_t=F\}\right|\,.
\end{align}

\begin{remark}
	The relationship between the empirical mean of the independent processes $(\overline{X}^{i,N}_t,\overline{\state}^{i,N}_t)_{i = 1,\ldots,N}$ given by
	\begin{equation}\label{eq:averaged-empirical}
	\overline{\nu}^N_t \defin \frac1N \sum^N_{i = 1}\delta_{\overline{X}^{1,N}_t} \quad \text{ and } \quad \overline{\sigma}^N_t \defin \frac1N \sum^N_{i = 1}\delta_{\overline{\state}^{i,N}_t}
	\end{equation}
	and $(\nu_t,\sigma_t)$ solution of system \eqref{eq:macronusigma} is clear: by the Glivenko-Cantelli's theorem, $(\overline{\nu}^N_t,\overline{\sigma}^N_t)$ converges weakly to $(\nu_t,\sigma_t)$ as $N\rightarrow+\infty$. The rate of convergence can actually be quantified thanks to \cite[Theorem 1]{fournier} (which holds for $\nu_t$ and $\sigma_t$, since their support is uniformly compact in time): we may apply it once for $\sigma_t$ and the values $p = d = 1, q = 3$ to get for every $t \in [0,T]$
	\begin{align}\begin{split}\label{eq:wassfourniersigma}
	\mathbb{E}\left[\W_1(\overline{\sigma}^N_t,\sigma_t)\right] &\leq K_1 (N^{-1/2} + N^{-2/3}),
	\end{split}
	\end{align}
	for a given constant $K_1>0$. If we apply it for $\nu_t$ and the values $p = 2d, q = 3p$ (where $d$ here denotes the dimension of the state space $\R^d$) we get for every $t \in [0,T]$
	\begin{align*}\begin{split}
	\mathbb{E}\left[\W_{2d}(\overline{\nu}^N_t,\nu_t)^{2d}\right] &\leq K_2 (N^{-1/2} + N^{-2/3}),
	\end{split}
	\end{align*}
	for some $K_2>0$. Since it is well-known that $\W_1 \leq \W_p$ for any $p \geq 1$, an application of Jensen's inequality yields the estimate
	\begin{align}\begin{split}\label{eq:wassfourniernu}
	\mathbb{E}\left[\W_{1}(\overline{\nu}^N_t,\nu_t)\right] &\leq K_2^{1/2d} (N^{-1/4d} + N^{-1/3d}).
	\end{split}
	\end{align}
	Setting
	\begin{align*}
	\Theta(N) \defin K_1 (N^{-1/2} + N^{-2/3})+K_2^{1/2d} (N^{-1/4d} + N^{-1/3d}),
	\end{align*}
	and putting together \eqref{eq:wassfourniernu} and \eqref{eq:wassfourniersigma} we obtain
	\begin{align}\label{eq:wassfourniertotal}
	\mathbb{E}\left[\W_{1}(\overline{\nu}^N_t,\nu_t)\right]+\mathbb{E}\left[\W_1(\overline{\sigma}^N_t,\sigma_t)\right] &\leq \Theta(N).
	\end{align}
\end{remark}

\begin{remark}[Exchangeability of processes]\label{rem:exchange}
    Notice a fundamental property of the processes $(\overline{X}^{i,N}_t,\overline{\state}^{i,N}_t)_{i = 1,\ldots,N}$: for every $i,j=1,\ldots,N$ and every $t\geq0$ we have
    $$\mathbb{E}\left|X^{i,N}_t - \overline{X}^{i,N}_t\right| = \mathbb{E}\left|X^{j,N}_t - \overline{X}^{j,N}_t\right| \quad\text{ and }\quad \mathbb{E}\left|\state^{i,N}_t - \overline{\state}^{i,N}_t\right| = \mathbb{E}\left|\state^{j,N}_t - \overline{\state}^{j,N}_t\right|.$$
    After noticing that both identities hold trivially at $t=0$, this clearly follows from the simmetry of the processes $(X^{i,N}_t, \state^{i,N}_t)$ and the fact that       $(\overline{X}^{i,N}_t,\overline{\state}^{i,N}_t)_{i = 1,\ldots,N}$ are independent. In particular, this exchangeability implies that
    $$\frac1N\sum^N_{j=1} \mathbb{E}\left|X^{j,N}_t - \overline{X}^{j,N}_t\right| = \mathbb{E}\left|X^{i,N}_t - \overline{X}^{i,N}_t\right|,$$
    as well as
    $$\frac1N\sum^N_{j=1} \mathbb{E}\left|\state^{j,N}_t - \overline{\state}^{j,N}_t\right| = \mathbb{E}\left|\state^{i,N}_t - \overline{\state}^{i,N}_t\right|\,.$$
\end{remark}

\subsection{The mean-field limit}

The main goal of this section is to show that, for $N$ large, the random empirical distributions $\nu^N_t$ and $\sigma^N_t$ associated to the processes $(X^{1,N}_t,\state^{1,N}_t),\ldots,(X^{N,N}_t,\state^{N,N}_t)$ defined in the previous subsection are close, in a probabilistic sense, to the deterministic measures $\nu_t$ and $\sigma_t$, solutions of system \eqref{eq:macronusigma} with initial datum $(\overline{\nu},\overline{\sigma})$.
The result we aim to prove is namely the following.

\begin{theorem}\label{th:main}
	Under Assumptions \ref{ass:lipK} and \ref{ass:lipalpha}, there exists a  function $\Psi:\N\rightarrow\R_+$ satisfying $\lim_{N\to +\infty}\Psi(N)=0$ such that
	\begin{equation*}
	\sup_{t\geq 0}\mathbb{P}\left(\W_1(\nu^N_t,\nu_t) + \W_1(\sigma^N_t,\sigma_t) > \varepsilon\right) \leq \varepsilon^{-1}\Psi(N).
	\end{equation*}
\end{theorem}

For the proof of Theorem we will need a key intermediate result that we state below.
\begin{lemma}[Uniform propagation of chaos]\label{le:propchaos}
        Define the empirical  measures $\nu^N_t$, $\sigma^N_t$, $\overline{\nu}^N_t$, and $\overline{\sigma}^N_t$ through \eqref{eq:empmean}, and \eqref{eq:averaged-empirical}, respectively.
	Under Assumptions \ref{ass:lipK} and \ref{ass:lipalpha}, there exists a  function $\Phi:\N\rightarrow\R_+$ satisfying $\lim_{N\to +\infty}\Phi(N)=0$ such that
	\begin{equation}\label{eq:propchaos}
	\sup_{t \geq 0} \mathbb{E}\bigg[\left| X^N_i(t) - \overline{X}^N_i(t)\right| + \left| \state^N_i(t) - \overline{\state}^N_i(t)\right| \bigg] \leq \Phi(N).
	\end{equation}
\end{lemma}

The proof of this result is postponed to the next section. Let us first show how Theorem \ref{th:main} can be easily derived, once Lemma \ref{le:propchaos} is established.

\begin{proof}[Proof of Theorem \ref{th:main}]
	By the triangular inequality it follows that
	\begin{align}
	\begin{split}\label{eq:wasstriangle}
	\W_1(\nu^N_t,\nu_t) & \leq \W_1(\nu^N_t,\overline{\nu}^N_t) + \W_1(\overline{\nu}^N_t,\nu_t),\\
	\W_1(\sigma^N_t,\sigma_t) & \leq \W_1(\sigma^N_t,\overline{\sigma}^N_t) + \W_1(\overline{\sigma}^N_t,\sigma_t).\\
	\end{split}
	\end{align}
	Since $\nu^N_t,\overline{\nu}^N_t,\sigma^N_t$ and $\overline{\sigma}^N_t$ are all atomic measures, by the properties of the Wasserstein distance we have
	\begin{align}\label{eq:wassersteintomean}
	\W_1(\nu^N_t,\overline{\nu}^N_t) \leq \frac1N\sum^N_{i = 1}\left|X^{i,N}_t - \overline{X}^{i,N}_t\right| \quad \text{ and } \quad \W_1(\sigma^N_t,\overline{\sigma}^N_t) \leq \frac1N\sum^N_{i = 1}\left|\state^{i,N}_t - \overline{\state}^{i,N}_t\right|.
	\end{align}
	
	%%%%%%%%%%%%%%%%%%%%%%%%%%%%%
	%%%%%%%%%%%%%%%%%%%%%%%%%%%%%
	\begin{comment}
	Therefore, by Lemma \ref{le:propchaos}, there exists a constant $K > 0$ such that
	\begin{align}\label{eq:wassalmostthere}
	\sup_{t \geq 0}\mathbb{E}\left[\W_1(\nu^N_t,\overline{\nu}^N_t) + \W_1(\sigma^N_t,\overline{\sigma}^N_t)\right] \leq K N^{-1/2}.
	\end{align}
	On the other hand, by \eqref{eq:wassfourniersigma} and \eqref{eq:wassfourniernu}, there exist $K_1,K_2 > 0$ for which
	\begin{align}\label{eq:wasslaststep}
	\sup_{t \geq 0}\mathbb{E}\left[\W_1(\overline{\nu}^N_t,\nu_t) + \W_1(\overline{\sigma}^N_t,\sigma_t)\right] \leq K_1 (N^{-1/2} + N^{-1/3}) + K_2^{1/2d} (N^{-1/4d} + N^{-1/3d})
	\end{align}
	holds.
	\end{comment}
	%%%%%%%%%%%%%%%%%%%%%%%%%%%%%
	%%%%%%%%%%%%%%%%%%%%%%%%%%%%%
	
	Therefore, by taking expectations in \eqref{eq:wasstriangle} and plugging \eqref{eq:wassfourniertotal} and \eqref{eq:propchaos} on the the right-hand side, we obtain
	\begin{align*}
	\mathbb{E}\left[\W_1(\nu^N_t,\nu_t) + \W_1(\sigma^N_t,\sigma_t)\right]\leq \Phi(N) + \Theta(N).
	\end{align*}
	If we set
	$$\Psi(N)\defin \Phi(N) + \Theta(N),$$
	an application of Markov's inequality concludes the proof.
\end{proof}

\subsection{Proof of Lemma \ref{le:propchaos}}
First, we start from the term $\mathbb{E}|X^{i,N}_t - \overline{X}^{i,N}_t|$. By integrating the dynamics of $X^{i,N}_t$ from $0$ to $t$ we obtain
$$X^{i,N}_t = X^{i,N}_0 + \int^t_0 \langle K, \nu^N_s\times\sigma^N_s\rangle (X^{i,N}_s)\,\mathrm{d}s,$$
and similarly for $\overline{X}^{i,N}_t$ we get
$$\overline{X}^{i,N}_t = X^{i,N}_0 + \int^t_0 \langle K, \nu_s\times\sigma_s\rangle (\overline{X}^{i,N}_s)\,\mathrm{d}s.$$
Above we used that, by definition, $\overline{X}^{i,N}_0 = X^{i,N}_0$. Therefore, adding and subtracting the terms $\langle K, \overline{\nu}^N_s\times\overline{\sigma}^N_s\rangle (X^{i,N}_s)$ and $\langle K, \overline{\nu}^N_s\times\overline{\sigma}^N_s\rangle (\overline{X}^{i,N}_s)$, we get the estimate
\begin{align}\label{eq:meanfieldnu0}
    \mathbb{E}|X^{i,N}_t &- \overline{X}^{i,N}_t| = \nonumber\\
    & = \mathbb{E}\Bigg|X^{i,N}_0 + \int^t_0 \langle K, \nu^N_s\times\sigma^N_s\rangle (X^{i,N}_t)\,\mathrm{d}s- X^{i,N}_0 - \int^t_0 \langle K, \nu_s\times\sigma_s\rangle (\overline{X}^{i,N}_s)\,\mathrm{d}s\Bigg|\nonumber\\
    &\leq \underbrace{\mathbb{E}|X^{i,N}_0 - X^{i,N}_0|}_{=0} + \underbrace{\mathbb{E}\left[\int^t_0 |\langle K, \nu^N_s\times\sigma^N_s\rangle (X^{i,N}_s) - \langle K, \overline{\nu}^N_s\times\overline{\sigma}^N_s\rangle (X^{i,N}_s)|\,\mathrm{d}s\right]}_{I_1}\nonumber\\
    &\quad\quad +\underbrace{\mathbb{E}\left[\int^t_0 \left|\langle K, \overline{\nu}^N_s\times\overline{\sigma}^N_s\rangle (X^{i,N}_s) - \langle K, \overline{\nu}^N_s\times\overline{\sigma}^N_s\rangle (\overline{X}^{i,N}_s)\right|\,\mathrm{d}s\right]}_{I_2}\nonumber\\
    &\quad\quad + \underbrace{\mathbb{E}\left[\int^t_0 \left|\langle K, \overline{\nu}^N_s\times\overline{\sigma}^N_s\rangle (\overline{X}^{i,N}_s) - \langle K, \nu_s\times\sigma_s\rangle (\overline{X}^{i,N}_s)\right|\,\mathrm{d}s\right]}_{I_3}.
\end{align}
We shall now estimate from above the terms $I_1,I_2$ and $I_3$. Recall that  \eqref{eq:a-priori} and property (3) in Definition \ref{def:solnusigma} hold. This latter also gives that
\begin{equation}\label{eq:a-priori2}
|\overline{X}^{i, N}_t|\le R_T, \mbox{ so that  supp}(\overline{\nu}^N_t) \subset B(0, R_T)\,,
\end{equation}
with probability $1$. With this, \eqref{e-vf2}, and \cite[Lemma A.7]{MFOC},  for $I_1$ we have
\begin{align}\begin{split}\label{eq:meanfieldnu1}
    I_1 &\leq \mathbb{E}\left[\int^t_0 L_K(\W_1(\nu^N_s,\overline{\nu}^N_s) + \W_1(\sigma^N_s,\overline{\sigma}^N_s))\,\mathrm{d}s\right]\\
    & = \int^t_0 L_K(\mathbb{E}\left[\W_1(\nu^N_s,\overline{\nu}^N_s)\right] + \mathbb{E}\left[\W_1(\sigma^N_s,\overline{\sigma}^N_s)\right])\,\mathrm{d}s\\
    &\leq L_K\int^t_0\left(\frac1N\sum^N_{j=1}\mathbb{E}|X^{j,N}_s - \overline{X}^{j,N}_s| + \frac1N\sum^N_{j=1}\mathbb{E}|\state^{j,N}_s - \overline{\state}^{j,N}_s|\right)\,\mathrm{d}s\\
    &=L_K\int^t_0(\mathbb{E}|X^{i,N}_s - \overline{X}^{i,N}_s| + \mathbb{E}|\state^{i,N}_s - \overline{\state}^{i,N}_s|)\,\mathrm{d}s,
\end{split}
\end{align}
where we additionally used  the inequalities \eqref{eq:wassersteintomean} and Remark \ref{rem:exchange}. With \eqref{e-vf2} and the same argument in \eqref{e-Kmunu} we deduce
\begin{align}\begin{split}\label{eq:meanfieldnu2}
    I_2 &\leq L_K\int^t_0 \mathbb{E}|X^{i,N}_s - \overline{X}^{i,N}_s|\,\mathrm{d}s.
\end{split}
\end{align}
Finally, using \eqref{eq:a-priori2} within the same steps used for $I_1$, together with \eqref{eq:wassfourniertotal}, yields
\begin{align}\begin{split}\label{eq:meanfieldnu3}
    I_3 &\leq  \int^t_0 L_K(\mathbb{E}\left[\W_1(\overline{\nu}^N_s,\nu_s)\right] + \mathbb{E}\left[\W_1(\overline{\sigma}^N_s,\sigma_s)\right])\,\mathrm{d}s\\
    &\leq L_K \Theta(N) t.
\end{split}
\end{align}
By plugging \eqref{eq:meanfieldnu1}, \eqref{eq:meanfieldnu2} and \eqref{eq:meanfieldnu3} into \eqref{eq:meanfieldnu0} we finally obtain
\begin{align}\begin{split}\label{eq:meanfieldnu4}
    \mathbb{E}|X^{i,N}_t - \overline{X}^{i,N}_t| &\leq L_K \Theta(N) t \\
    & \quad\quad+ 2L_K\int^t_0(\mathbb{E}|X^{i,N}_s - \overline{X}^{i,N}_s| + \mathbb{E}|\state^{i,N}_s - \overline{\state}^{i,N}_s|)\,\mathrm{d}s.
\end{split}
\end{align}

We now turn to the term $\mathbb{E}|\state^{i,N}_t - \overline{\state}^{i,N}_t|$.
Using  \eqref{stoc-Y} and \eqref{eq:optimalcoupling2}, and since $\state^{i,N}_0=\overline{\state}^{i,N}_0$ we  have
\begin{align*}
\mathbb{E}|\state^{i,N}_t &- \overline{\state}^{i,N}_t| \le\\
    & \le \int_0^t \left|\frac{\mathrm{d}}{\mathrm{d}s}\mathbb{P}\{\state^{i, N}_s=F\}-\frac{\mathrm{d}}{\mathrm{d}s}\mathbb{P}\{\overline\state^{i, N}_s=F\}\right|\,\mathrm{d}s\\
    & \le \int_0^t \left|\mathbb{E}(\alpha_F(\nu_s, \sigma_s))\mathbb{P}\{\overline\state^{i, N}_s=F\}-\mathbb{E}(\alpha_F(\nu^N_s, \sigma^N_s))\mathbb{P}\{\state^{i, N}_s=F\}\right|\,\mathrm{d}s  \\
 & +
\int_0^t \left|\mathbb{E}(\alpha_L(\nu^N_s, \sigma^N_s))\mathbb{P}\{\state^{i, N}_s=L\}-\mathbb{E}(\alpha_L(\nu_s, \sigma_s))\mathbb{P}\{\overline{\state}^{i, N}_s=L\}\right|\,\mathrm{d}s \\
 & \le 
 \int_0^t\mathbb{E} \left|\alpha_F(\nu_s, \sigma_s)-\alpha_F(\nu^N_s, \sigma^N_s)\right|\,\mathrm{d}s +\int_0^t \mathbb{E} \left|\alpha_L(\nu_s, \sigma_s)-\alpha_L(\nu^N_s, \sigma^N_s)\right|\,\mathrm{d}s\\
 & +\int_0^t
\left(|\mathbb{E}(\alpha_F(\nu^N_s, \sigma^N_s))|+|\mathbb{E}(\alpha_L(\nu^N_s, \sigma^N_s))|\right)\,\left|\mathbb{P}\{\state^{i, N}_s=F\}-\mathbb{P}\{\overline\state^{i, N}_s=F\}\right|\mathrm{d}s\,,
\end{align*}
where we additionally exploited that clearly 
$$
\left|\mathbb{P}\{\state^{i, N}_s=F\}-\mathbb{P}\{\overline\state^{i, N}_s=F\}\right|=\left|\mathbb{P}\{\state^{i, N}_s=L\}-\mathbb{P}\{\overline\state^{i, N}_s=L\}\right|\,.
$$
By Assumption \ref{ass:supalpha} there exists a uniform constant $L'_\alpha > 0$ such that $|\mathbb{E}(\alpha_F(\nu^N_s, \sigma^N_s))|+|\mathbb{E}(\alpha_L(\nu^N_s, \sigma^N_s))|\le L'_\alpha$.  We also recall that, by \eqref{eq:a-priori} and property (3) in Definition \ref{def:solnusigma}, $\nu^N_s$ and $\nu_s$ have by construction support contained in a compact set $B(0, R_T) \subset \R^d$ independent of $N$ and $s$. We can therefore use \eqref{e-w1lip} and \eqref{eq:optimalcoupling2}, and continue the above estimate to get
\begin{align*}
\mathbb{E}|\state^{i,N}_t &- \overline{\state}^{i,N}_t| \le\\
    & \le 2L_\alpha\int_0^t \left(\mathbb{E}\left[\W_1(\nu^N_s,\nu_s)\right] + \mathbb{E}\left[\W_1(\sigma^N_s,\sigma_s)\right]\right)\,\mathrm{d}s+ L'_\alpha\int_0^t \mathbb{E}|\state^{i,N}_s - \overline{\state}^{i,N}_s|\,\mathrm{d}s\,.
\end{align*}
With \eqref{eq:wasstriangle} and \eqref{eq:wassersteintomean}, plugging \eqref{eq:wassfourniertotal}, and with Remark \ref{rem:exchange}, we then have
\begin{align*}
\mathbb{E}|\state^{i,N}_t &- \overline{\state}^{i,N}_t| \le\\
    & \le 2L_\alpha\Theta(N)t + 2L_\alpha\int_0^t \mathbb{E}|X^{i,N}_s - \overline{X}^{i,N}_s|\,\mathrm{d}s+(2L_\alpha+ L'_\alpha)\int_0^t \mathbb{E}|\state^{i,N}_s - \overline{\state}^{i,N}_s|\,\mathrm{d}s\\
    & \leq 2L_\alpha\Theta(N)t + (2L_\alpha+ L'_\alpha)\int_0^t \left(\mathbb{E}|X^{i,N}_s - \overline{X}^{i,N}_s| + \mathbb{E}|\state^{i,N}_s - \overline{\state}^{i,N}_s|\right)\,\mathrm{d}s \,.
\end{align*}

Summing the above estimate to \eqref{eq:meanfieldnu4}, we obtain the integral inequality
\begin{align*}
    \mathbb{E}|X^{i,N}_t - \overline{X}^{i,N}_t| &+ \mathbb{E}|\state^{i,N}_t - \overline{\state}^{i,N}_t| \leq (L_K+2L_\alpha) \Theta(N)t \\ 
    & \quad\quad+ (2L_K+2L_{\alpha}+L'_\alpha)\int^t_0\left(\mathbb{E}|X^{i,N}_s - \overline{X}^{i,N}_s| + \mathbb{E}|\state^{i,N}_s - \overline{\state}^{i,N}_s|\right)ds.
\end{align*}
Hence, an application of Gronwall's inequality to the function $\mathbb{E}|X^{i,N}_t - \overline{X}^{i,N}_t| + \mathbb{E}|\state^{i,N}_t - \overline{\state}^{i,N}_t|$ inside the interval $[0,T]$ yields
\begin{align*}
    \mathbb{E}|X^{i,N}_t - \overline{X}^{i,N}_t| + \mathbb{E}|\state^{i,N}_t - \overline{\state}^{i,N}_t| \leq (L_K+2L_\alpha) T \mathrm{e}^{(2L_K+2L_{\alpha}+L'_\alpha)T}\cdot \Theta(N)\,,
\end{align*}
which is the desired statement.

\section{Numerical experiments and applications}\label{sec-num}
We finally provide some practical applications of the present study, by numerically implementing some examples of social interaction dynamics. We will discuss the well-posedness of these examples according to theoretical assumptions. In particular, we remark that in all examples we account a bounded computational domain, therefore condition (H3) will be automatically satisfied.
Numerical simulations are performed with a first-order finite volume scheme. Details of the implementation are reported in \ref{sec:FVscheme}. 

%Since the computational domain is bounded we do non need to worry about conidtion (H3)

\subsection{Test I: Consensus dynamics}
We aim to show the evolution of the mean-field system when the measures $\mu_t^F, \mu_t^L$  have a bounded confidence interaction kernel. Therefore, we consider the  Hegselmann-Krause type interactions
\begin{align}\label{HKs}
K^i(x) = a^i(x)x,\qquad a^i(x) = \chi_{\{|x|\leq C^i\}} (x),\qquad i\in\{F,L\}
\end{align}
where $C^F,C^L>0$ are the {\em confidence thresholds} respectively for the followers, and leaders. We remark that Assumptions \ref{ass:lipK} would require to replace the indicator functions $a_i$'s with  Lipschitz approximations thereof. When $a^F$, and $a^L$ are two bounded, Lipschitz continuous functions, a direct computation shows indeed that the functions $K^i(x)=a^i(x)x$
satisfy Assumptions \ref{ass:lipK}  inside $B(0,R_T)$, which is enough since our measures are compactly supported. 
%Assumption \ref{ass:sublinK} holds globally. 
On the other hand, the experiments are not affected by such a smoothing procedure, hence we keep the definition \eqref{HKs} throughout this section.

We want to solve numerically the evolution of the mean-field interaction dynamics, observing the impact of different choices of  birth rates functions $\alpha_F,\alpha_L$.  We select the computational domain $\Omega =[-1,1]$ and system \eqref{eq:macroleadfollstrong} complemented by zero-flux boundary conditions.

Let $\sigma_t(F)$ and $\sigma_t(L)$ be the total mass of followers and leaders at time $t$, respectively. Since the total mass is preserved, by renormalizing at the initial time, it holds $\sigma_t(L)+\sigma_t(F)=1$. We assume that at time $t=0$ the initial data is uniformly distributed in $\Omega$ with initial density $\sigma_0(L)=1-\sigma_t(F)=0.75$. 

We report in Table \ref{Tab1} the  model parameters for two different test cases. In both cases, we assume that leaders have larger range of influence than followers, with $C^F=0.2$ and $C^L=0.6$.
For the numerical discretization, we select $N=80$ space grid points, time step $\Delta t = 0.0127$ and final time $T =25$.
\begin{table}[t]
\centering
\caption{Computational parameters for Test I.}
\begin{tabular}{c|c|c|c|c|c|c|c|c|}
\hline
Test  & $C^F$ & $C^L$ &$\alpha_F$& $\alpha_L$&$\sigma_0(F)$ &$\sigma_0(L)$ & $\delta_F$& $\delta_L$ \\
\hline
\hline
 Ia & 0.2 & 0.6 & 0.1 & 0.95&0.75&0.25 & -- & --\\
 Ib & 0.2 & 0.6 & \eqref{T1b:alfaF}& \eqref{T1b:alfaL}&0.75&0.25 & 0.35 & 0.2\\
\hline
\end{tabular}
\label{Tab1}
\end{table}

\vspace{+0.25cm}
\noindent
{\em Test Ia: constant rates.} We have reported  in Figure \ref{fig1} the evolution of the mean-field system with different simulations, when constant transition rates $\alpha_F,\alpha_L$ are selected. According to Table \ref{Tab1}, we selected 
$\alpha_F=0.1, \alpha_L=0.95$.
Notice that in the case of constant rates the total masses $\sigma_t(F),\sigma_t(L)$ converge to
\[
\sigma_\infty(F) = \frac{\alpha_L}{\alpha_F+\alpha_L}, \qquad \sigma_\infty(L) = \frac{\alpha_F}{\alpha_F+\alpha_L}.
\]
Figure \ref{fig1} depicts the density $\nu_t(x)$ in the time interval $\Omega\times[0,T]$, and the time evolution of $\sigma_t(F)$ and $\sigma_t(L)$. In Figure \ref{fig2}  we report different time frame of the densities $\mu^F_t,\mu^L_t$. We observe that at final time the system has clustered around three states.
\begin{figure}
\centering
\includegraphics[scale=0.375]{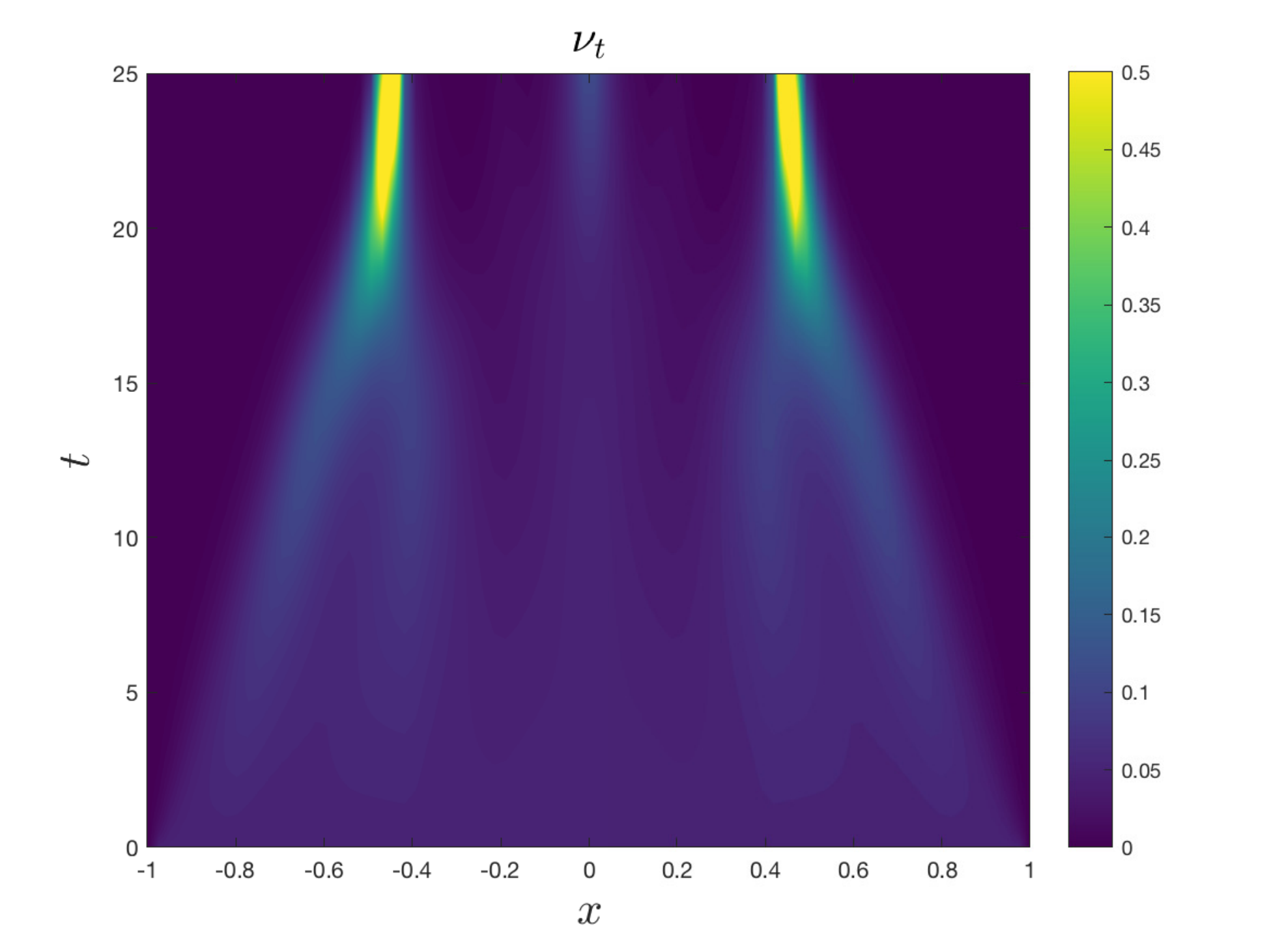}
\qquad
\includegraphics[scale=0.375]{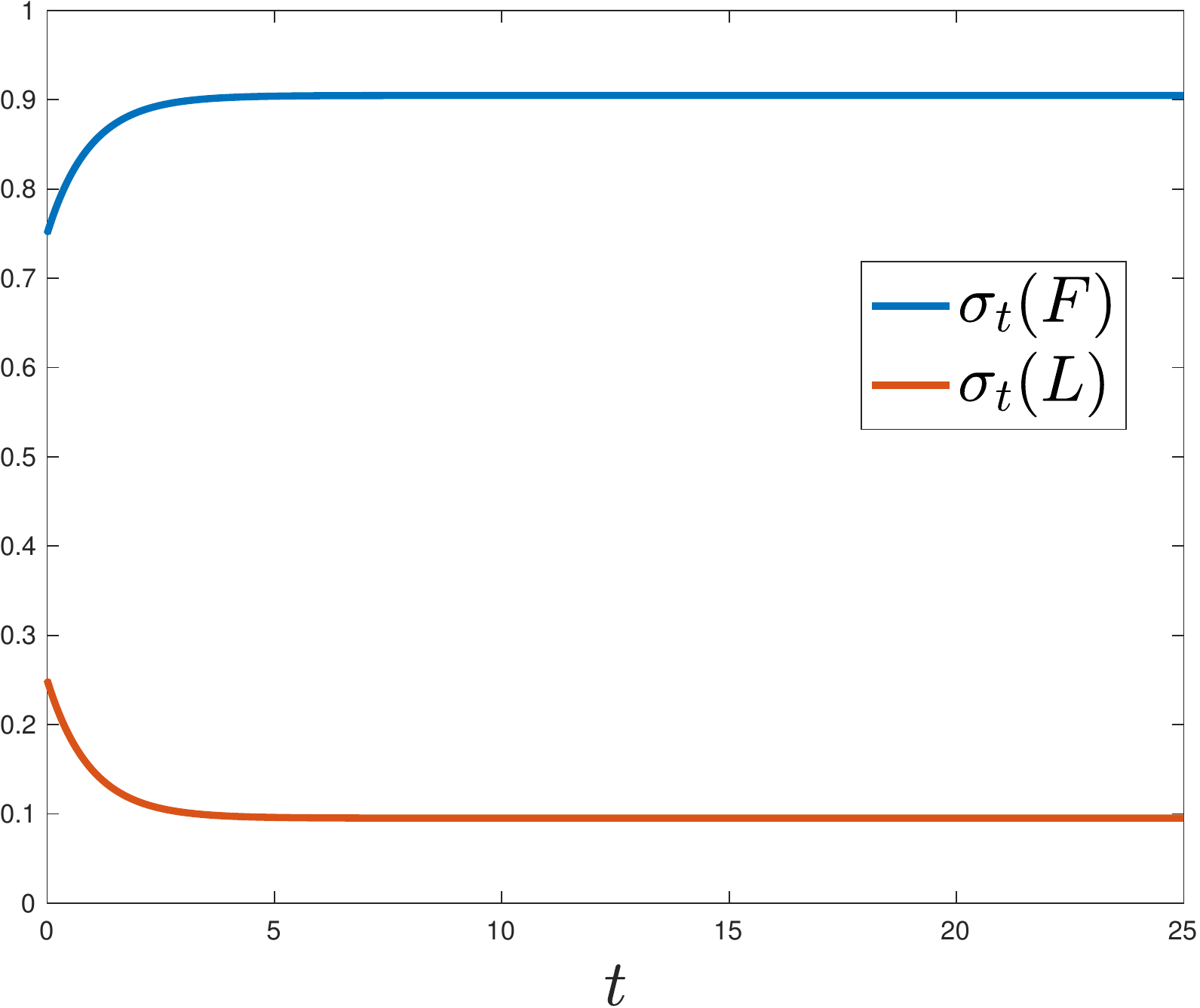}
\caption{{\em Test Ia}. 
Left: the total density $\nu_t$ in the space time domain $[-1,1]\times[0,T]$. Right: the followers' and leaders' mass, $\sigma_t(F),\sigma_t(L)$, with a monotonic evolution in time induced by the constant rates $\alpha_F=0.15$ and $\alpha_L=0.95$.}\label{fig1}
\end{figure}

\begin{figure}
\centering
\includegraphics[scale=0.275]{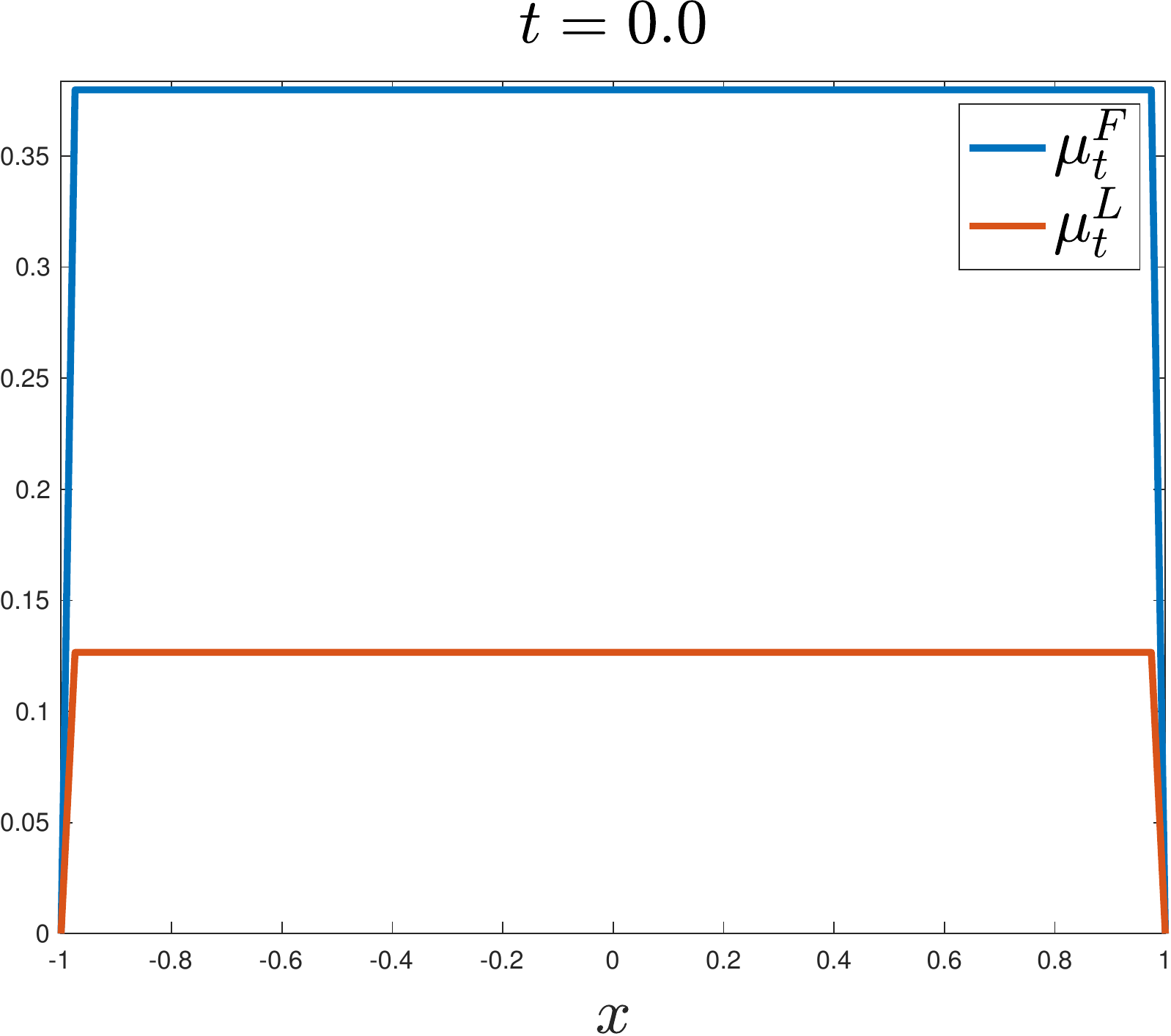}
\includegraphics[scale=0.275]{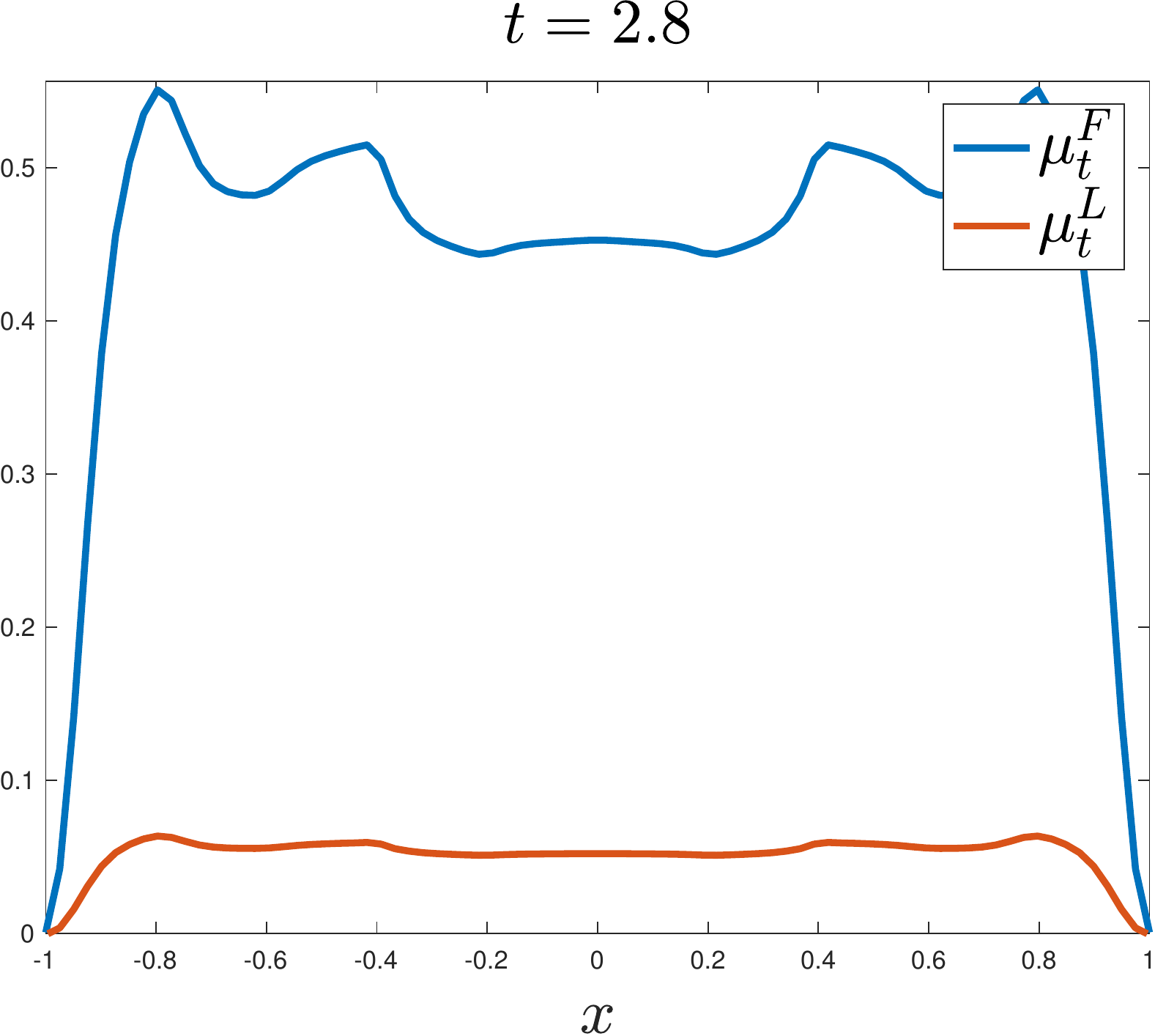}
\includegraphics[scale=0.275]{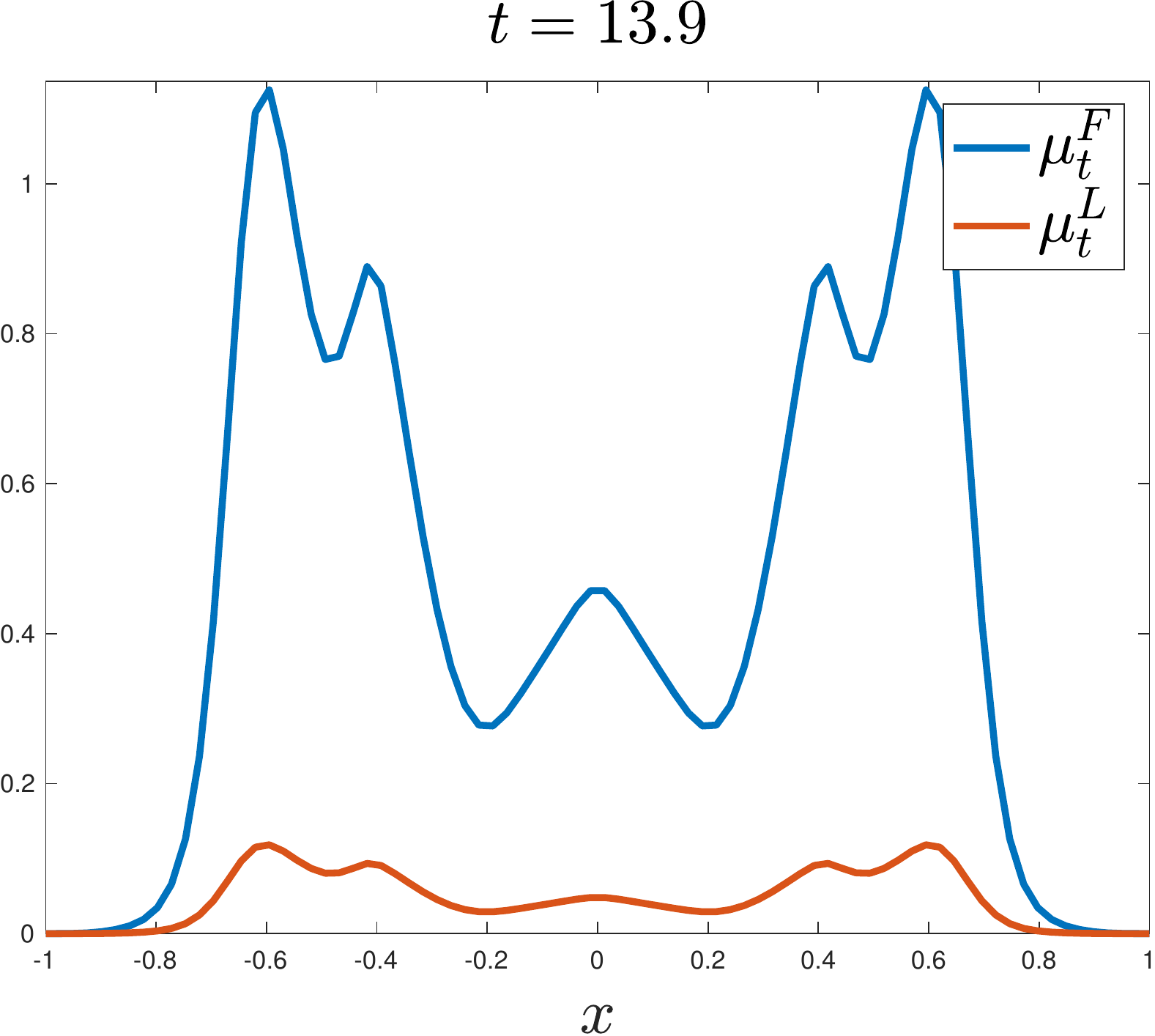}
\\
\includegraphics[scale=0.275]{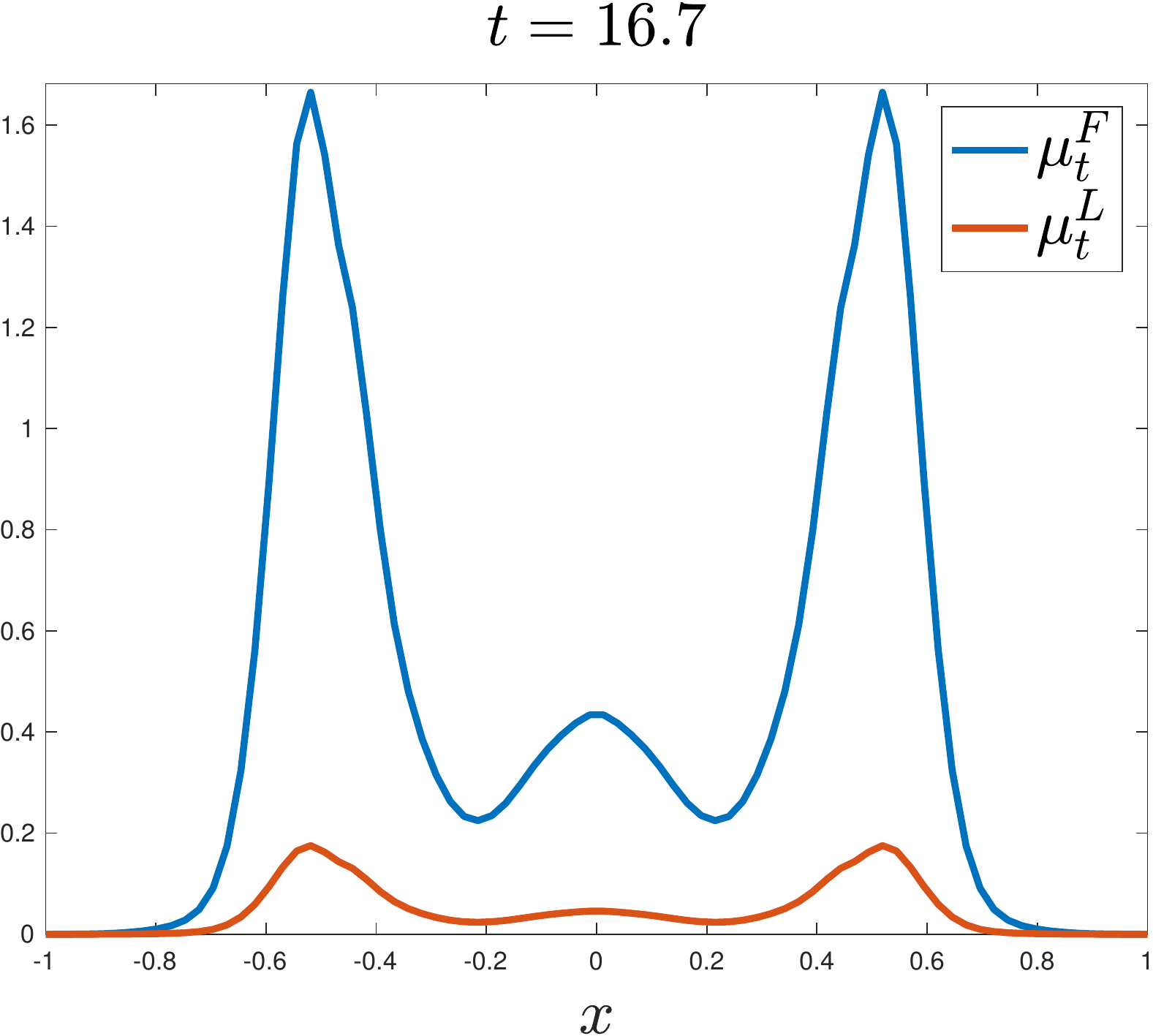}
\includegraphics[scale=0.275]{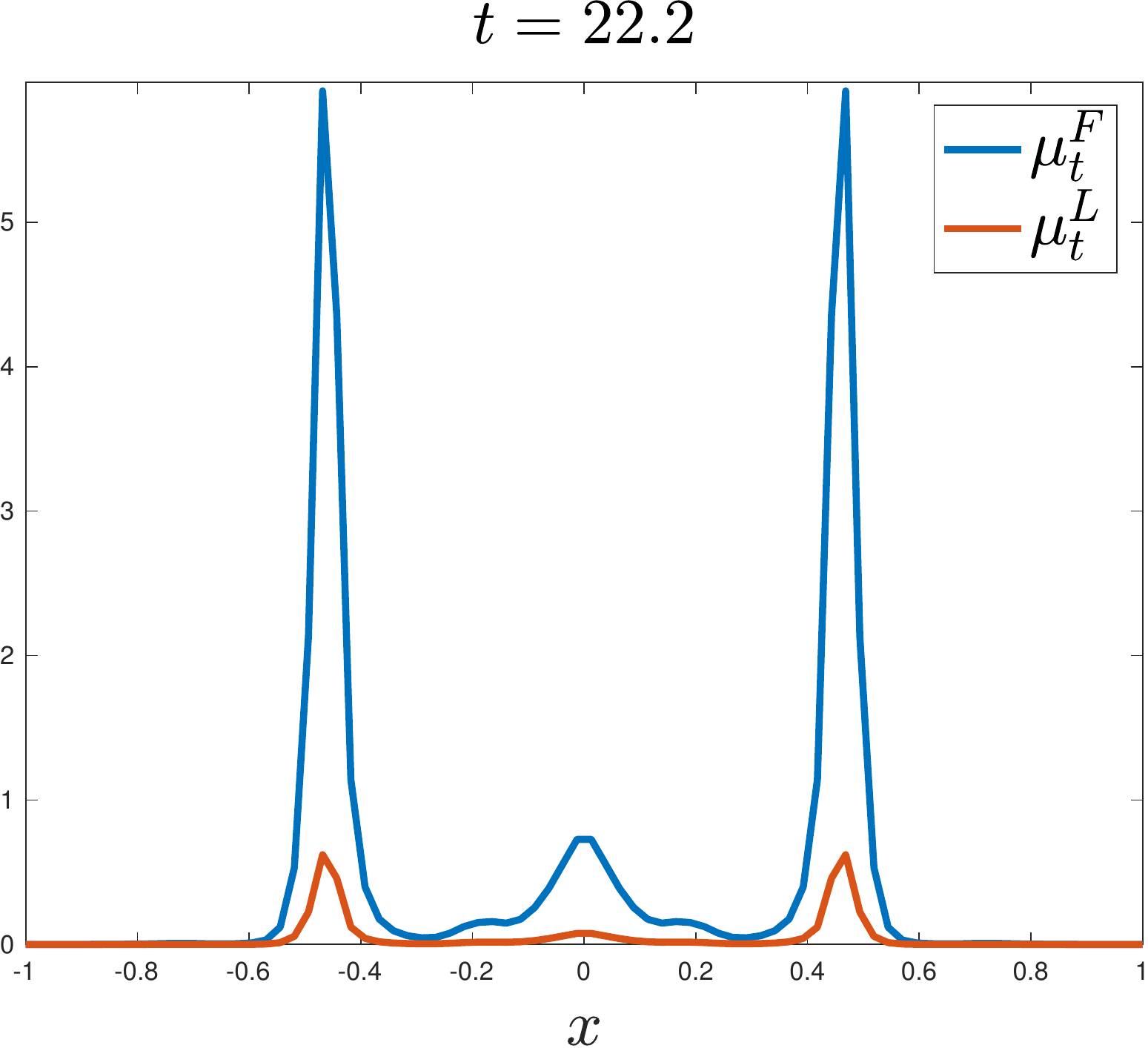}
\includegraphics[scale=0.275]{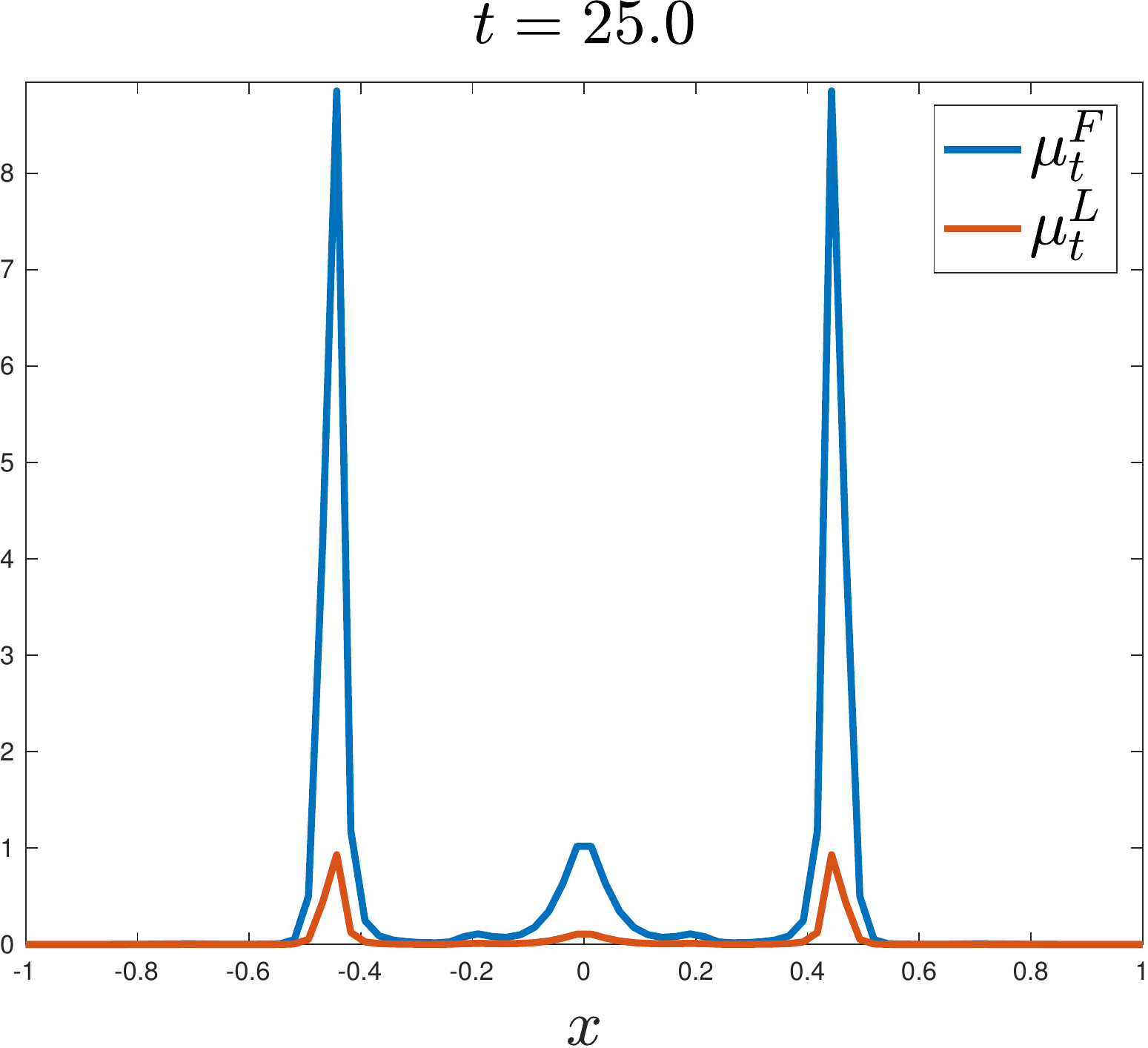}

\caption{{\em Test Ia}. From left to right, and top to bottom row we show the emergence of consensus with leaders' interaction confidence level $C^L=0.6$, and followers' interaction confidence level $C^F=0.2$. Consensus state is not yet reached and three main clusters emerge.}\label{fig2}
\end{figure}

\vspace{+0.25cm}
\noindent
{\em Test Ib: Density-dependent rates.} 
We consider birth rates depending on the densities $\mu_t^F,\mu_t^L$. We consider the variance measure of $\mu_t^L$ defined as follows
\begin{align}\label{eq:DmuL}
\mathcal{V}(\mu_t^L)=\frac{1}{|\sigma_t(L)|^2}\int_{\Omega\times\Omega} |x-y |^2\,d\mu_t^L(x)\, d\mu_t^L(y),
\end{align}
which measures the spread of the solution $\mu^L_t$ over $\Omega$.  The birth rate of leaders $\alpha_F$ is selected as a switching function with respect to the dispersion measure \eqref{eq:DmuL}, such that creation is activated only when the dispersion is above a certain threshold $\delta_F\geq 0$. 
Thus, we consider the following Lipschitz approximation of the indicator function
\begin{align}\label{T1b:alfaF}
\alpha_F(\mu_t^F,\mu_t^L) =\frac{1}{1+e^{c_F(\delta_F-\mathcal{V}(\mu^L_t))}}
\end{align}
with $c_F\gg1$, here we select $c_F = 1000$, and $\delta_F=0.15$.

Note that function \eqref{eq:DmuL} is exactly of the form \eqref{eq:esempio}, with $f(x)=|x|^2$. At the same time equation \eqref{T1b:alfaF} complies with Assumption \ref{ass:lipalpha}, as shown in \ref{sec:assumptions}, as long as $|\sigma_t(L)|\ge \epsilon$ for a fixed threshold $\epsilon >0$. This last condition can be easily checked along the evolution. 

The creation of followers given by rate $\alpha_L$ is instead determined by the following switching function
\begin{align}\label{T1b:alfaL}
\alpha_L(\mu_t^F,\mu_t^L) =\frac{1}{1+e^{c_L(\delta_L-|\sigma_t(L)|)}},
\end{align}
namely when the total mass of leaders is above a  threshold $\delta_L$. Here we selected  $\delta_L = 0.25$, and $c_L= 1000$. 

Similarly to the previous test, we show in Figure \ref{fig3} the total density $\nu_t(x)$ on $[0,T]\times\Omega$ and the time evolution of $\sigma_t(F)$ and $\sigma_t(L)$. In this case we observe the emergence of a consensus state before final time $T= 25$. This is explained by the large amount of leaders, whose mass increases until the total mass is too spread over the domain $\Omega$ (and so the measure $\mathcal{V}(\mu^L_t)$ is above the threshold $\delta_F$). As soon as the threshold is reached, the creation of leaders is stopped and $\sigma_t(F), \sigma_t(L)$ converge to an asymptotic state thanks to the concentration of the total mass.
In Figure \ref{fig4} we show some frames of the time evolution of $\mu_t^F,\mu_t^L$. 

\begin{figure}
\centering
\includegraphics[scale=0.375]{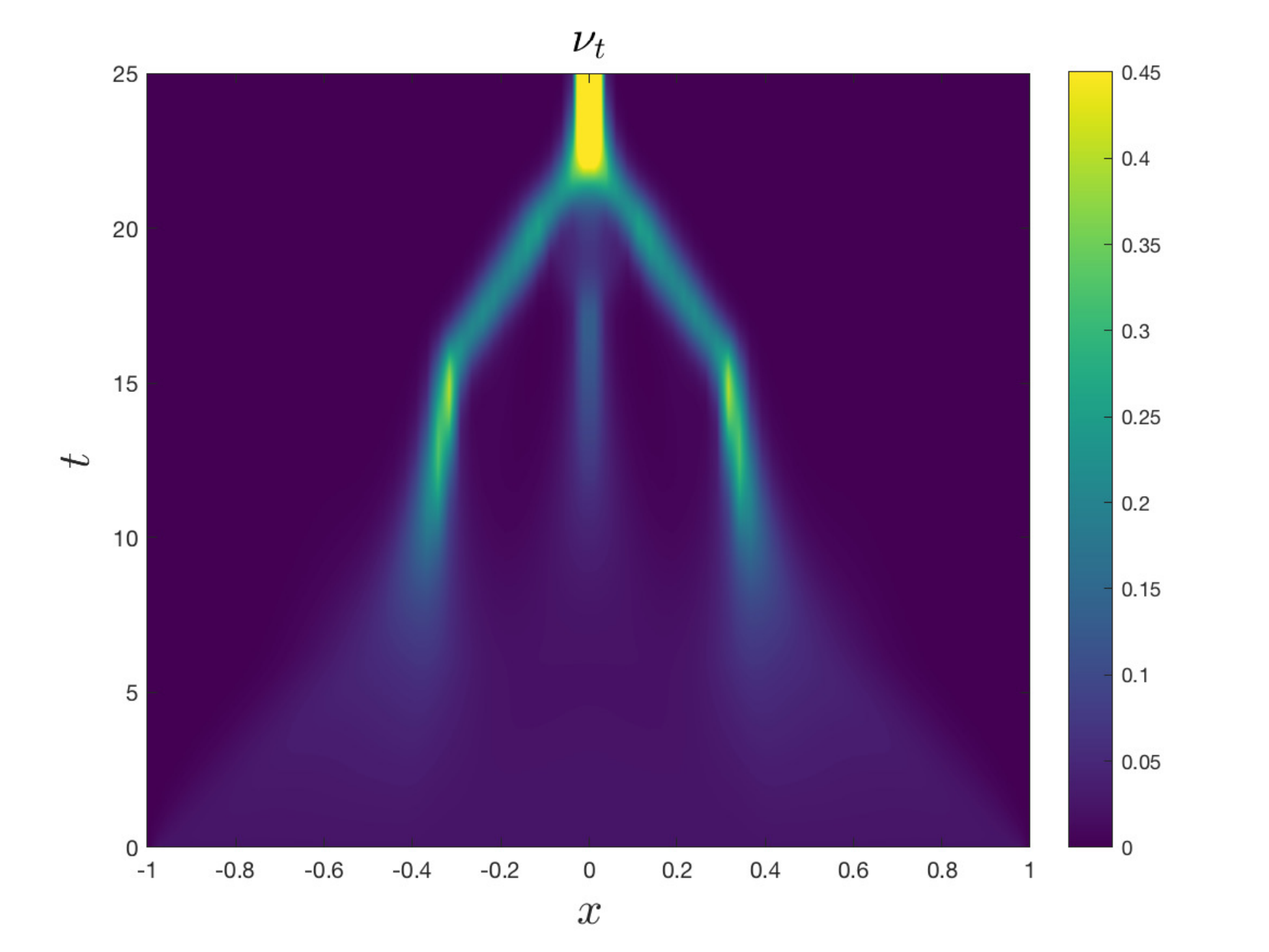}
\qquad
\includegraphics[scale=0.375]{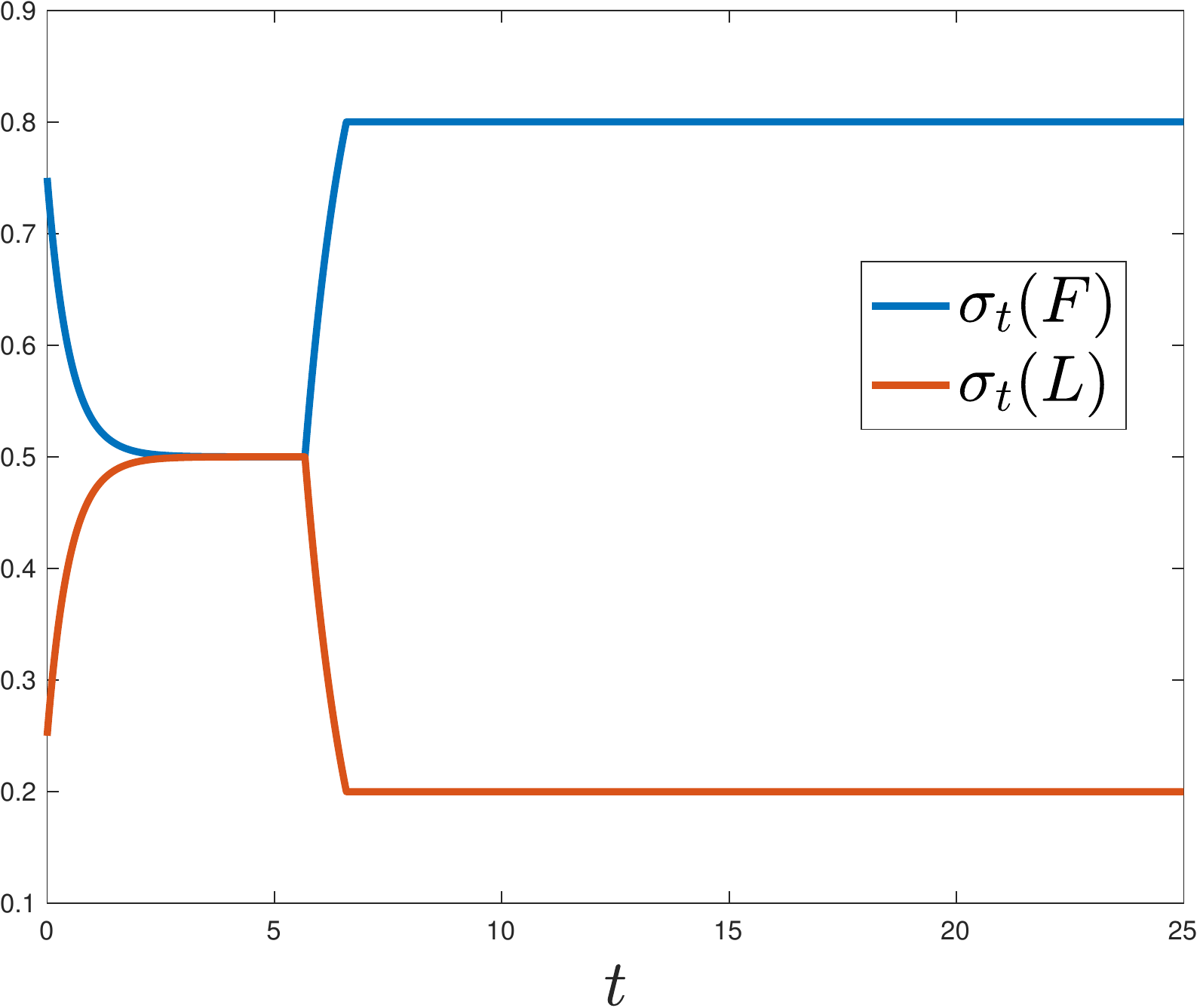}
\caption{{\em Test Ib}. Left: the total density $\nu_t$ in the space time domain $[-1,1]\times[0,T]$. Right: the non-linear evolution of the followers' and leaders' mass, $\sigma_t(F),\sigma_t(L)$.}\label{fig3}
\end{figure}

\begin{figure}
\centering
\includegraphics[scale=0.275]{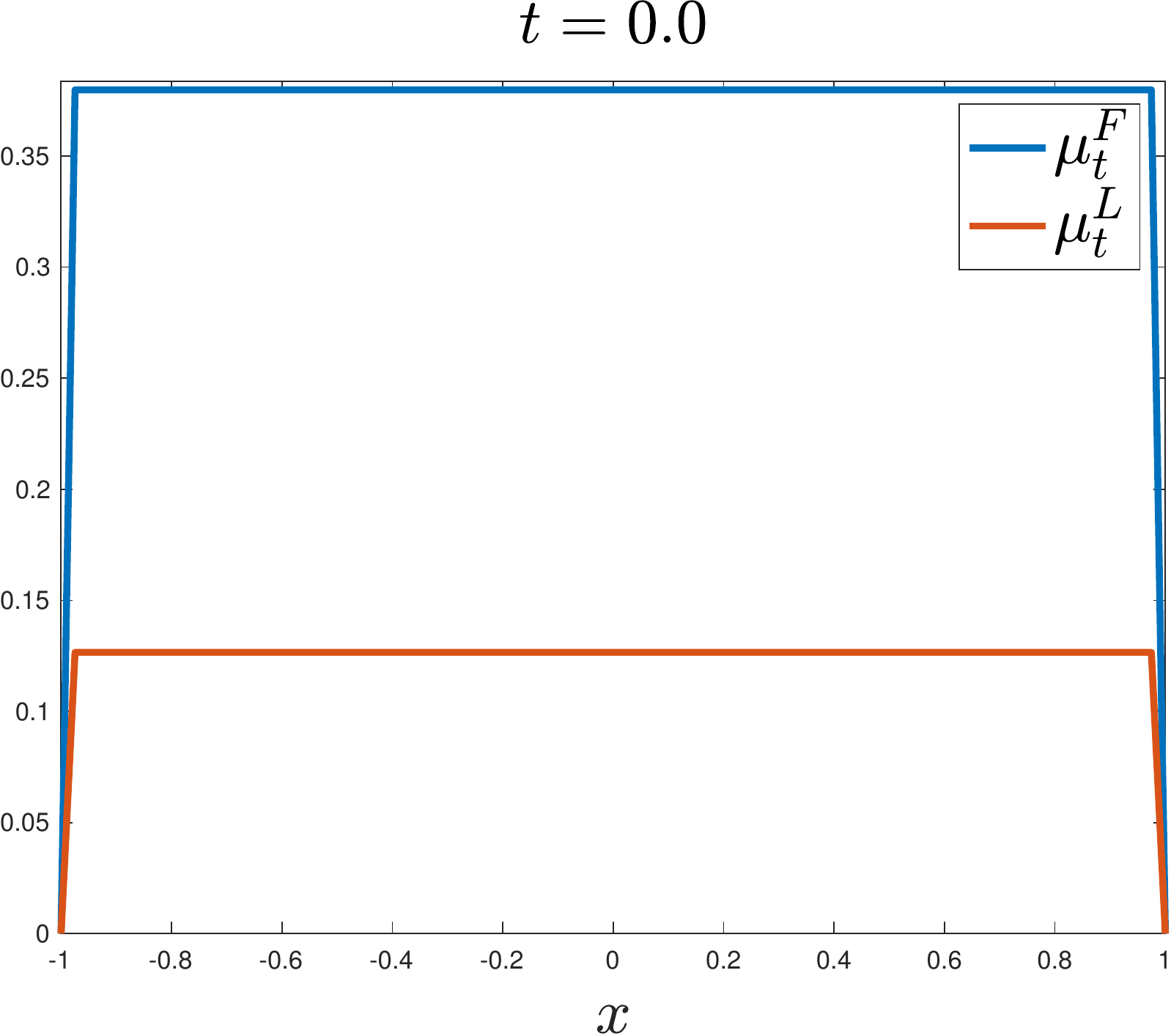}
\includegraphics[scale=0.275]{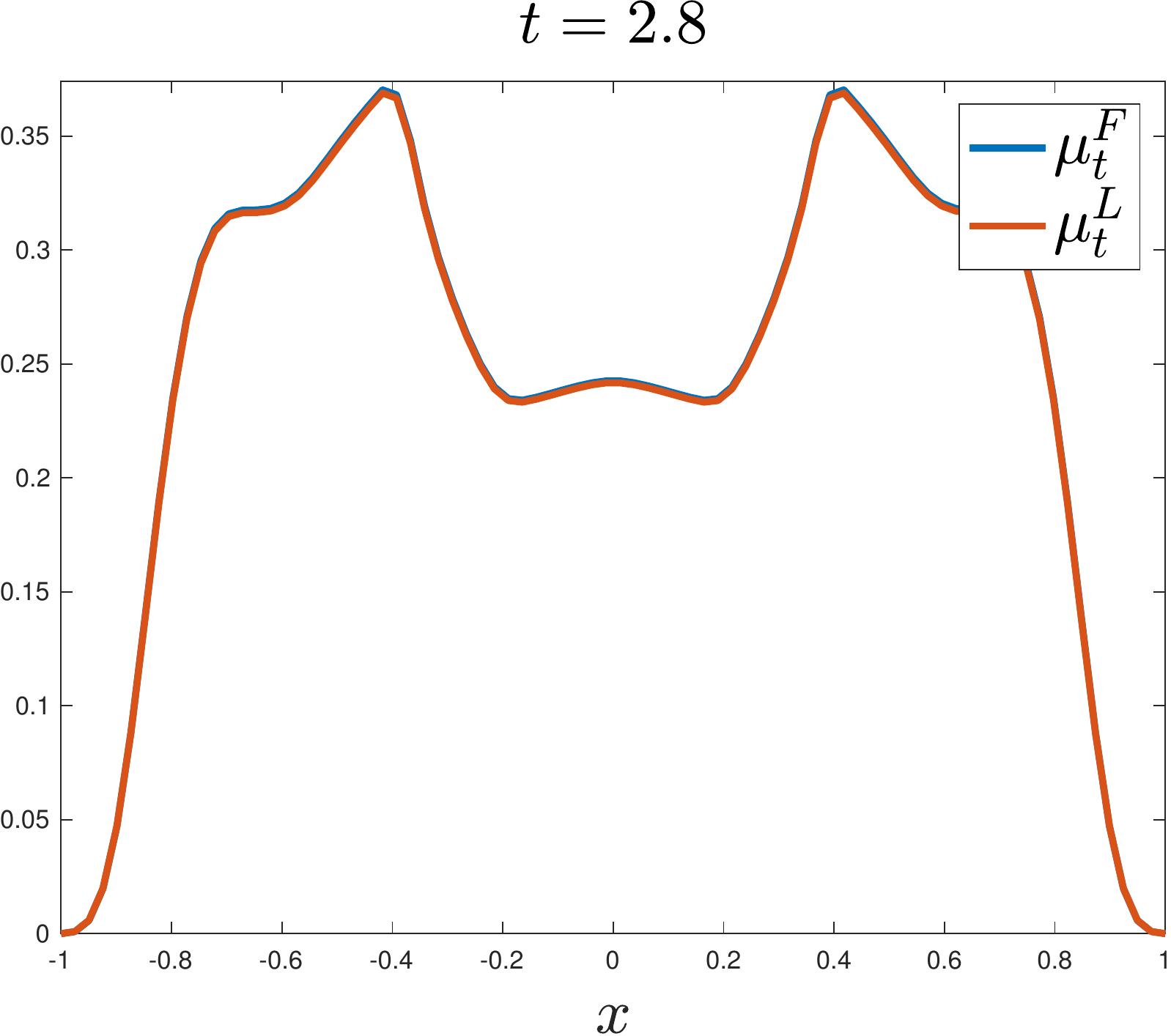}
\includegraphics[scale=0.275]{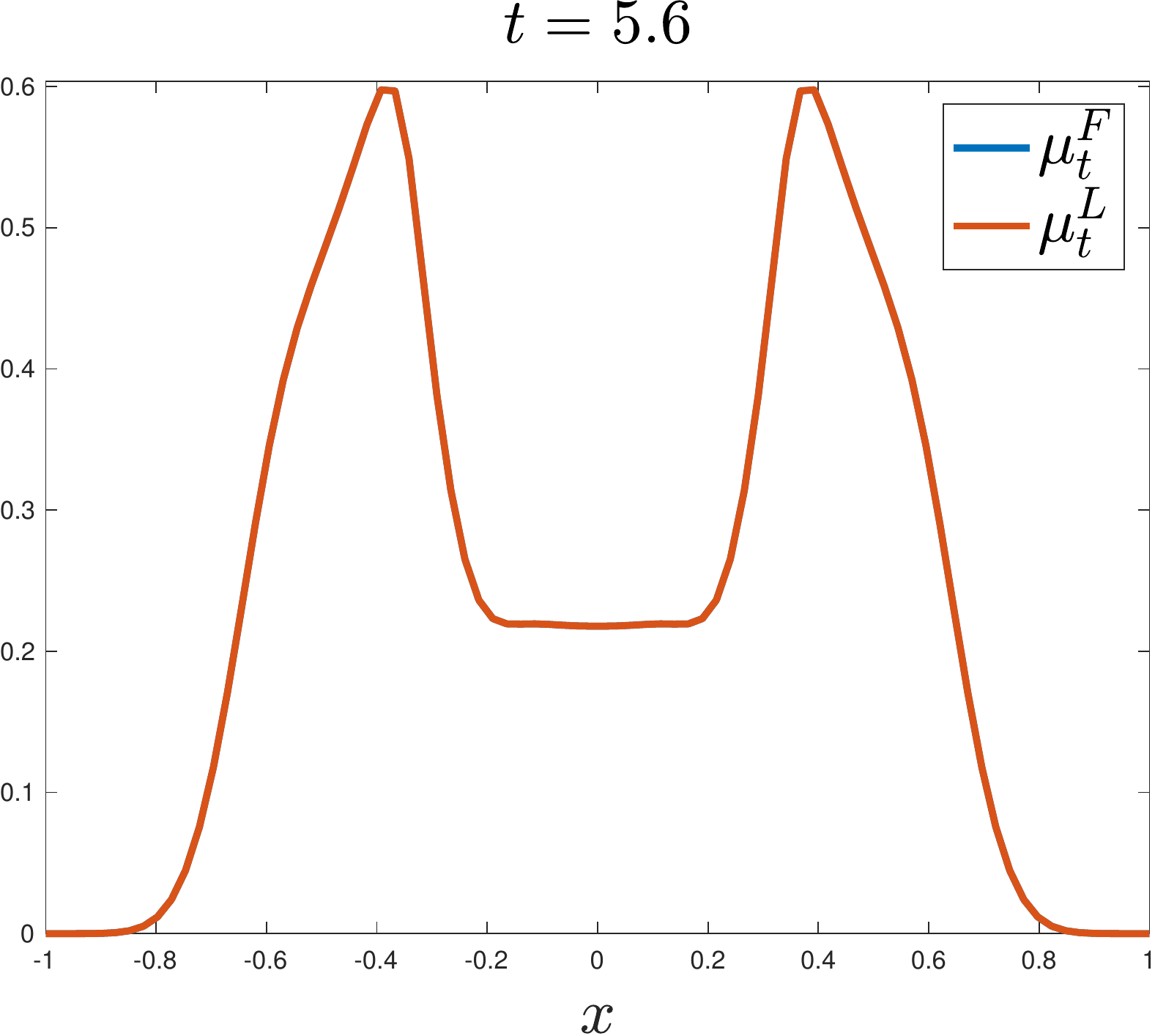}
\\
\includegraphics[scale=0.275]{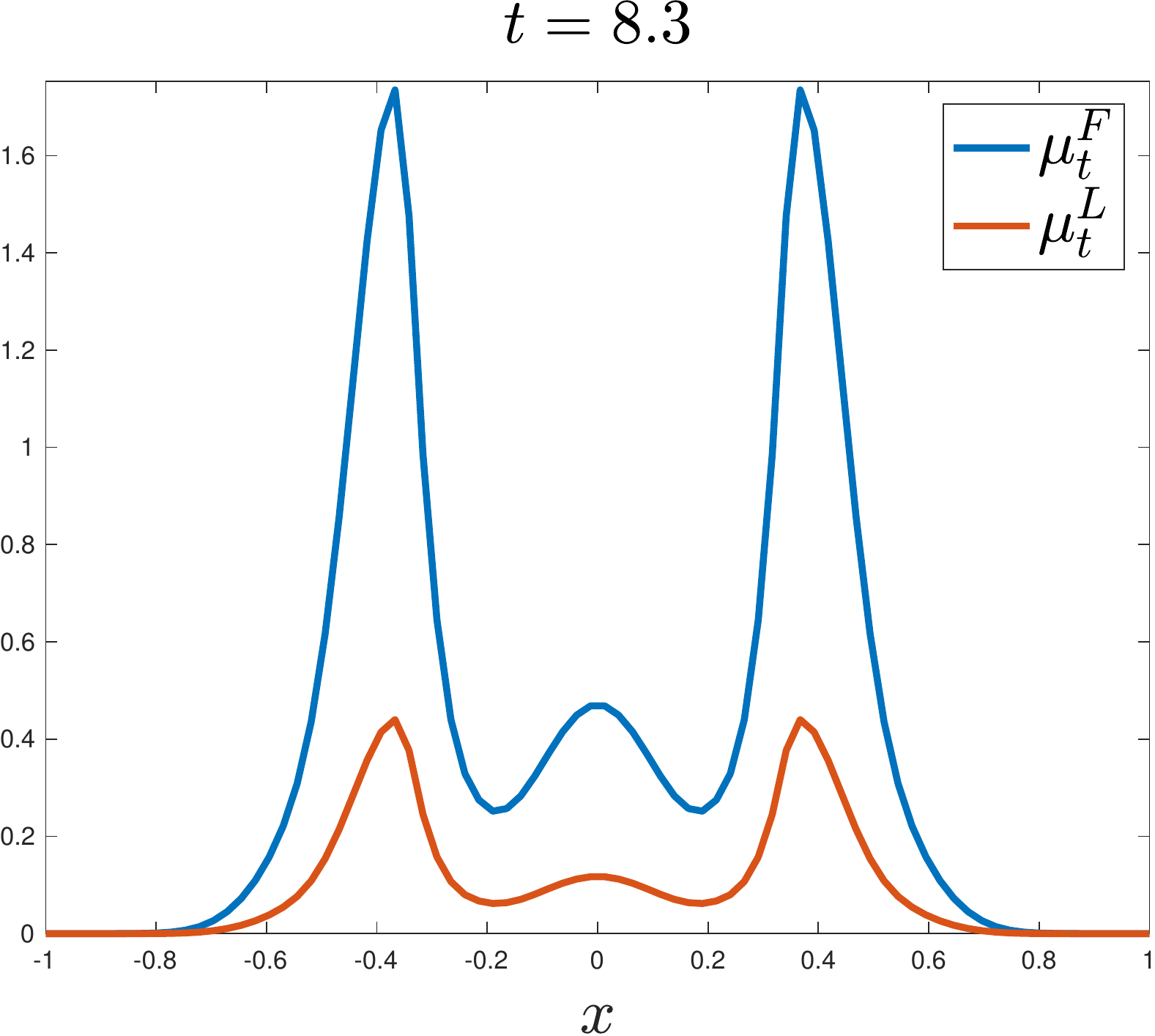}
\includegraphics[scale=0.275]{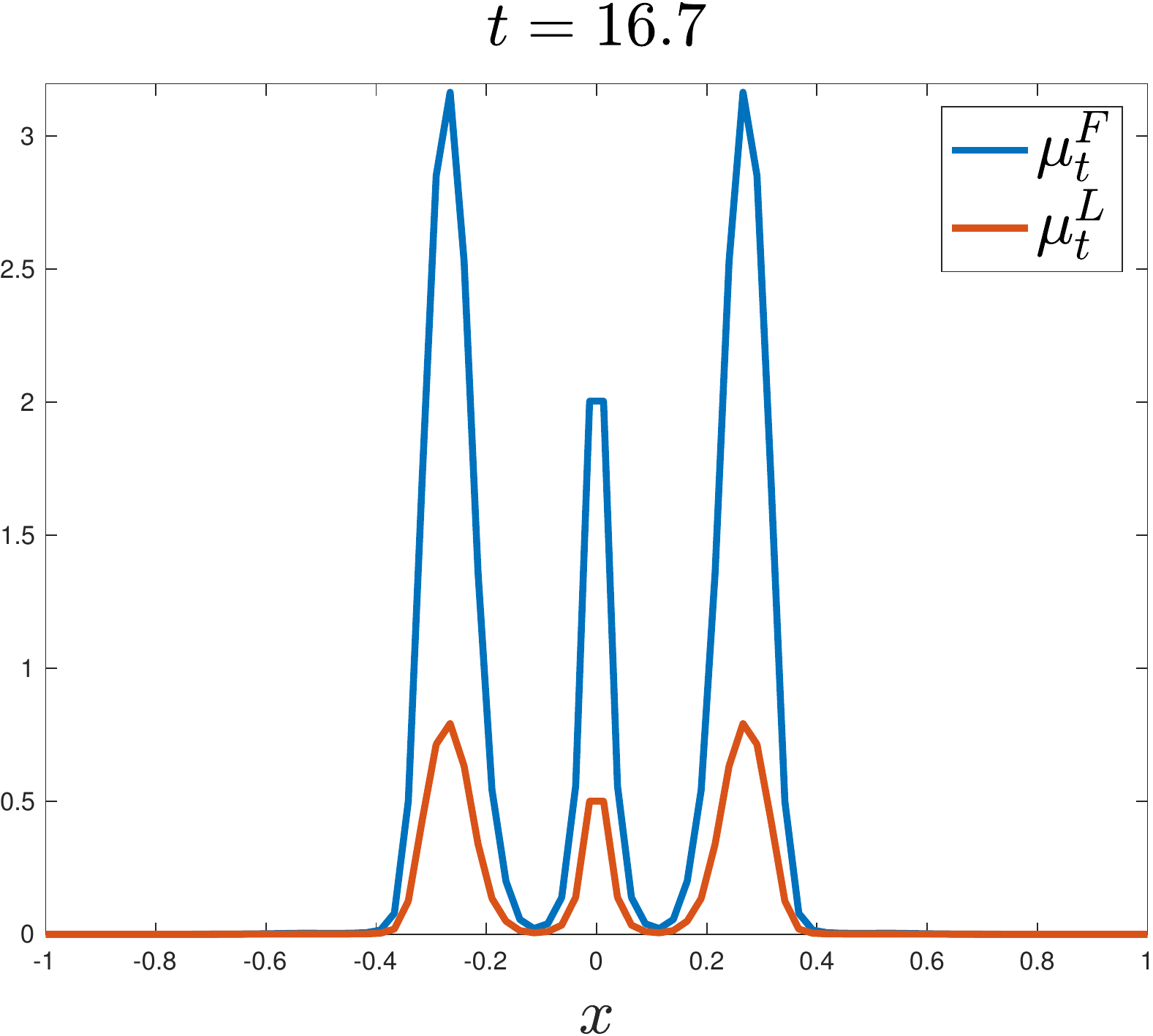}
\includegraphics[scale=0.275]{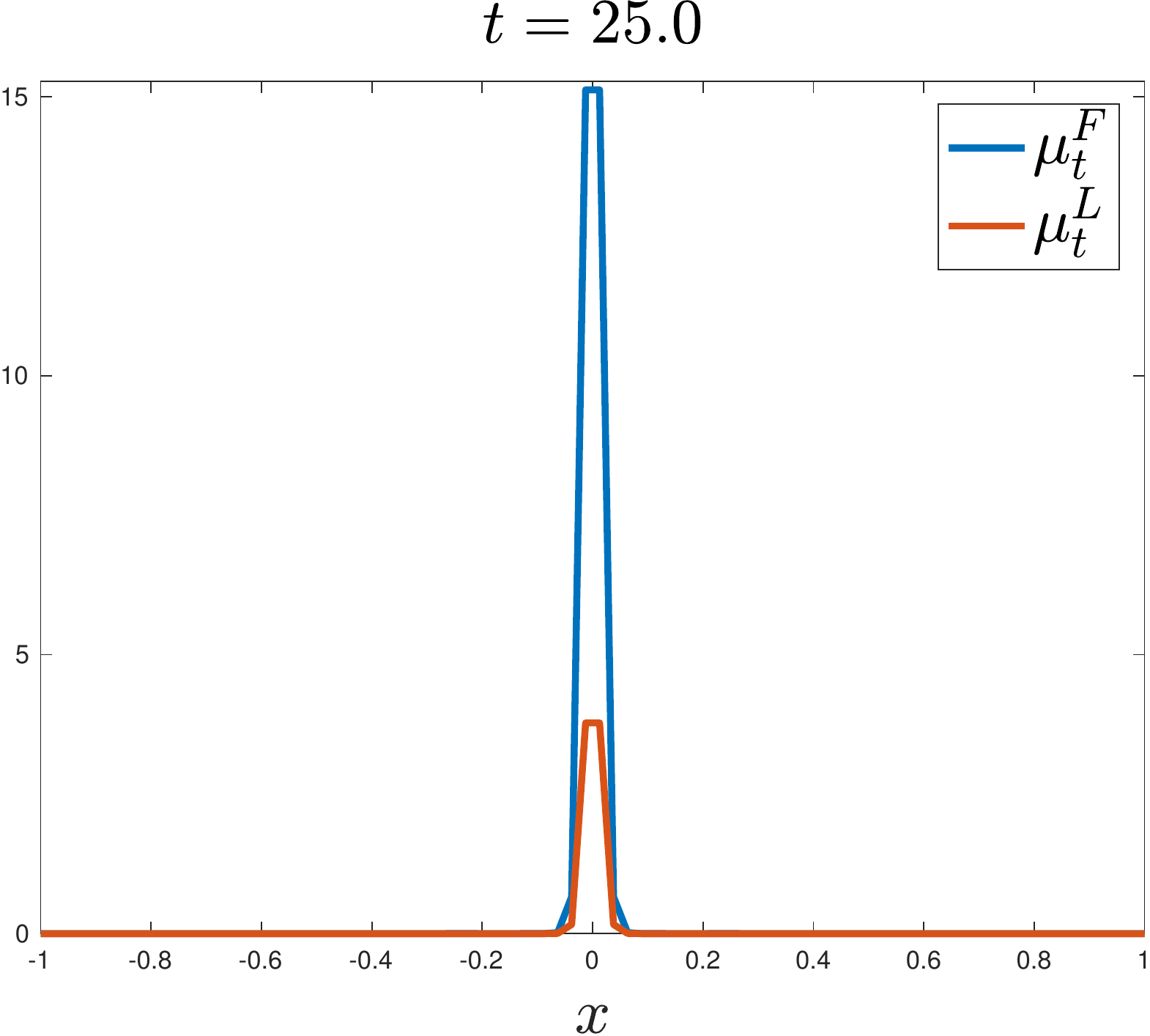}

\caption{{\em Test Ib}. From left to right, and top to bottom row we show the emergence of consensus with leaders' interaction confidence level $C^L=0.6$, and followers' interaction confidence level $C^F=0.2$. Consensus in $x = 0$ at final time is reached.}\label{fig4}
\end{figure}

\subsection{Test II: Aggregation dynamics}
We consider an aggregation dynamics ruled by an attraction towards the population of leaders, and repulsion towards the followers. 
Hence, we assume the following interaction kernels
\begin{align*}%\label{repulsion}
K^F(x) &= a^F(x)x,\qquad a^F(x) = -\frac{\ell_R}{(\epsilon+|x|)^{c_R}},\\
%\label{attraction}
K^L(x) &= a^L(x)x,\qquad a^L(x) = (\epsilon+|x|)^{c_A},
\end{align*}
with non-negative parameters $\ell_R, c_R,c_A$ and $\epsilon = 0.001$. 

The exchange of mass between leaders and followers is described as follows: we consider a constant rate $\alpha_L$, whereas leaders' birth rate $\alpha_F$ depends non-linearly on the followers' density. Similarly to Test I we use the variance measure \eqref{eq:DmuL} for $\mu_t^F$ as follows 
\begin{align}\label{eq:DmuF}
\mathcal{V}(\mu_t^F)=\frac{1}{|\sigma_t(F)|^2}\int_{\Omega\times\Omega} |x-y |^2\,d\mu_t^F(x)\, d\mu_t^F(y).
\end{align}
The birth rate $\alpha_F$ is the switching function \eqref{T1b:alfaF}, modified as follows
\begin{align}\label{T2:alfaF}
\alpha_F(\mu_t^F,\mu_t^L) =\frac{1}{1+e^{c_F(\delta_F-\mathcal{V}(\mu^F_t))}}, 
\end{align}
with $c_F \gg 1$ and $\delta_F\geq 0$.
Hence, we expect the total mass of leaders to increase when the followers' density is too spread over the domain $\Omega$, and to decrease when followers' density is sufficiently concentrated. 

Note that this choice controls the competition between the repulsive action of followers' kernel and the attraction of the leaders' one. In order to show the richness of this setting we consider two different cases. The choice of the parameters are reported in Table \ref{Tab2}.

For the numerical solution of the mean-field dynamics we fix the computational domain $\Omega =[-1,1]$ with zero-flux boundary conditions, discretized with $N=80$ space grid points, and time step $\Delta t = 0.0063$ and final time $T =25$. 

\begin{table}[t]
\centering
\caption{Computational parameters for Test II.}
\begin{tabular}{c|c|c|c|c|c|c|c|c}
\hline
Test  & $c_A$& $c_R$ & $\ell_R$ & $\alpha_F$  &$\delta_F$& $\alpha_L$&$\sigma_0(F)$ &$\sigma_0(L)$ \\
\hline
\hline
 IIa & 3 & 0.75  &0.1& \eqref{T2:alfaF} &0.15& 0.25&0.75&0.25 \\
 IIb & 2 & 0.5&0.1& \eqref{T2:alfaF}&0.2& 0.25&0.75&0.25 \\
\hline
\end{tabular}
\label{Tab2}
\end{table}

\vspace{+0.25cm}
\noindent

\vspace{+0.5cm}
\noindent
{\em Test IIa: Uniform initial data.} 
We consider an initial configuration where leaders and followers occupy the same domain's portion identified by  the function
\[
h(x) = \frac{1}{u-d}\chi_{\left[-\frac{u}{2},-\frac{d}{2}\right]\cup\left[\frac{d}{2},\frac{u}{2}\right]}(x)
\]
with $d = 0.3$ and $u=1.3$. The initial data of \eqref{eq:macroleadfollstrong} is defined as follows
\begin{align}\label{initialdataT2a}
\mu^F_0(x) =\sigma_0(F)h(x),\quad \mu^L_0(x) = \sigma_0(L)h(x).
\end{align}

We report in Figure \ref{fig:Test2a} the evolution of the system, observing an oscillating behavior of the total mass of leaders and followers towards a stable configuration of the densities' profiles. Indeed, initially the density of leaders increases to balance the spread of the initial mass \eqref{initialdataT2a}, up to the moment when the birth rate $\alpha_F$ is switched off. Subsequently, the mass of followers starts to increase, together with the intensity of the repulsion force. Therefore, the dispersion measure \eqref{eq:DmuF} increases again, until the reactivation of the birth rate function $\alpha_F$. At final time $T = 25$, the system has reached a stationary configuration of the densities $\mu^L_t,\mu^F_t$ as well as of the total masses $\sigma_t(L),\sigma_t(F)$.

\begin{figure}
\centering
\includegraphics[scale=0.375]{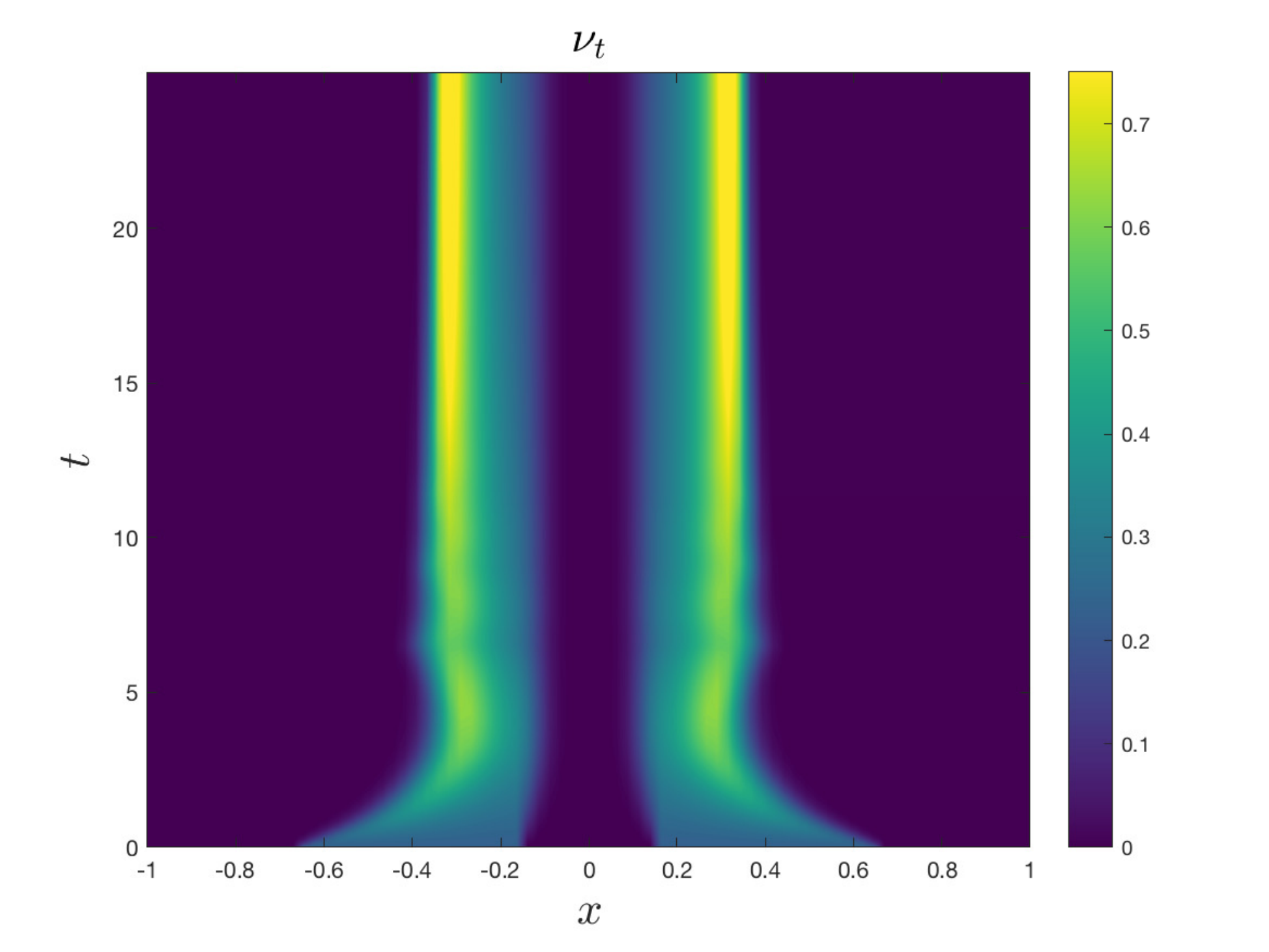}
\qquad
\includegraphics[scale=0.375]{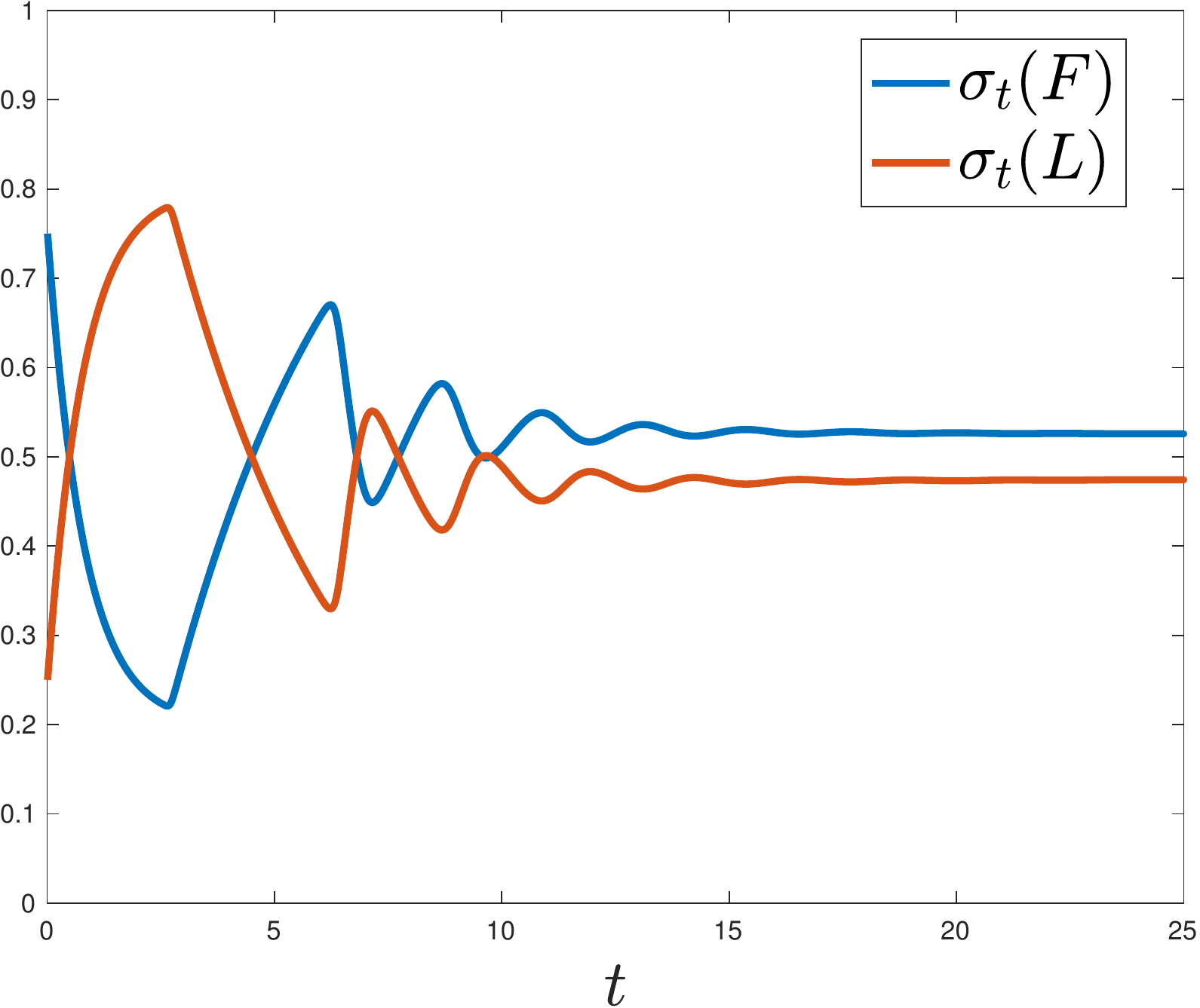}\\
%\caption{{\em Test IIa}. On the right-hand side total density $\nu_t$, and $\sigma_t^F,\sigma_t^L$. }\label{fig5}
%\end{figure}
%
%\begin{figure}
%\centering
\includegraphics[width=3.5cm]{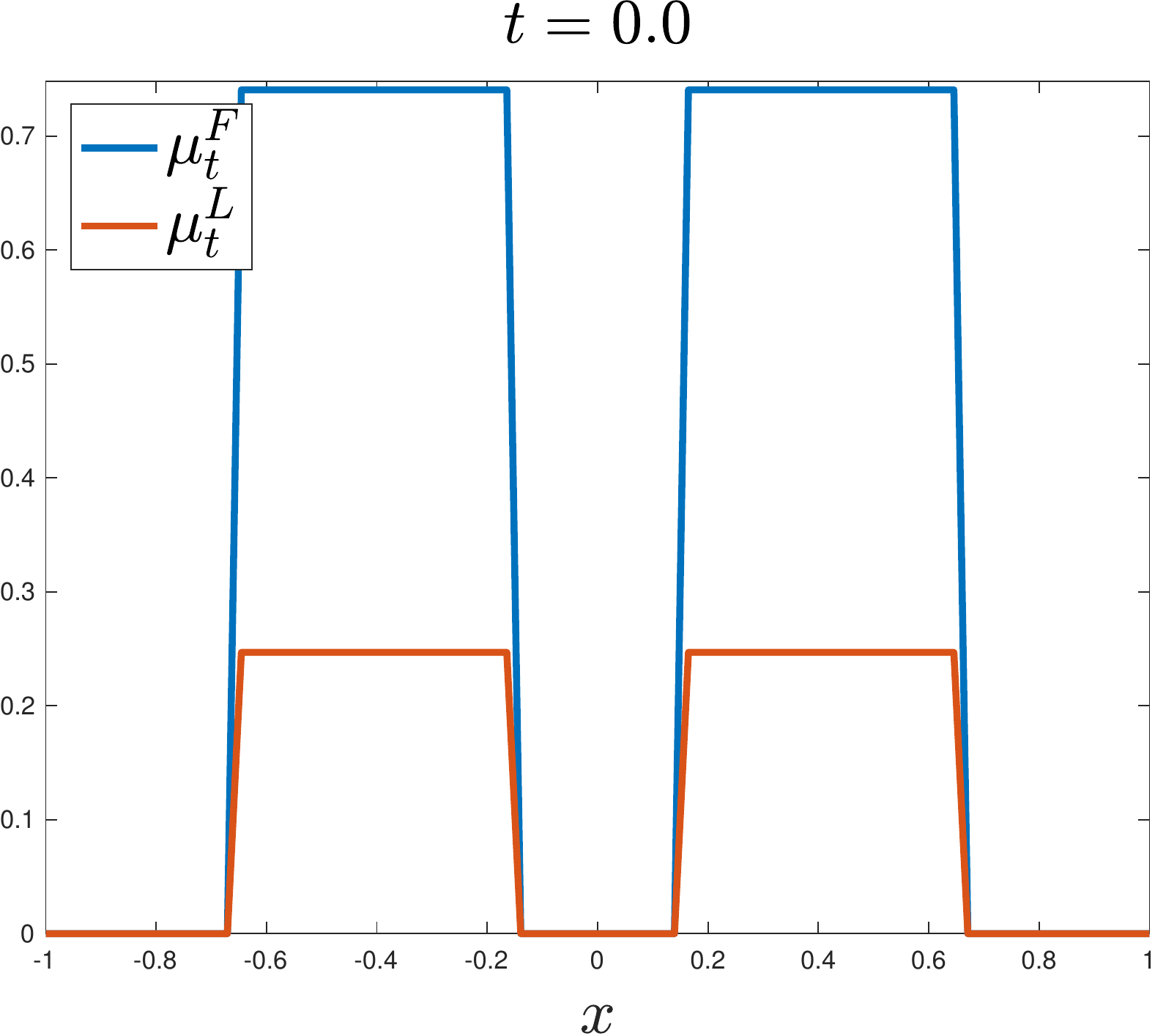}
\includegraphics[width=3.5cm]{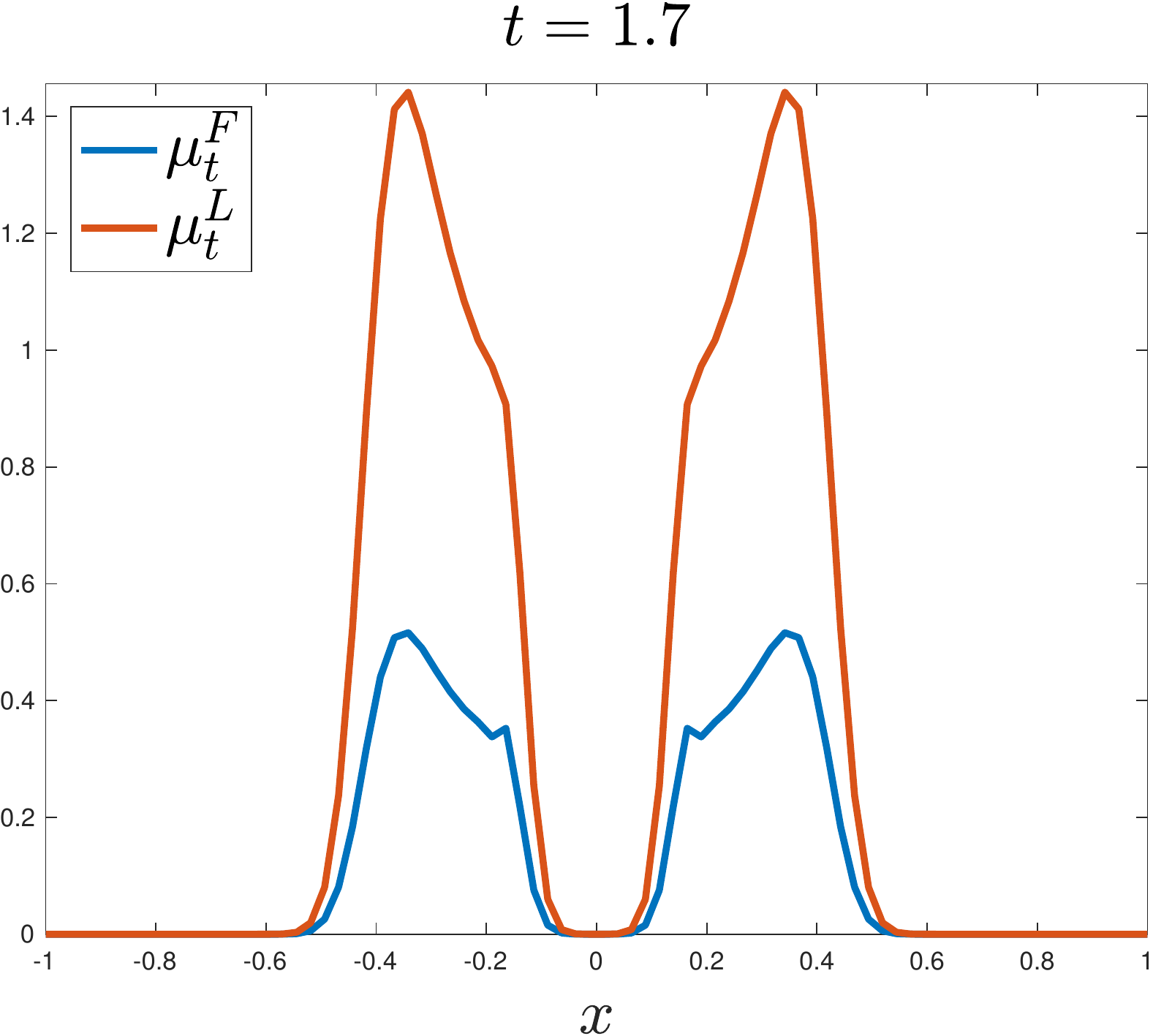}
\includegraphics[width=3.5cm]{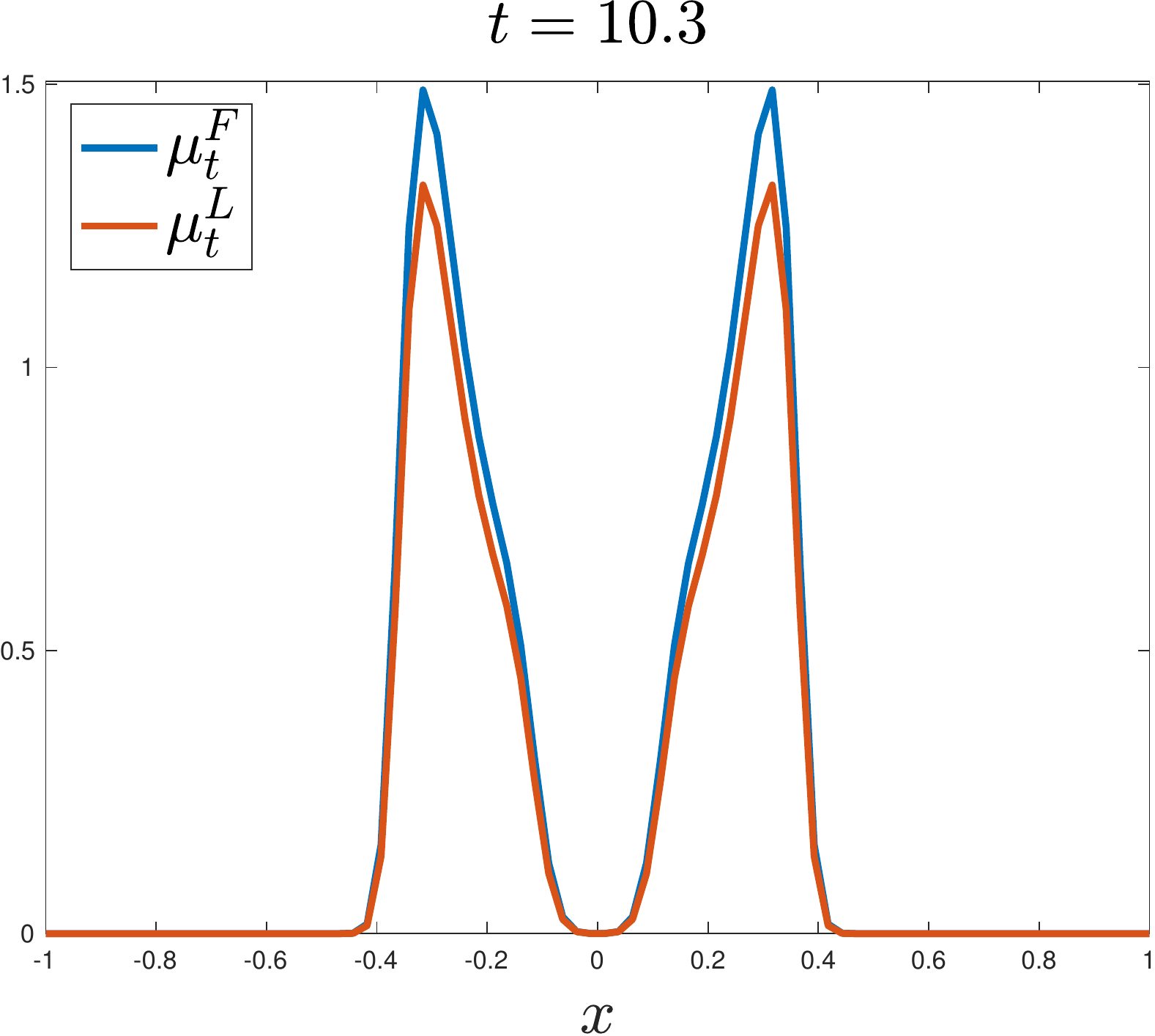}
\includegraphics[width=3.5cm]{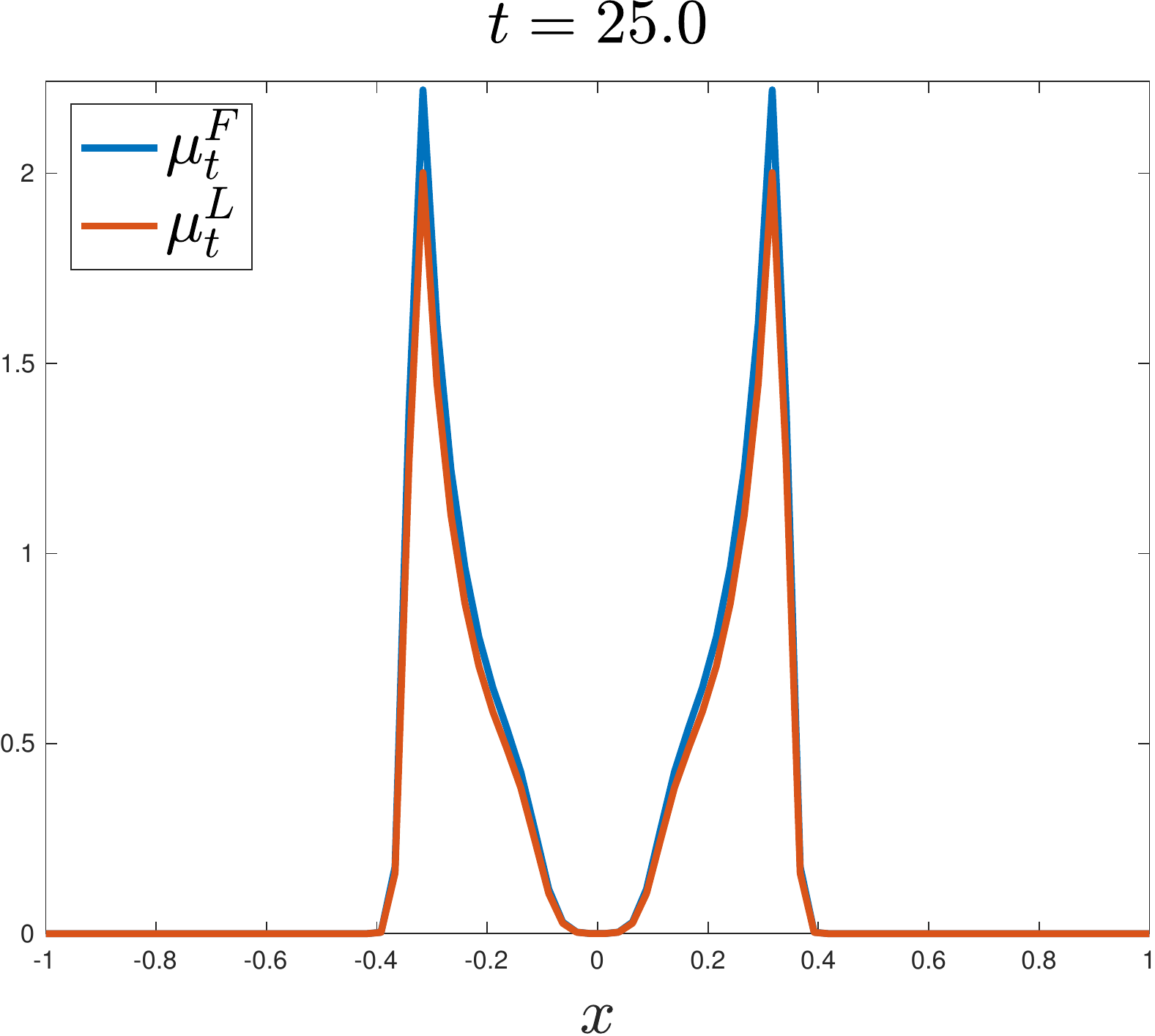}

\caption{{\em Test IIa}. Top line: the left-hand picture shows the total density $\nu_t$, the right-hand picture shows the evolution of the masses $\sigma_t^F,\sigma_t^L$.
Bottom line: From left to right we depict the evolution of the leaders' and followers' densities from the initial data to the final stationary state.}\label{fig:Test2a}
\end{figure}

%%%% TEST IIb
\vspace{+0.5cm}
\noindent
{\em Test IIb: Confinement.} 
We consider a confinement setting, where the leaders' density surrounds the initial density of followers.  In this particular situation, differently from the previous cases, Assumption \ref{ass:posmeas0} on the initial data is not plausible anymore, therefore we renounce to it. We however recall the reader that an existence and uniqueness theory for system \eqref{eq:macroleadfollstrong} is still available, since Propositions \ref{p-ex} and \ref{p-exun1} do not require   \ref{ass:posmeas0} to be fulfilled.

We introduce the Gaussian function 
 \[
G(x;\varsigma^2) =  \frac{1}{\sqrt{2\pi\varsigma^2}}e^{-\frac{x^2}{\varsigma^2}},
\]
then we define the initial data as follows
\begin{align*}%\label{initialdataT2a}
\mu^F_0(x) =\sigma_0(F)G(x;1/30),\qquad \mu^L_0(x) = \frac{\sigma_0(L)}{2}\left(G(x-0.6;1/90) + G(x+0.6;1/90)\right).
\end{align*}
In this setting, the initial dispersion of followers is not large enough to activate the birth rate $\alpha_F$, \eqref{T2:alfaF}. Indeed we can observe from the first two frames of Figure \ref{fig:Test2b}-bottom row that the density of followers starts to grow on the support of $\nu_t$, while the creation of leaders is not inhibited. In a second step, when the interaction becomes too repulsive, the spread of $\mu^F_t$ activates the creation of leaders, and eventually stabilizes the total density towards a stable configuration, with the masses $\sigma_t(L),\sigma_t(F)$ converging towards a stationary value. 
%Figure \ref{fig:Test2bb} reports additional information on the evolution of the densities $\mu^F_t$ and $\mu^L_t$ in $[0,T]\times\Omega$.
\begin{figure}
\centering
\includegraphics[scale=0.375]{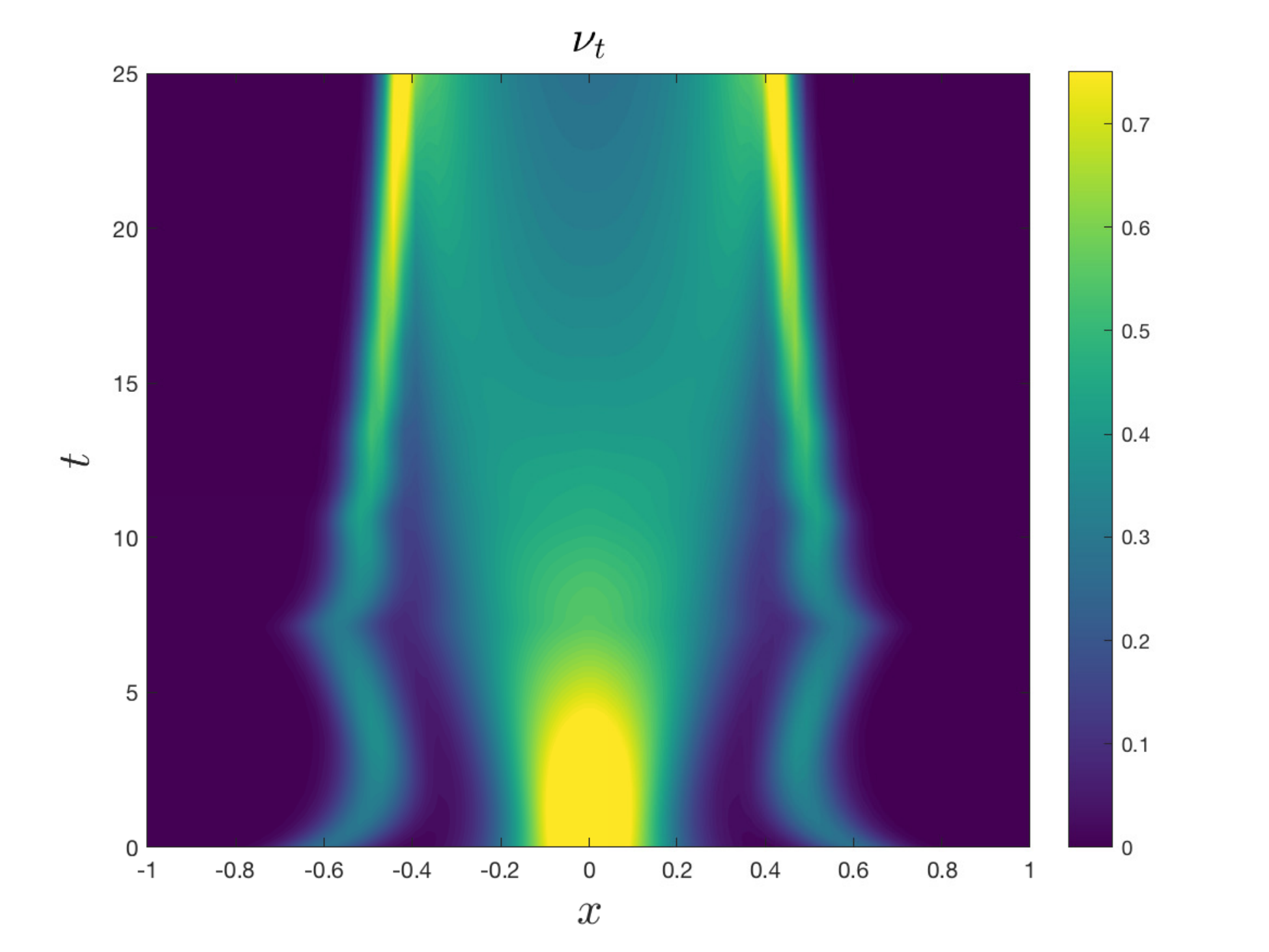}
\includegraphics[scale=0.375]{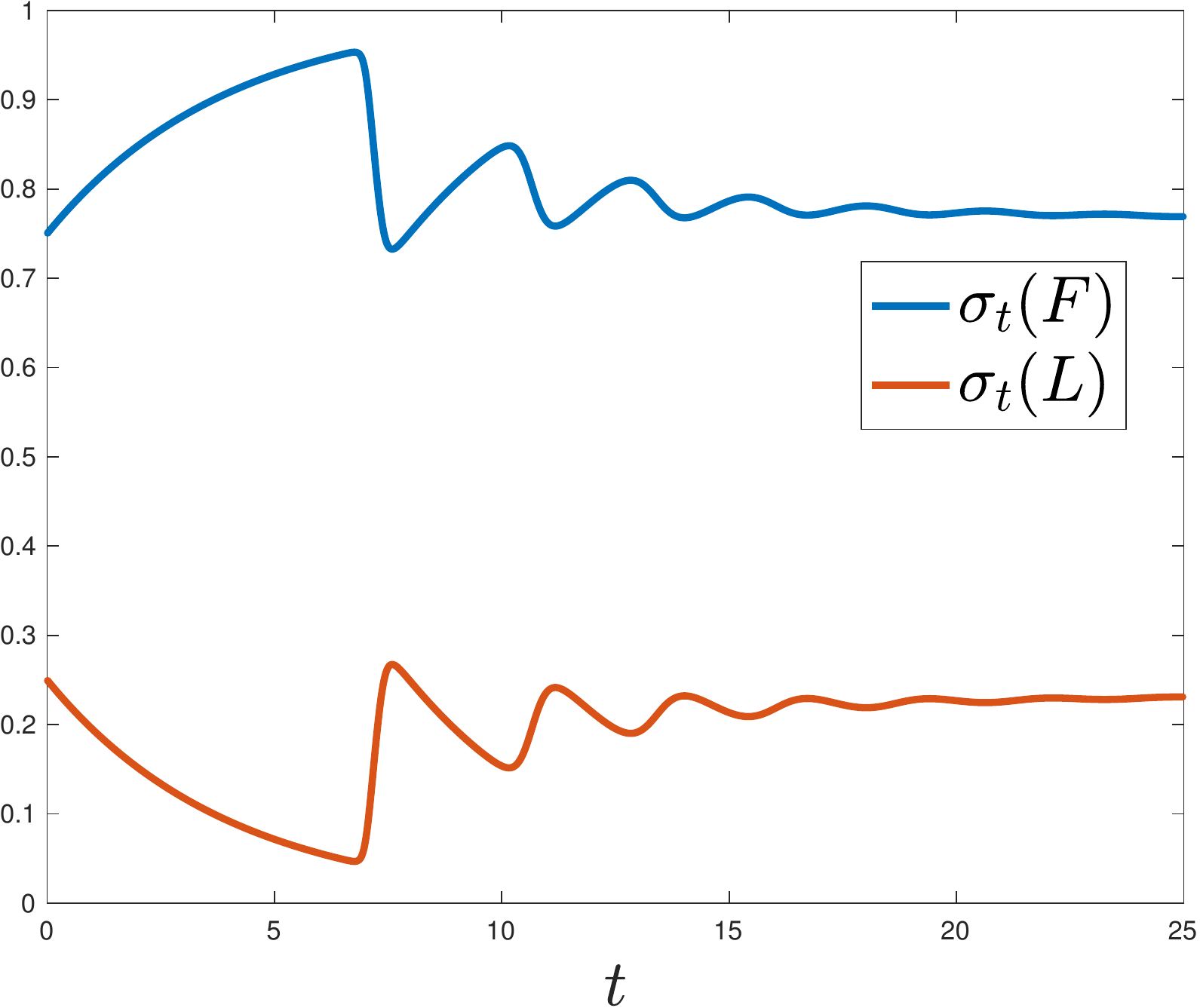}\\
\includegraphics[width=3.5cm]{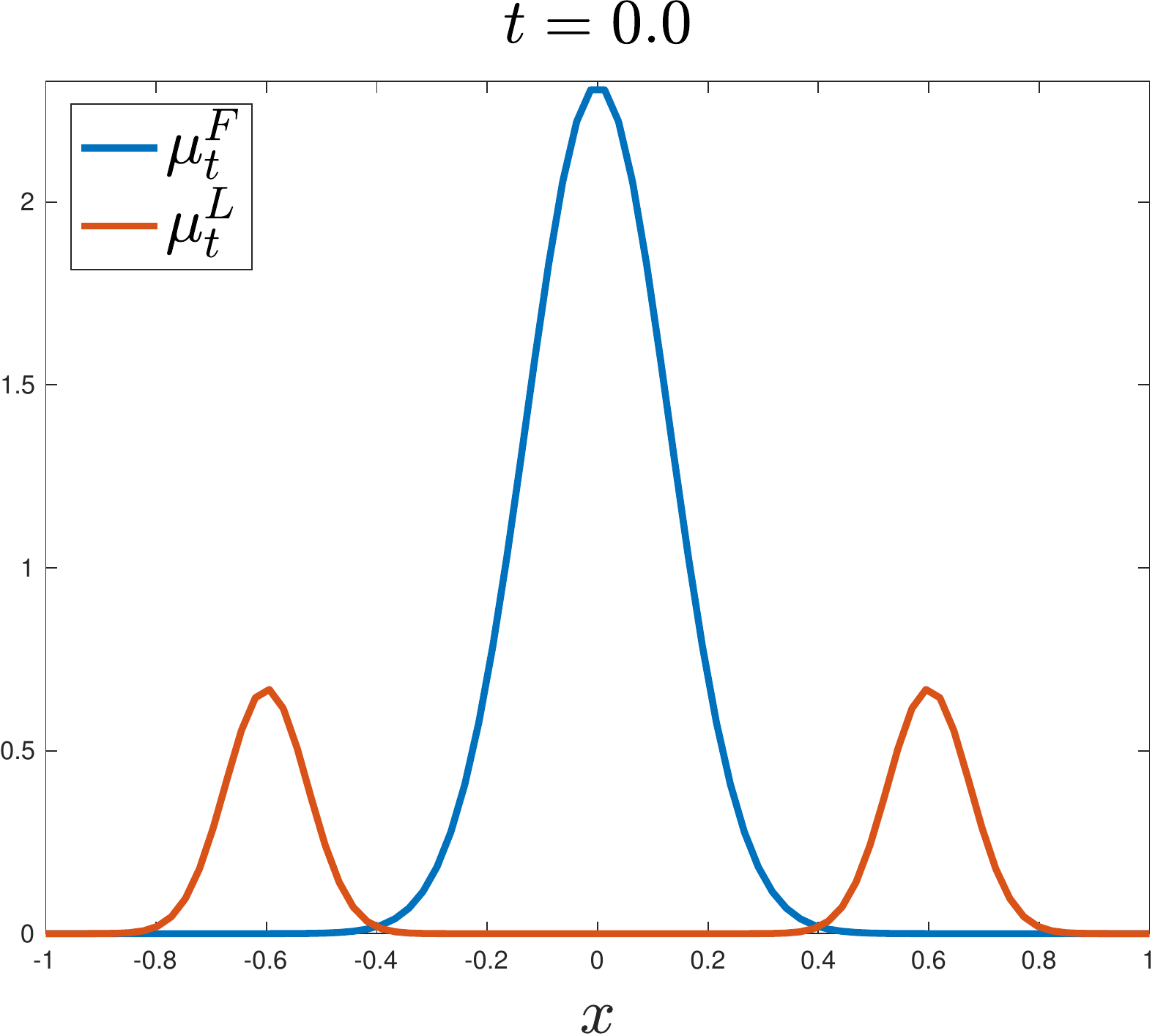}
\includegraphics[width=3.5cm]{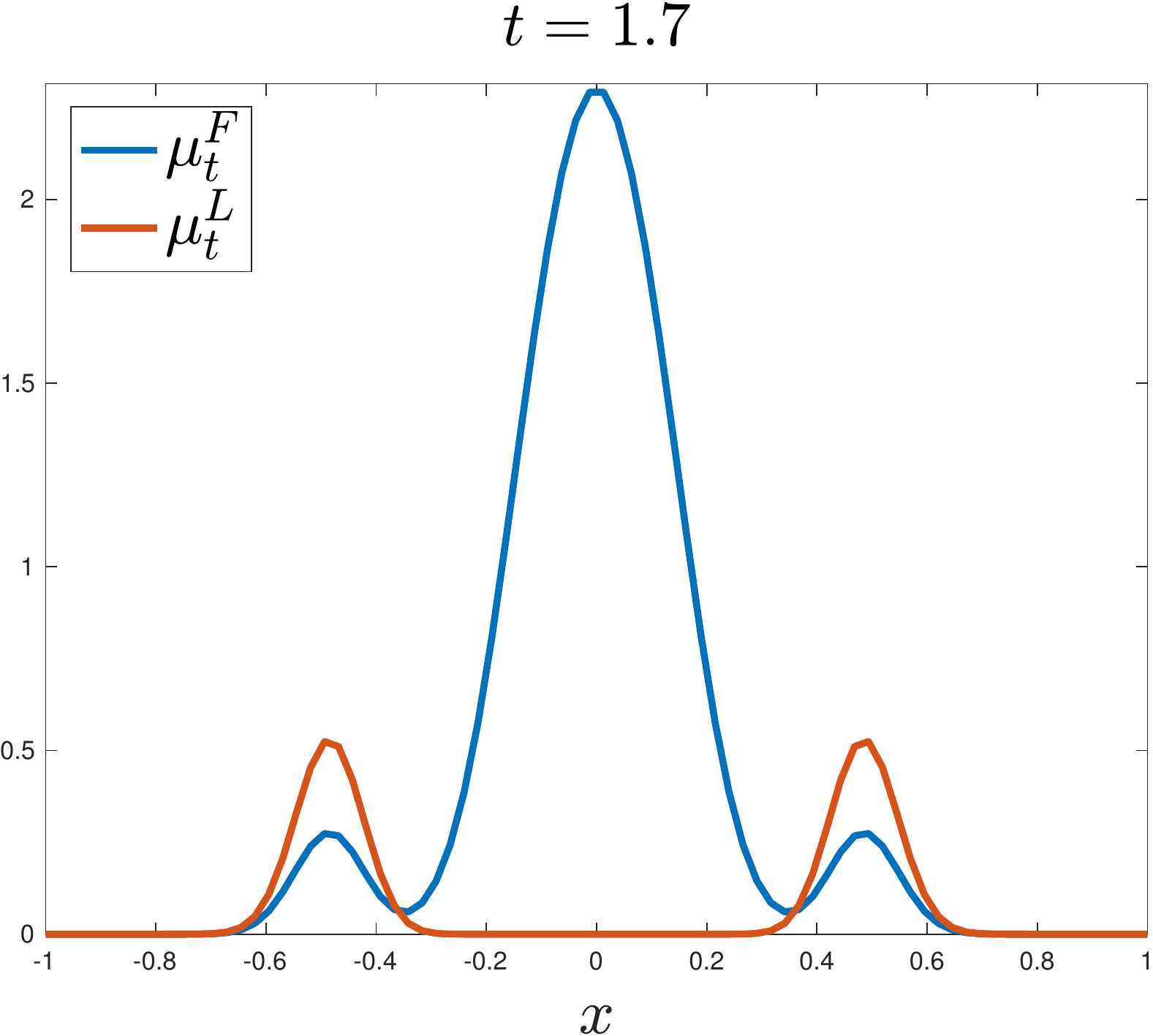}
\includegraphics[width=3.5cm]{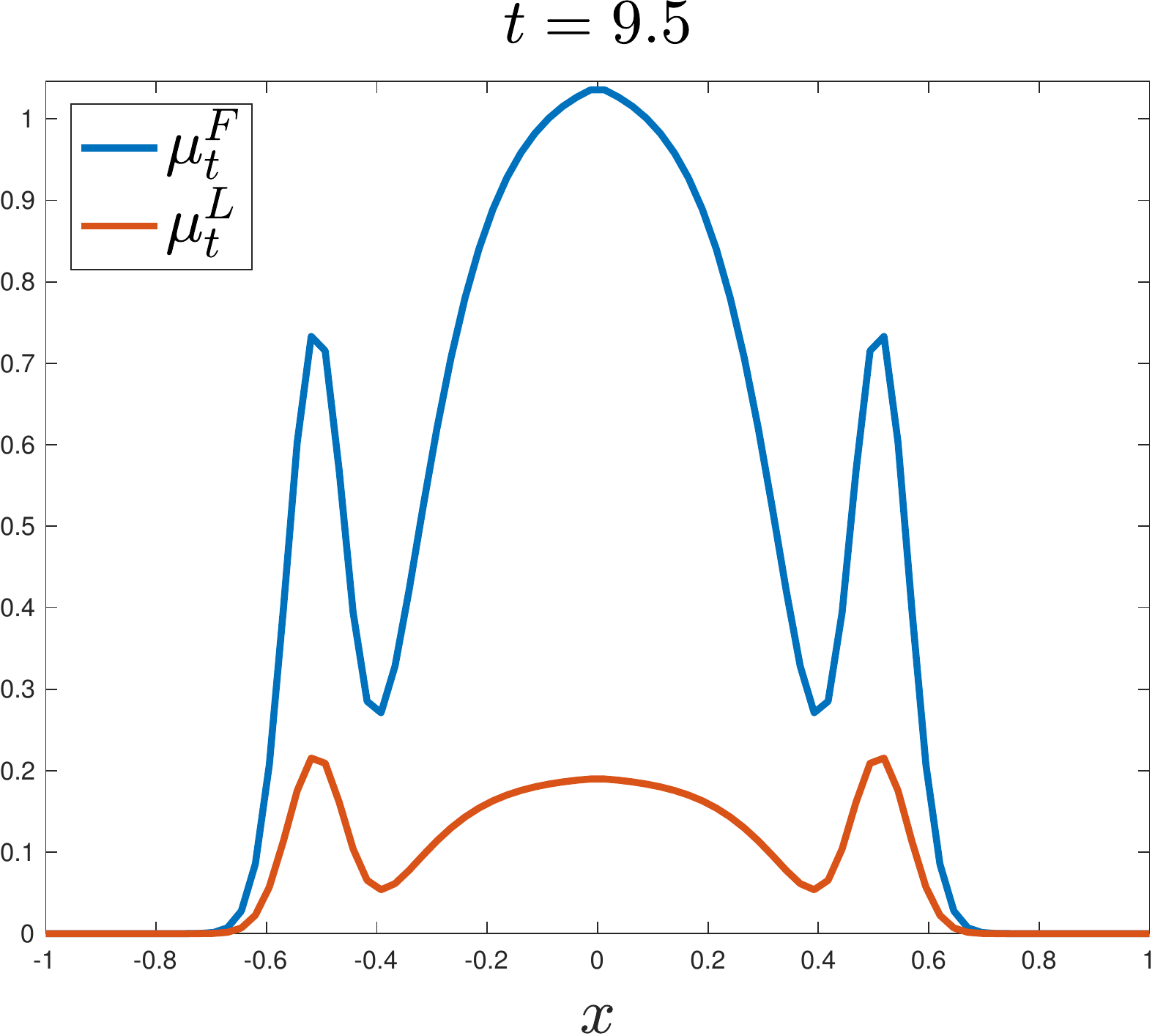}
\includegraphics[width=3.5cm]{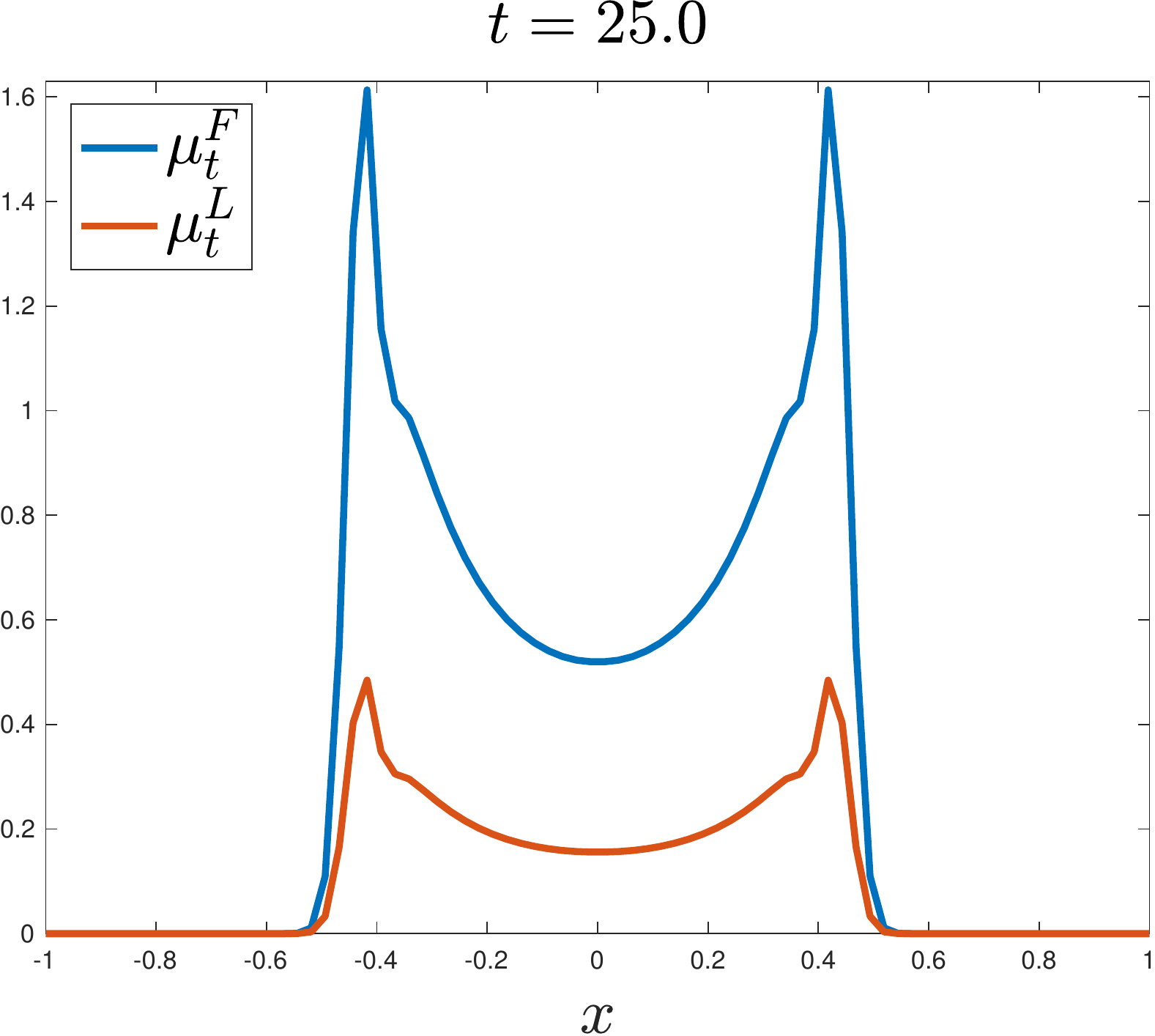}
\caption{{\em Test IIb}. Top line: the left-hand picture shows the total density $\nu_t$, the right-hand picture shows the evolution of the masses $\sigma_t^F,\sigma_t^L$.
Bottom line: From left to right we depict the evolution of the leaders' and followers' densities from the initial data to the final state.}\label{fig:Test2b}
\end{figure}
%\begin{figure}
%\centering
%\includegraphics[scale=0.35]{pics/TestIIb/FF0T_Data_80_aw2.eps}
%\includegraphics[scale=0.35]{pics/TestIIb//LL0T_Data_80_aw2.eps}
%\caption{{\em Test IIb}. On the left  evolution of $\mu^F_t$, on the right  evolution of $\mu^L_t$.}\label{fig:Test2bb}
%\end{figure}

\subsection{Test III: Leaders with steering action }
We study a population of leaders aiming to reach a desired position $\hat x\in \Omega$, and how their motion influences the followers' density. 
The followers' dynamics is governed by an aggregation equation of the type
\begin{align*}%\label{aggreg}
K^F(x) = a^F(x)x,\qquad a^F(x) = (\epsilon+|x|)^{c_A}-\frac{\ell_R}{(\epsilon+|x|)^{c_R}}.
\end{align*}
Leaders have a steering drift towards $\hat x$ of the form
\begin{align*}%\label{steer}
|\mu^L|K^L(x)=\sigma(L)\left(\hat x-x\right)
\end{align*}
which also complies with our  abstract framework, as discussed in Remark \ref{rem:servepersteer}.

We choose a constant rate for the death of leaders $\alpha_L = 0.25$, and the following state-dependent rate for the birth of leaders
\begin{align}\label{T3:alfaF}
\alpha_F(\mu_t^F,\mu_t^L) =\frac{1}{1+e^{c_F(\delta_F-\mathcal{D}(\mu_t^F))}},
\end{align}
with $c_F=1000$, and where $\mathcal{D}(\mu^F_t)$ is the variance of followers' density  with respect to the desired configuration $\hat x$, 
\begin{equation*}%\label{variance}
\mathcal{D}(\mu_t^F) = \frac{1}{|\sigma_t(F)|}\int_ \Omega|\hat x - x|^2d\mu_t^F(x).
\end{equation*}
Hence, we expect the leaders' density to increase when followers are not concentrated around $\hat x$, and to vanish as soon as the desired state is approached.

This test case is inspired by applications in pedestrian dynamics, where a part of the total mass of agents (leaders) is used as control variable to improve the evacuation time of a crowd \cite{albi2015invisible,burger2014mean}.
We remain in a simplified setting: similarly to the previous tests, we solve numerically the evolution of the mean-field interaction dynamics in the one-dimensional domain $\Omega =[-1,1]$ with zero-flux boundary conditions. For the numerical discretization we select $N=80$ space grid points, time step $\Delta t = 0.0127$ and final time $T =15$. We have reported in Table \ref{Tab2} the parameters' choice for the different cases.

\begin{table}[t]
\centering
\caption{Computational parameters for Test III.}
\begin{tabular}{c|c|c|c|c|c|c|c|c|c|c}
\hline
Test  & $c_A$& $c_R$ & $\ell_R$ & $\epsilon$ & $\alpha_F$  &$\delta_F$& $\alpha_L$&$\sigma_0(F)$ &$\sigma_0(L)$ & $\hat x$\\
\hline
\hline
 III& 2 & 1  &0.05& 0.0001& \eqref{T3:alfaF} &0.15& 0.25&0.75&0.25 & 0.5 \\
 %IIb & 2 & 0.5&0.1& \eqref{T2:alfaF}&0.2& 0.25&0.75&0.25 \\
\hline
\end{tabular}
\label{Tab3}
\end{table}

%Finally we remark that the presence of a vector field, such as \eqref{steer},  does not affect the main results of the article. For a general treatment, see \cite{miogw}.\Giac{Francesco R. che referenza avevi in mente qui?}

Figure \ref{fig:Test3} shows the evolution of the density $\nu_t(x)$ and the evolution of the followers' and leaders' masses $\sigma^F_t,\sigma^L_t$ in the top row. Bottom row shows the evolutions of $\mu_t^F,\mu_t^L$: the mass of leaders increases initially since followers are far away from $\hat x = 0.5$, as soon as $\mu^F_t$ approaches $\hat x$, while the density of leaders tends to vanish.
%
%On the other hand the concentration of the density $\nu_t(x)$ is limited by  the repulsion. Figure \ref{fig8} reports the time evolution of the  leaders' and followers' densities.

\vspace{+0.25cm}
\noindent

\begin{figure}
\centering
\includegraphics[scale=0.375]{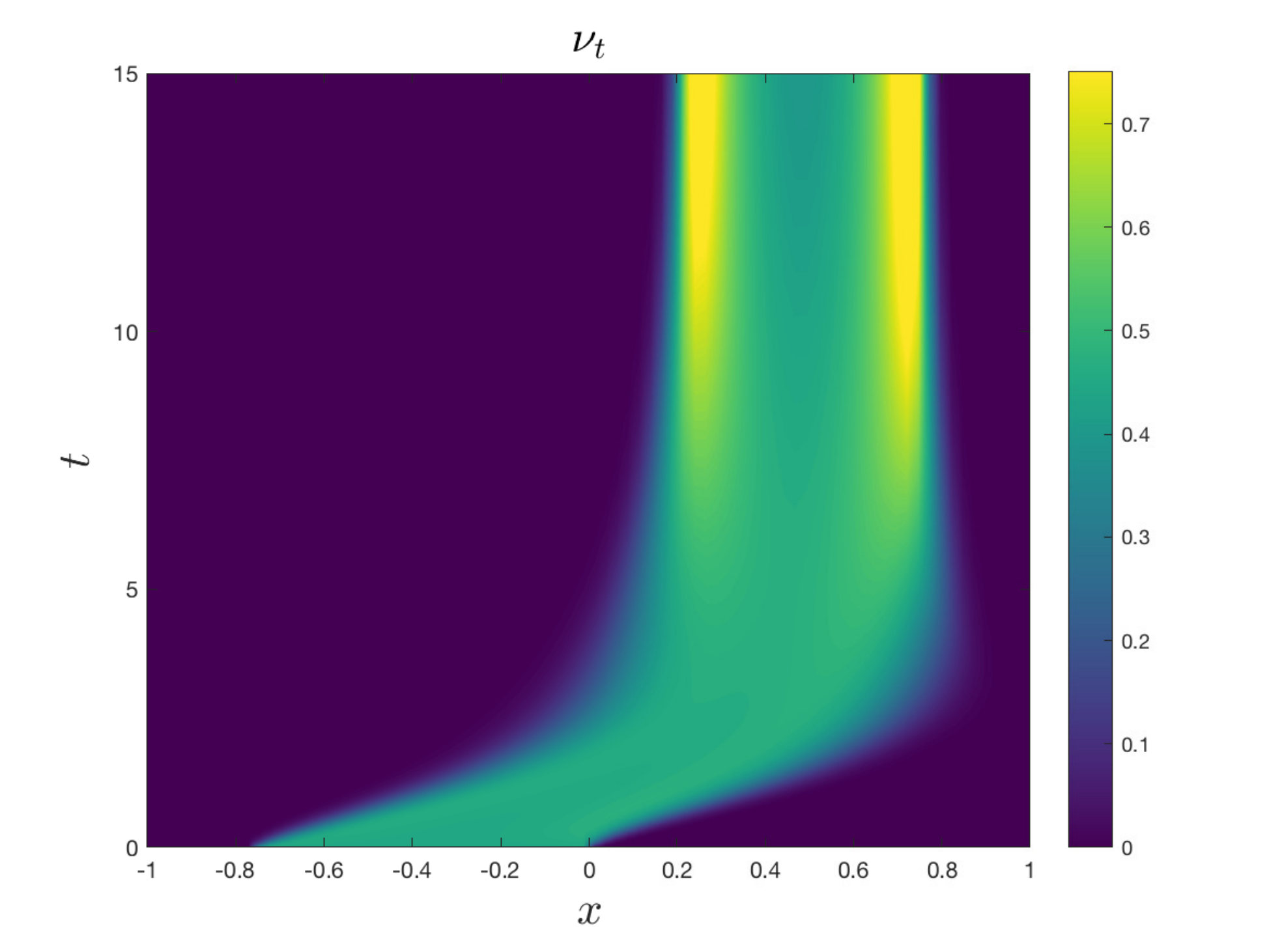}
\qquad
\includegraphics[scale=0.375]{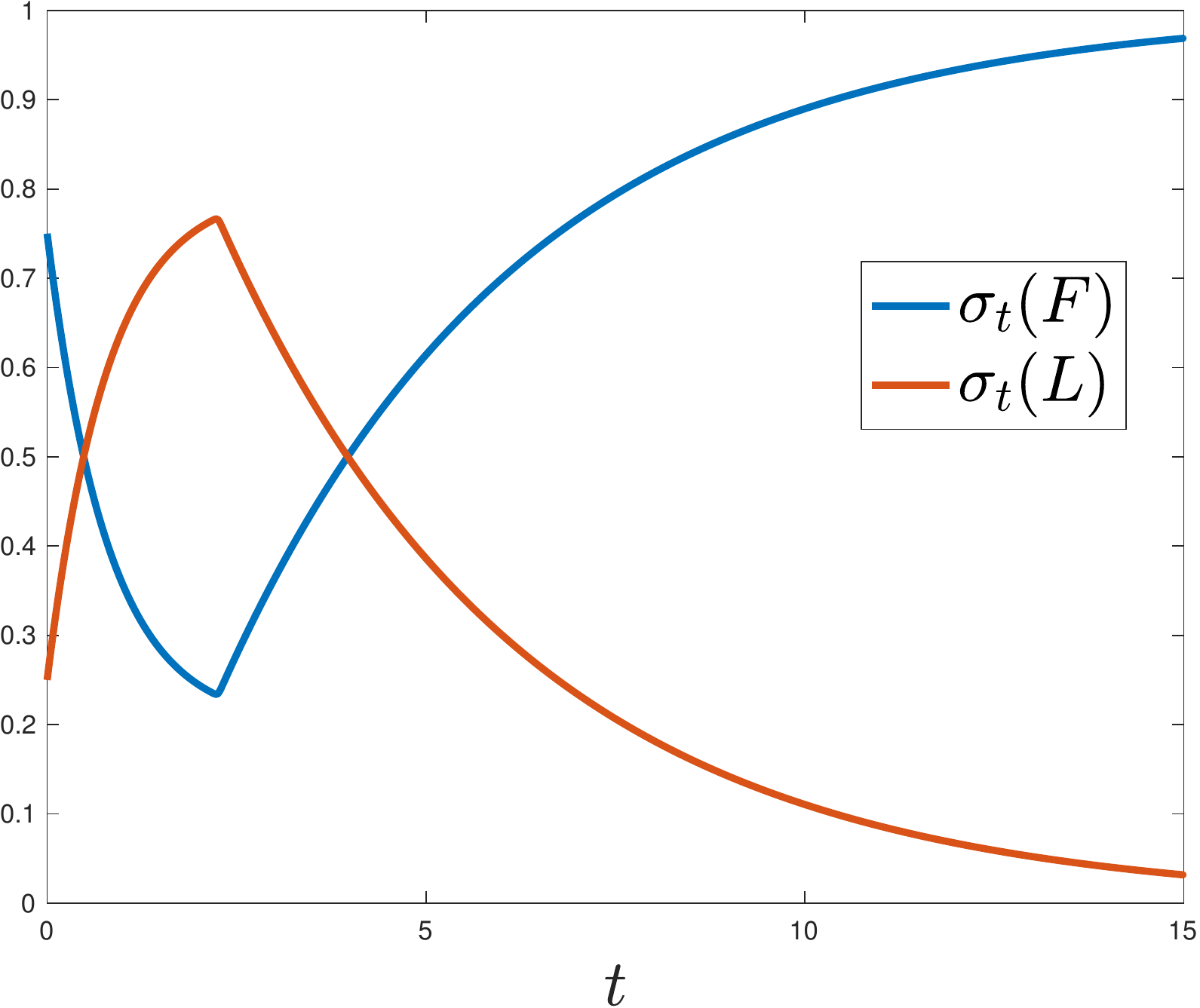}
\\
\includegraphics[width=3.5cm]{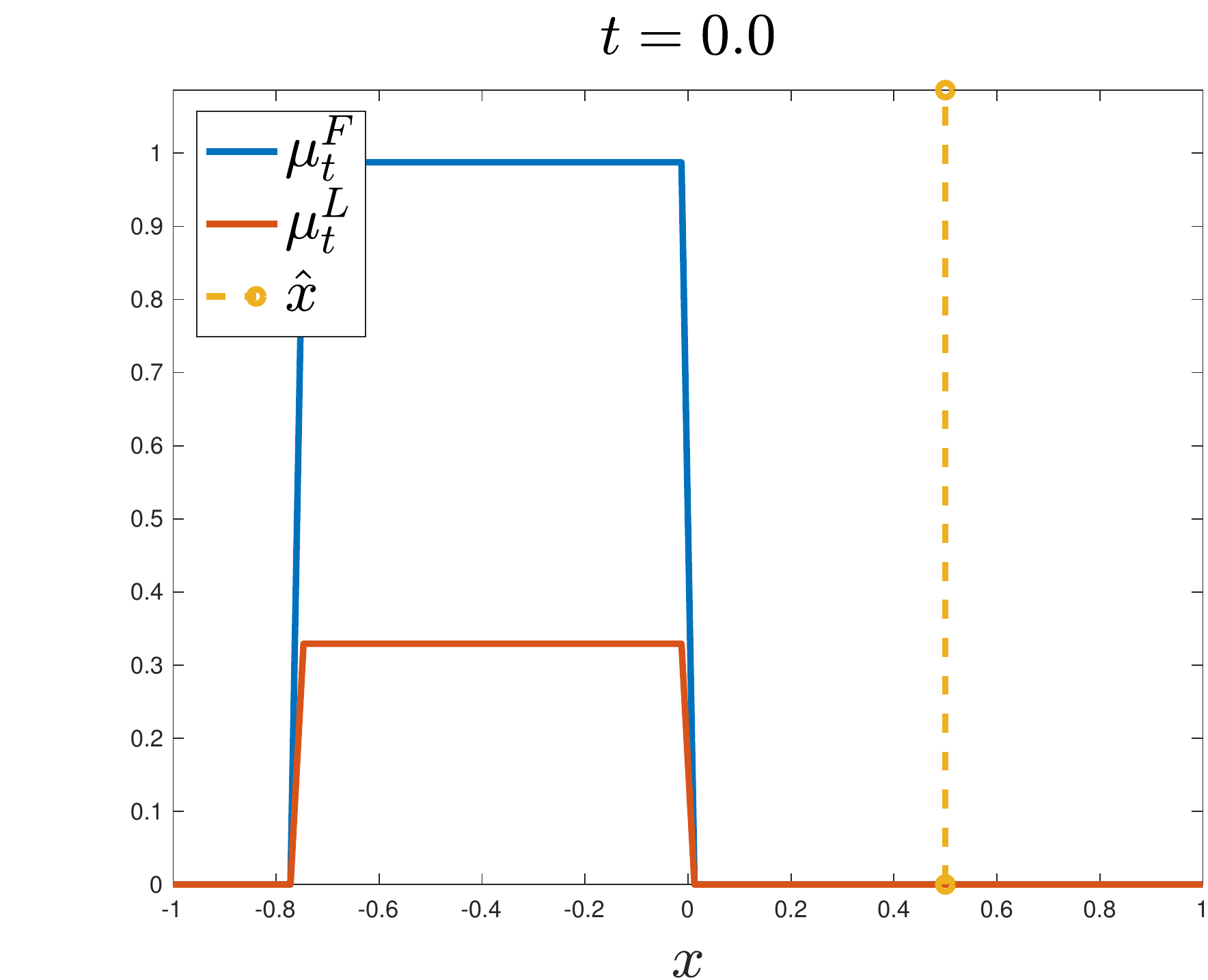}
\includegraphics[width=3.5cm]{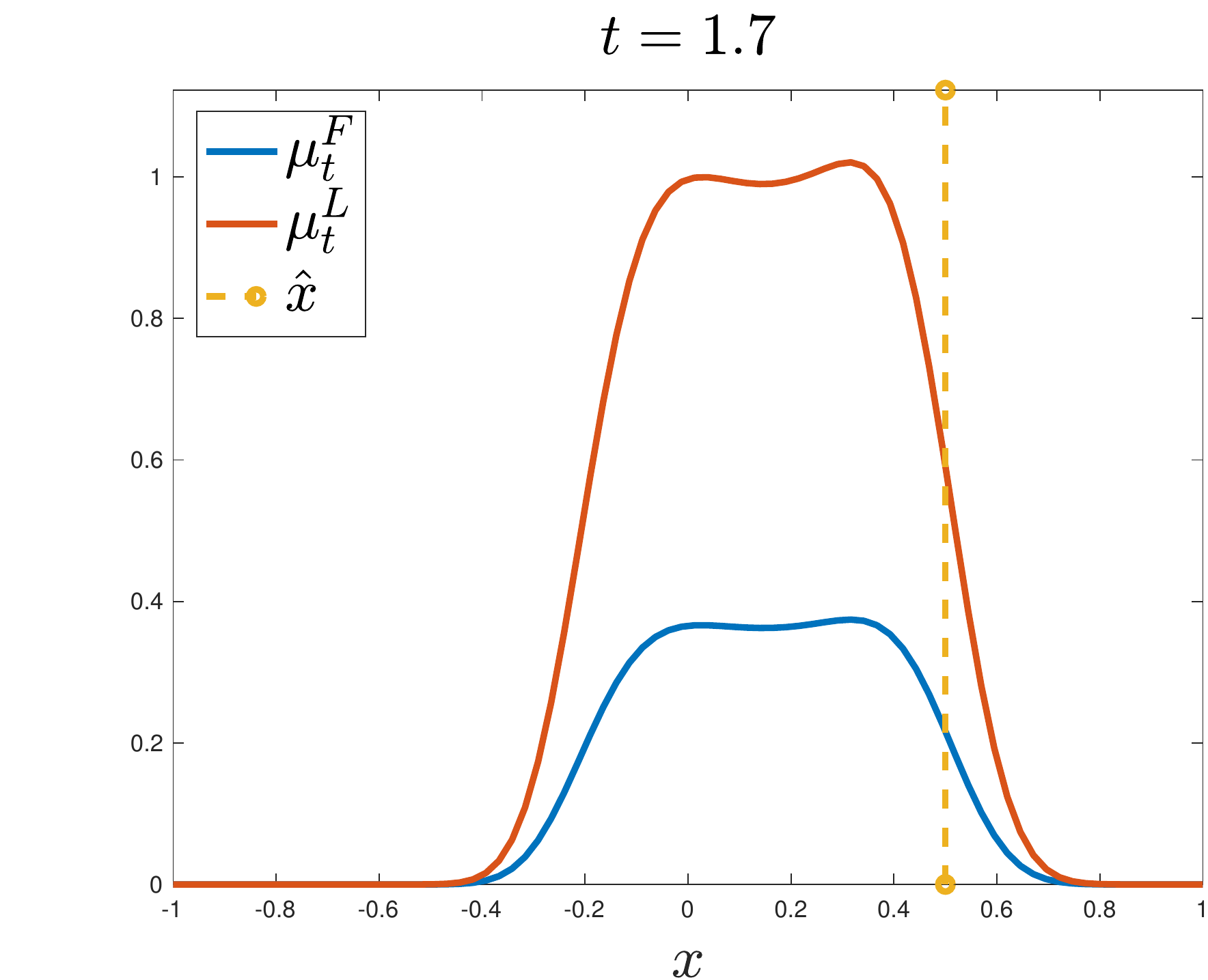}
%\\
\includegraphics[width=3.5cm]{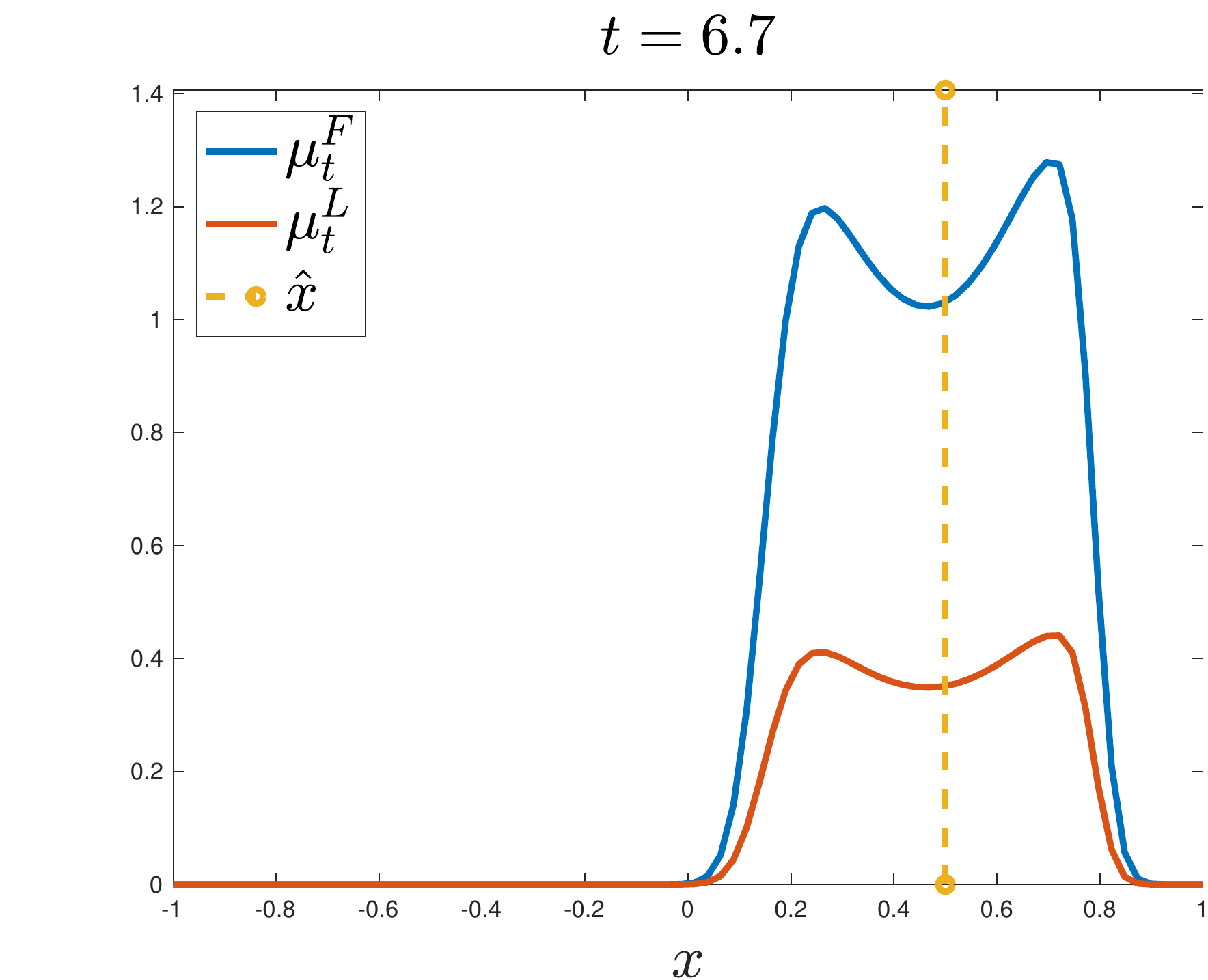}
\includegraphics[width=3.5cm]{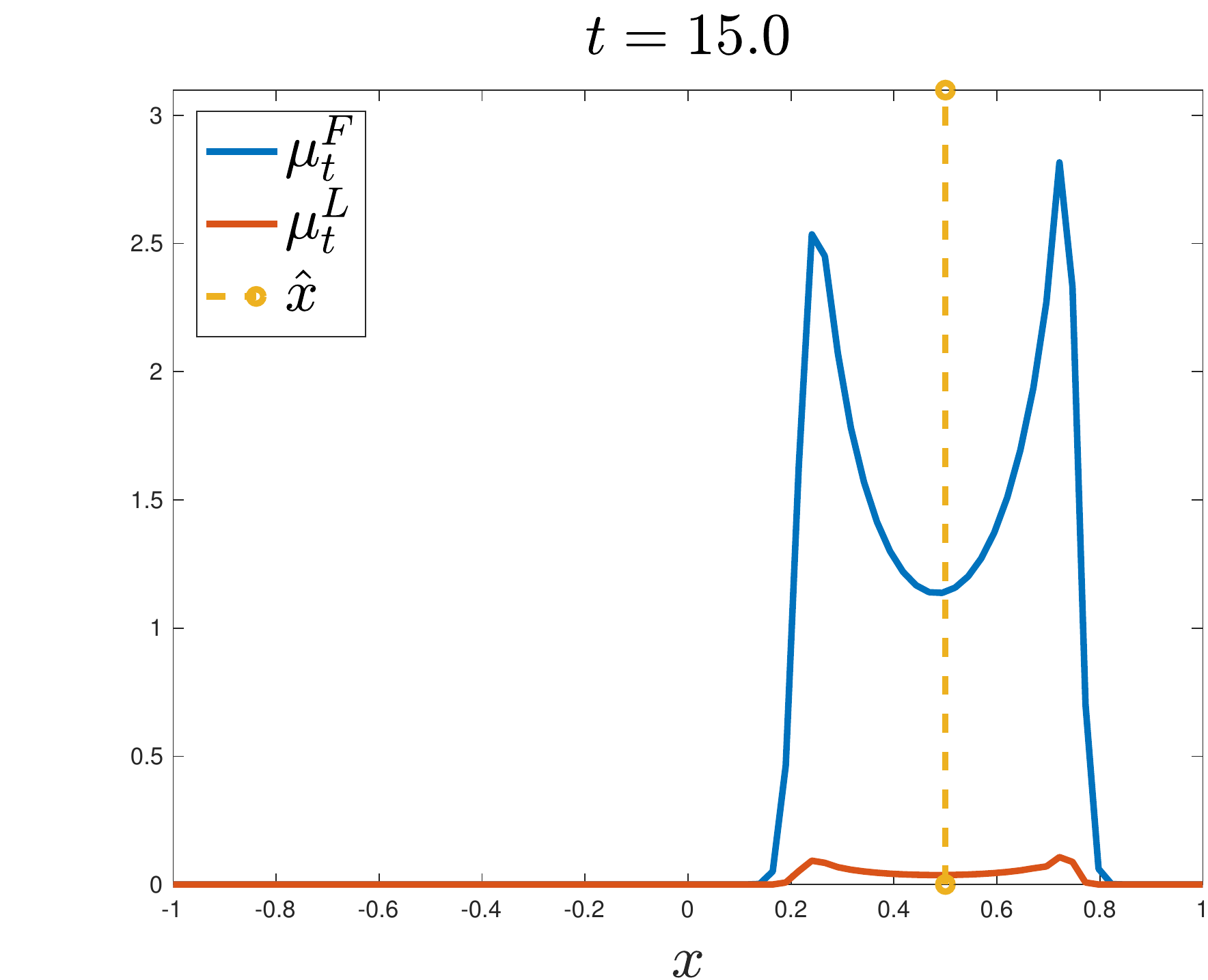}
\caption{{\em Test III}.  Top line: the left-hand picture shows the total density $\nu_t$, the right-hand picture shows the evolution of the masses $\sigma_t^F,\sigma_t^L$.
Bottom line: From left to right we depict the evolution of the leaders' and followers' densities towards the desired position $\hat x =0.5$.}\label{fig:Test3}
\end{figure}
%%
%\begin{figure}
%\centering
%\includegraphics[scale=0.25]{pics/TestIII/pic_aw2_n1.eps}
%\includegraphics[scale=0.25]{pics/TestIII/pic_aw2_n176.eps}
%\includegraphics[scale=0.25]{pics/TestIII/pic_aw2_n527.eps}
%\\
%\includegraphics[scale=0.25]{pics/TestIII/pic_aw2_n703.eps}
%\includegraphics[scale=0.25]{pics/TestIII/pic_aw2_n1229.eps}
%\includegraphics[scale=0.25]{pics/TestIII/pic_aw2_n1580.eps}
%
%\caption{{\em Test III}. From left to right, and top to bottom row we show the leaders' and followers' density towards the desired position $\hat x= 0.5$.}\label{fig}
%\end{figure}

\subsubsection*{Acknowledgments}
This work has been supported by the INdAM-GNCS 2018 project {\it Numerical methods for multiscale control problems}, and the project {\it ``MIUR Departments of Excellence 2018-2022".}

%\clearpage
\appendix

\section{Examples of transition functionals}\label{sec:assumptions}

%\setcounter{proposition}{0}
 %    \renewcommand{\theproposition}{\Alph{section}.\arabic{proposition}}

%Before focusing on the mean-field derivation of system \eqref{eq:macronusigma}, we shall provide practical examples of 
We prove a simple sufficient condition for $\alpha_F$ and $\alpha_L$ to satisfy \ref{ass:lipalpha}.

\begin{proposition}\label{prop:lipalpha}
For $i=1, \dots, 5$, let $f_i:\R^d\to \R^{m_1}$ be given locally Lipschitz continuous functions and consider a locally Lipschitz function $\alpha: \mathbb{R}^{m_1+m_2+m_3+m_4+m_5}\to \mathbb{R}$. Then the map
\[
\alpha(\mu, \eta)=\alpha\left(\int_{\R^d} (f_1 \conv \mu)\,\mathrm{d}\mu, \int_{\R^d} (f_2 \conv \eta)\,\mathrm{d}\mu, \int_{\R^d} (f_3 \conv \eta)\,\mathrm{d}\eta, \int_{\R^d} f_4\,\mathrm{d}\mu, \int_{\R^d}f_5\,\mathrm{d}\eta \right)
\]
satisfies Assumption \ref{ass:lipalpha}.
\end{proposition}
\begin{proof}
By possibly arguing componentwise on the $f_i$'s we can only consider the case $m_1=m_2=m_3=m_4=m_5=1$. For all $\mu$, $\eta$ satisfying $\mu(\R^d)$, $\eta(\R^d)\le M$ and with support contained in $B(0, R)$, we clearly have
\[
\begin{split}
&\left|\int_{\R^d} (f_1 \conv \mu)\,\mathrm{d}\mu\right|\le M^2 \mathrm{max}_{B(0, 2R)}|f_1|,\quad \left|\int_{\R^d} (f_2 \conv \mu)\,\mathrm{d}\eta\right|\le M^2 \mathrm{max}_{B(0, 2R)}|f_2|,\\
&\left|\int_{\R^d} (f_3 \conv \eta)\,\mathrm{d}\eta\right|\le M^2 \mathrm{max}_{B(0, 2R)}|f_3|,\quad \left|\int_{\R^d} f_4\,\mathrm{d}\mu\right|\le M \mathrm{max}_{B(0, R)}|f_4|\\
&  \left|\int_{\R^d} f_5\,\mathrm{d}\eta\right|\le M \mathrm{max}_{B(0, R)}|f_5|\,.
\end{split}
\]
With this hypothesis, since the function $\alpha$ is Lipschitz, it only suffices to show that the functions
\[
\begin{split}
& (\mu, \eta) \mapsto \int_{\R^d} (f_1 \conv \mu)\,\mathrm{d}\mu,\quad (\mu, \eta) \mapsto \int_{\R^d} (f_2 \conv \mu)\,\mathrm{d}\eta\\
& (\mu, \eta) \mapsto \int_{\R^d} (f_3 \conv \eta)\,\mathrm{d}\eta, \quad (\mu, \eta) \mapsto \int_{\R^d} f_4\,\mathrm{d}\mu, \quad \quad (\mu, \eta) \mapsto \int_{\R^d} f_5\,\mathrm{d}\eta
\end{split}
\] satisfy \ref{ass:lipalpha}. We only discuss the second case, since the proof in the other cases is similar.

Denote with $\tilde f_2$ the function defined by $\tilde f_2(x)=f_2(-x)$.  Whenever $\mu$ has support contained in $B(0, R)$  and satisfies $\mu(\R^d)\le M$ we clearly have
\begin{equation}\label{e-supremo}
\sup_{x\in B(0, R)} |f_2 \conv \mu|(x)\le M \sup_{x \in B(0, 2R)}|f_2(x)|, \quad \sup_{x\in B(0, R)} |\tilde f_2 \conv \mu|(x)\le M \sup_{x \in B(0, 2R)}|f_2(x)|\,.
\end{equation}
A direct computation also shows that, if we denote with $\mathrm{Lip}_R$ the Lipschitz constant on a ball of radius $R$, it holds
\begin{equation}\label{e-lipconst}
\mathrm{Lip}_R(f_2 \conv \mu) \le M \mathrm{Lip}_{2R}(f_2), \quad \mathrm{Lip}_R(\tilde f_2 \conv \mu) \le M \mathrm{Lip}_{2R}(f_2) \,.
\end{equation}
Take now $(\mu_1, \eta_1)$ and $(\mu_2, \eta_2)$ positive measures satisfying \eqref{e-massconstraint} and \eqref{e-boundedsupp}. Use \eqref{e-supremo} and \eqref{e-lipconst}, toghether with the Kantorovich-Rubinstein duality, we have
\[
\begin{split}
&\left|\int_{\R^d} (f_2 \conv \mu_1)\,\mathrm{d}\eta_1- \int_{\R^d} (f_2 \conv \mu_2)\,\mathrm{d}\eta_2\right|= \left|\int_{B(0, R)} (f_2 \conv \mu_1)\,\mathrm{d}\eta_1- \int_{B(0, R)} (f_2 \conv \mu_2)\,\mathrm{d}\eta_2\right|\\
&=\left|\int_{B(0, R)} (f_2 \conv \mu_1)\,\mathrm{d}(\eta_1-\eta_2)+ \int_{B(0, R)} (\tilde f_2 \conv \eta_2)\,\mathrm{d}(\mu_1-\mu_2)\right|\\
& \le C_{M, R} (\gw{\mu_1, \mu_2}+\gw{\eta_1, \eta_2}),
\end{split}
\]
with $C_{M, R}= M ( \sup_{x \in B(0, 2R)}|f_2(x)|+ \mathrm{Lip}_{2R}(f_2))$. This concludes the proof.
\end{proof}

\begin{example}
The statement above is clearly still valid if $\alpha$ only depends on some of the variables indicated above. In some applications (as for instance in \cite{markowich}) the transition rate $\alpha_i$ behaves countercyclically with respect to the mass of $\mu^i_t$: whenever $|\mu^i_t|$ is below a certain threshold $\varepsilon > 0$, the function $\alpha_i$ increases in order to restore  $|\mu^i_t|$ to higher levels. To model this phenomenon, let $\chi_{\varepsilon}$ be a mollification of the function
	$$\overline{\chi}_{\varepsilon}(x) = \begin{cases}
	1 & \text{ if } x \leq \varepsilon \\
	\overline{c} & \text{ otherwise,}
	\end{cases}$$
	with $0\le \overline{c}<1$. Then, by Proposition \ref{prop:lipalpha} (with $f_i=0$ for $i=1,\dots,4$ and $f_5\equiv 1$) the function
	$\alpha(\mu,\eta) \defin \chi_{\varepsilon}\left(|\eta|\right)$
	satisfies Assumption \ref{ass:lipalpha}.

Also quotients of functions of the type considered in Proposition \ref{prop:lipalpha} are easily seen to comply with Assumption \ref{ass:lipalpha}, provided that the denominator is bounded away from zero. For instance, for a given scalar-valued $f:\R^d \to \R$ and $g(\lambda)=((1-\lambda) \wedge \epsilon)^2$, where $\epsilon>0$ is a fixed threshold, one can consider a function of the type
\begin{equation}\label{eq:esempio}
\alpha(\mu,\eta):= \alpha\left(\frac{\int_{\R^d} (f\conv \mu)\,\mathrm{d}\mu}{g(|\eta|)}\right)\,.
\end{equation}
If $\nu \in \mathcal{P}_1(\R^d)$, $\sigma \in \mathcal{P}_1(\{F, L\})$ and we set $\mu:=\sigma(L)\nu, \eta:=\sigma(F)\nu$, then the above function reduces to
\[
\alpha(\nu, \sigma)=\alpha\left(\frac{\sigma(L)^2\int_{\R^d} (f\conv \nu)\,\mathrm{d}\nu}{(\sigma(L)\wedge \epsilon)^2}\right)
\]
which, as long as $\sigma(L)\ge \epsilon$, coincides with $\alpha(\int_{\R^d} (f\conv \nu)\,\mathrm{d}\nu)$ and only takes into account the total distribution $\nu$ of the two populations.
\end{example}

\section{Finite-volume scheme for mean-field leader-follower dynamics}\label{sec:FVscheme}
We introduce a finite-volume scheme for the discretization of the mean-field system \eqref{eq:macroleadfollstrong} in one-space dimension. Hence we consider a constant discretization step $\Delta x>0$, and we define $x_\ell=\ell\Delta x$ with $\ell\in\mathbb{Z}$, and the cells $C_\ell = [x_{\ell-1/2},x_{\ell+1/2}]$, with $x_\ell\pm1/2=x_\ell\pm\Delta x/2$, over which we define the averages
\begin{equation*}%\label{FVave}
\mu_\ell^i(t) = \frac{1}{\Delta x}\int_{x_{\ell-1/2}}^{x_{\ell+1/2}} \mu^i(x,t) dx,\qquad i\in\{F,L\}, 
\end{equation*}
where we used the notation $\mu^i(x,t)$ for the measure $\mu_t^i(x)$. In the same spirt we define the numerical fluxes as follows
\begin{equation*}%\label{Fluxes}
\begin{split}
\mathcal{F}^{i}_{\ell+1/2} = \mathcal{K}_{\ell+1/2}[\mu^F,\mu^L]\mu^i_{\ell+1/2},\qquad i\in\{F,L\}.
\end{split}
\end{equation*}
In what follows we will consider an upwinding scheme, where the convolutional operator $\mathcal{K}(\mu^F,\mu^L) := (K^{F}\conv\mu^F+K^{L}\conv\mu^L)(x_{\ell+1/2})$  is evaluated at the interfaces according to quadrature formula, and the densities $\mu^i_{\ell+1/2}$ are defined as follows
\begin{equation*}%\label{FluxesUpwind}
\begin{split}
&\mu^i_{\ell+1/2} = \begin{cases}
\mu_{\ell+1}^{i}\qquad &\textrm{if }~  \mathcal{K}_{\ell+1/2}<0,\\
\mu_{\ell}^{i}\qquad &\textrm{otherwise}.
\end{cases}
\end{split}
\end{equation*}
The sources terms are computed by averaging the transition rates $\alpha_F,\alpha_L$,  as follows
\begin{align*}%\label{eq:FVKappaFL}
\mathcal{A}^{i}_{\ell}(t) =\frac{1}{\Delta x}\int_{x_{\ell-1/2}}^{x_{\ell+1/2}}\alpha_i(\mu^F,\mu^L,t)\mu^i(x)dx,\qquad i\in\{F,L\}.
\end{align*}
We employ a first-order time marching scheme to compute the solution $\mu^i_{\ell}(t)$ over the time grid $0=t_0,\ldots,t_{N_t}=T$, with fixed time step $\Delta t= t_{n+1}-t_n$.  Moreover we used a splitting technique to evaluate separately the contribution by the non-linear transport and the source terms. Thus the full discrete scheme reads
\begin{subequations}
\begin{align*}%\label{eq:FVmacroLF}
&\begin{cases}
&\mu^{F,\star}_\ell =\mu^{F,n}_\ell-\frac{\Delta t}{\Delta x}\left(\mathcal{F}^F_{\ell+1/2} - \mathcal{F}^F_{\ell-1/2} \right),\cr
&\mu^{L,\star}_\ell =\mu^{L,n}_\ell-\frac{\Delta t}{\Delta x}\left(\mathcal{F}^L_{\ell+1/2} - \mathcal{F}^L_{\ell-1/2} \right),
\end{cases}
\\
&\begin{cases}
&\mu^{F,n+1}_\ell =\mu^{F,\star}_\ell-{\Delta t}\left(\mathcal{A}^{F,\star}_\ell -\mathcal{A}^{L,\star}_\ell\right),\cr
&\mu^{L,n+1}_\ell =\mu^{L,\star}_\ell+{\Delta t}\left(\mathcal{A}^{F,\star}_\ell -\mathcal{A}^{L,\star}_\ell\right).
\end{cases}
\end{align*}
\end{subequations}

\bibliographystyle{abbrv}

\end{document}